\documentclass[a4paper, 12pt]{amsart}
\usepackage{amssymb,amsmath,amsthm,bm}
\usepackage[all]{xy}
\allowdisplaybreaks[1]
\theoremstyle{definition}
\newtheorem{dfn}{Definition}[section]
\theoremstyle{plain}
\newtheorem{thm}{Theorem}[section]
\newtheorem{lem}{Lemma}[section]
\newtheorem{prop}{Proposition}[section]
\newtheorem{cor}{Corollary}[section]

\theoremstyle{remark}
\newtheorem{rem}{Remark}[section]

\makeatletter

\@addtoreset{equation}{section}
\makeatother
\newcommand{\Boxed}[1]{\mathbin{\ooalign{$\Box$\crcr
  \hidewidth
  \raise0.55ex\hbox{$\scriptscriptstyle{#1}$}%
  \hidewidth
}}}

\title{$\mathcal{W}$-algebras with  non-admissible levels and the Deligne exceptional series}
\author{Kazuya Kawasetsu}
\address{Department of Mathematical Sciences, University of Tokyo, Komaba, Tokyo, 153-8914, Japan.}
\email{kawasetu@ms.u-tokyo.ac.jp}

\date{}

\begin{document}

\begin{abstract}
Structure of certain simple $\mathcal{W}$-algebras assocated with the Deligne exceptional Lie algebras and non-admissible levels are described as the {\it simple current extensions} of certain vertex operator algebras.
As an application, the $C_2$-cofiniteness and $\mathbb{Z}_2$-rationality of the algebras are proved.
\end{abstract}

\maketitle

\section{Introduction}

The {\it Deligne exceptional series of simple Lie algebras} is the series
\[
A_1\subset A_2 \subset G_2 \subset D_4 \subset F_4 \subset E_6 \subset E_7 \subset E_8
\]
of simple Lie algebras \cite{D}.
For irreducible components of some tensor products of the adjoint representations of the simple Lie algebras in the above exceptional series,
remarkable dimension formulas, called {\it Deligne dimension formulas}, were established \cite{CdM,D,LaM2}.
They are expressed as rational functions in the dual Coxeter number $h^\vee$.
For example,
\begin{equation*}\label{eqn:vogel}
\dim \mathfrak{g} = \frac{2(h^\vee +1)(5h^\vee -6)}{h^\vee +6},
\end{equation*}
and
\[
\dim L(2\theta) = \frac{5h^{\vee 2}(2h^\vee+3)(5h^\vee-6)}{(h^\vee+12)(h^\vee+6)}.
\]

The same exceptional series appeared in earlier studies of modular differential equations.
In 1988, Mathur, Mukhi and Sen, in their work of classification of rational conformal field theories ($C_2$-cofinite rational $\mathbb{Z}_+$-graded vertex operator algebras (VOAs) of CFT-type) with two characters \cite{MMS},
studied the modular differential equations of the form
\begin{equation}\label{eq:diff1}
\left( q\frac{d}{dq}\right)^2 f(\tau) + 2 E_2(\tau)\left(q\frac{d}{dq}\right)f(\tau) + 180\mu\cdot E_4(\tau)f(\tau)=0.
\end{equation}
Here $\mu$ is a numerical constant, $\tau$ a complex number in the complex upper half-plane $\mathbb{H}$ with $q=e^{2\pi i \tau}$, and $E_k(\tau) (k=2,4,6,\ldots)$ the Eisenstein series.
(Differential equations equivalent to (\ref{eq:diff1}) were studied by Kaneko and Zagier \cite{Kan3} in number theory. See also \cite{KNS}.)
By studying (\ref{eq:diff1}), they showed, roughly speaking, that the characters of the rational conformal field theories with two characters are that of the
level one affine VOAs $V_1(\mathfrak{g})$ associated to the Deligne exceptional simple Lie algebras $\mathfrak{g}$.
Note that the differential equations (\ref{eq:diff1}) and $V_1(\mathfrak{g})$  associated to the Deligne exceptional simple Lie algebras $\mathfrak{g}$ also appear
in the study of large symmetry of vertex operator algebras \cite{T,M}.

In this paper, we consider the Deligne exceptional series in the study of the {\it quantized Drinfel'd-Sokolov reduction},
thus the so-called {\it $\mathcal{W}$-algebras}, and the {\it simple current extensions}.
Then, we obtain new examples of $C_2$-cofinite rational $\mathcal{W}$-algebras.

The $\mathcal{W}$-algebras are generalization of the extensions of the Virasoro vertex algebras,
first introduced in \cite{Zam}.
After the considerably many studies, the construction of $\mathcal{W}$-algebras by using the {\it quantized Drinfel'd-Sokolov reduction} was introduced \cite{FF, KRW,KW1} (see also \cite{A1}).

Let $\mathfrak{g}$ be a finite dimensional simple Lie algebra, $f$ an even nilpotent element and $k$ a complex number.
Consider the level $k$ universal affine vertex operator algebra (VOA) $V^k(\mathfrak{g})$ with the Segal-Sugawara conformal vector $\omega^\mathrm{aff}$ and certain vertex operator algebras $\mathcal{F}^\mathrm{ne}$
 and $\mathcal{F}^\mathrm{ch}$ (Fermions) depending on $\mathfrak{g}$ with the conformal vectors $\omega^{\mathrm{ne}}$ and $\omega^{\mathrm{ch}}$.
Consider the tensor product vertex algebra $C=V^k(\mathfrak{g})\otimes \mathcal{F}^\mathrm{ne}\otimes \mathcal{F}^{\mathrm{ch}}$, and equip $C$ with certain grading and differential $d$ (depending on $f$), thus the complex structure $(C^\bullet,d)$ ({\it BRST complex}).
Equip $C$ with the vertex operator algebra structure with the conformal vector
\[
\omega=\omega^\mathrm{aff}+\omega^\mathrm{ne}+\omega^\mathrm{ch}+\partial x.
\]
Here, $x$ is a semisimple element of $\mathfrak{g}$ with $[x,f]=-1$ and certain conditions, and $\partial$ is the derivation of the vertex algebra $C$.
Then, the {\it universal $\mathcal{W}$-algebra} $\mathcal{W}^k (\mathfrak{g},f)$ associated with $(\mathfrak{g},f,k)$ is
defined to be the $0$-th cohomology of the BRST complex,
which is a $(1/2)\mathbb{Z}_+$-graded vertex operator algebra with the conformal vector $\omega$ (not a superVOA).
Denote the simple quotient vertex operator algebra by $\mathcal{W}_k(\mathfrak{g},f)$ and call it the {\it simple $\mathcal{W}$-algebra}.

Modular invariance of the characters of the modules of vertex operator algebras are important property.
The space spanned by the characters of the modules of a RCFT ($C_2$-cofinite rational $\mathbb{Z}_+$-graded VOA of CFT-type) is invariant under modular transformation \cite{Zhu}.
Modular invariance of the characters of twisted modules is considered in \cite{DLM}.
By generalizing the result of \cite{DLM}, modular invariance for $\mathbb{Q}_+$-graded (super)VOAs of CFT-type
is considered in \cite{E}.

The modular invariance of the characters of the modules of $\mathcal{W}$-algebras with {\it admissible} levels $k$ and certain nilpotent elements $f$ (exceptional pairs) have been studied
by using the modular invariance of the corresponding affine VOAs $V_k(\mathfrak{g})$ \cite{KW2}.
Here, the number $k$ is called {\it admissible} if
\[
k+h^\vee=\frac{p}{q},\quad p,q\in \mathbb{Z}_{>0},\quad (p,q)=1,\quad p\geq 
\begin{cases}
h^\vee & (r^\vee,q)=1,\\
h & (r^\vee,q)=r^\vee.
\end{cases}
\]
Here, $r^\vee$ denotes the lacing number, that is, $r^\vee=1$ for $\mathfrak{g}=A_l,D_l,E_l$, $r^\vee=2$ for $\mathfrak{g}=B_l,C_l,F_4$ and $r^\vee=3$ for $\mathfrak{g}=G_2$.
Later, considerably many simple $\mathcal{W}$-algebras with admissible levels and certain nilpotent elements (including exceptional pairs) were proved to be $C_2$-cofinite \cite{A3}.
Since it was conjectured and has been widely believed that a simple affine VOA of level $k$ has the modular invariance property if and only if $k$ is an admissible number \cite{KW0}, $\mathcal{W}$-algebras has been believed to be $C_2$-cofinite and rational only if the level $k$ is an admissible number (cf.\ \cite{KW2}).

In this paper, we prove the $C_2$-cofiniteness and rationality of certain simple $\mathcal{W}$-algebras with non-admissible levels by using the theory of {\it simple current extensions} of the vertex operator algebras.

Let $f$ be a minimal nilpotent element of $\mathfrak{g}$.
Then, we have the Cartan subalgebra $\mathfrak{h}$ and  the highest root $\theta$ of $\mathfrak{g}$ such that $f$ becomes a lowest root vector $f_\theta$.
Consider the $\mathrm{ad}(\theta/2)$-eigenspace decomposition ({\it minimal gradation})
\[
\mathfrak{g}=\mathfrak{g}_{-1}\oplus \mathfrak{g}_{-1/2}\oplus\mathfrak{g}_{0}\oplus\mathfrak{g}_{1/2}\oplus\mathfrak{g}_1.
\]
Let $\mathfrak{g}^\natural$ denote the centralizer of $A_1=(e,\theta,f)$.
Here, $e$ is a highest root vector.
Let $k$ be a complex number.
Then, the $\mathcal{W}$-algebra $W^k=\mathcal{W}^k(\mathfrak{g},f)$
is strongly generated by the conformal vector $\omega$ with $x=\theta/2$ and certain linearly-defined
primary vectors $J^{\{a\}}$ ($a\in \mathfrak{g}^\natural$) of conformal weight $1$
 and $G^{\{v\}}$ ($v\in \mathfrak{g}_{-1/2}$) of conformal weight $3/2$
subject to the OPEs ($\lambda$-brackets) ($a,b\in \mathfrak{g}^\natural$, $u,v\in \mathfrak{g}_{-1/2}$)
\[
[J^{\{a\}}{}_\lambda J^{\{b\}}]=J^{\{[a,b]\}}+\lambda (a,b)^\natural|0\rangle,
\]
\[
[J^{\{a\}}{}_\lambda G^{\{v\}}]=G^{\{[a,v]\}},
\]
and certain polynomial $[G^{\{u\}}{}_\lambda G^{\{v\}}]$ in $\lambda$ such that
the coefficients belong to the subVOA generated by $\omega$ and $J^{\{p\}}$, $p\in \mathfrak{g}^\natural$.
Here, the cocycle $(\cdot,\cdot)^\natural$ is certain invariant bilinear form on $\mathfrak{g}^\natural$.
(For more detail, see Proposition \ref{sec:propkw} \cite[Theorem 5.1]{KW1}).

Suppose that $\mathfrak{g}$ is not of type $A_l$.
Let $V^k$ denote the subVOA generated by $J^{\{a\}}$ ($a\in \mathfrak{g}^\natural$).
Then, $V^k$ is isomorphic to the universal affine vertex operator algebra associated with $\mathfrak{g}^\natural$ and the cocycle $(\cdot,\cdot)^\natural$.
Note that $V^k$ is not of level $k$.

Let $W=\mathcal{W}_k(\mathfrak{g},f)$ denote the simple quotient of $W^k$,
and $V\subset W$ the image of $V^k$.
The main concern of this paper is the {\it branching rule} of $V\subset W$.

So, let $\omega^\natural$ denote the Segal-Sugawara conformal vector of $V^k$.
Since $J^{\{a\}}\in W^k$ ($a\in \mathfrak{g}^\natural$) are primary vectors of conformal weight $1$, the vector $\omega^{\mathrm{Vir}}=\omega-\omega^\natural$ is a Virasoro vector.
Let $U^k$ denote the Virasoro vertex operator subalgebra of $W^k$ generated by $\omega^\mathrm{Vir}$, and $U\subset W$ the image of $U^k$ under the simple quotient $W^k\rightarrow W$.
Then, the tensor product VOA $V^k\otimes U^k$ and $V\otimes U$ are embedded in $W^k$ and $W$.

What do we have when we decompose $W$ as a $V\otimes U$-module?
We consider the most beautiful case, that is, the case when $W$ must be the simple current extension of $V\otimes U$.

Suppose that $\mathfrak{g}$ is not of type $A_l$ and 
\begin{enumerate}
\item $V$ and $U$ are simple, rational and $C_2$-cofinite vertex operator algebras;
\item $W\cong V\otimes U\oplus N\otimes M$ as $V\otimes U$-modules with non-identity simple currents $N$ of $V$ and $M$ of $U$.
\end{enumerate}

\begin{thm}\label{sec:mainthm}
The complete list of the pair $(\mathfrak{g},k)$ satisfying (1) and (2) above is given by the following pairs:

\begin{enumerate}
\item $\mathfrak{g}=C_2$ and $k=1/2$,
\item $\mathfrak{g}=G_2,D_4,F_4,E_6,E_7,E_8$ and $k=-h^\vee/6$.
\end{enumerate}
For each pair $(\mathfrak{g},k)$ in the above list, $\mathcal{W}_{k}(\mathfrak{g},f_\theta)$ is $C_2$-cofinite and $\mathbb{Z}_2$-rational with an automorphism group $\mathbb{Z}_2=\{\mathrm{id},\iota\}$ defined to be
\[
\mathrm{id}(a)=a,\quad\mathrm{id}(u)=u,\quad \iota(a)=a,\quad \iota(u)=-u, 
\]
($a\in V\otimes U$, $u\in N\otimes M$).
We have the following isomorphisms:
\[
\mathcal{W}_{1/2}(C_2,f_\theta)\cong V_1(A_1)\otimes L(-25/7,0)\oplus V_1(A_1;\alpha/2)\otimes L(-25/7,5/4),
\]
\[
\mathcal{W}_{-2/3}(G_2,f_\theta)\cong V_3(A_1)\otimes L(-3/5,0)\oplus V_3(A_1;\alpha/2)\otimes L(-3/5,3/4),
\]
\[
\mathcal{W}_{-1}(D_4,f_\theta)\cong V_1(A_1)^{\otimes 3}\otimes L(-3/5,0)\oplus V_1(A_1;\alpha/2)^{\otimes 3}\otimes L(-3/5,3/4),
\]
\[
\mathcal{W}_{-3/2}(F_4,f_\theta)\cong V_1(C_3)\otimes L(-3/5,0)\oplus V_1(C_3;\varpi_3)\otimes L(-3/5,3/4),
\]
\[
\mathcal{W}_{-2}(E_6,f_\theta)\cong V_1(A_5)\otimes L(-3/5,0)\oplus V_1(A_5;\varpi_3)\otimes L(-3/5,3/4),
\]
\[
\mathcal{W}_{-3}(E_7,f_\theta)\cong V_1(D_6)\otimes L(-3/5,0)\oplus V_1(D_6;\varpi_6)\otimes L(-3/5,3/4),
\]
and
\[
\mathcal{W}_{-5}(E_8,f_\theta)\cong V_1(E_7)\otimes L(-3/5,0)\oplus V_1(E_7;\varpi_7)\otimes L(-3/5,3/4).
\]
\end{thm}

Here, $\mathbb{Z}_2$-rationality says that for each $g\in \mathbb{Z}_2$, the $g$-twisted modules are completely reducible and there are finitely many inequivalent irreducible $g$-twisted modules,
$L(c,h)$ denotes the irreducible highest weight module of the Virasoro algebra of central charge $c$ and lowest conformal weight $h$.
Note that $L(-3/5,0)$ and $L(-25/7,0)$ are $(p,q)=(3,5)$ and $(3,7)$ Virasoro minimal model vertex operator algebras $\mathcal{M}(p,q)$.

Note that when $\mathfrak{g}=D_4,E_6,E_7,E_8$, the numbers $k=-h^\vee/6=-1,-2,-3,-5$ are not admissible numbers.
Therefore,  $W_{-h^\vee/6}(\mathfrak{g},f_\theta)$ with $\mathfrak{g}=D_4,E_6,E_7,E_8$ are new examples of $C_2$-cofinite $\mathcal{W}$-algebras.
These numbers satisfy the necessary condition \cite{GK1,GK2,KW2}
for the modular invariance property for the irreducible modules of 
affine Kac-Moody Lie algebras $\hat{\mathfrak{g}}^{(1)}$ of level $k$.

This result reminds us with the Deligne exceptional series.
In fact, consider $\mathfrak{g}=A_1,A_2$ and $k=-h^\vee/6=-1/3,-1/2$.
Then, we have
\begin{eqnarray}
\mathcal{W}_{-1/3}(A_1,f_\theta)\cong L(-3/5,0),
\end{eqnarray}
and
\begin{eqnarray}\label{eqn:a2case}
\\
\mathcal{W}_{-1/2}(A_2,f_\theta)\cong V_{\sqrt{3}A_1}\otimes L(-3/5,0)\oplus V_{\sqrt{3}A_1+\sqrt{3}\alpha/2}\otimes L(-3/5,3/4).\nonumber
\end{eqnarray}
Thus, $\mathcal{W}_{-1/3}(A_1,f_\theta)$ is $C_2$-cofinite and rational,
and $\mathcal{W}_{-1/2}(A_2,f_\theta)$ is $C_2$-cofinite and $\mathbb{Z}_2$-rational.

To prove the isomorphism in Theorem \ref{sec:mainthm} and eq.\ (\ref{eqn:a2case}), we explicitly give isomorphisms of vertex algebras from $W$ to the simple current extensions $W'$ of certain tensor product vertex operator algebras $V\otimes U$.

Let $\mathfrak{g}$ be a Deligne exceptional Lie algebra not of type $A_1$.
Let $\theta$ denote the highest root of $\mathfrak{g}$ and $A_1\subset \mathfrak{g}$ the $sl_2$-triple for $\theta$.
Consider the level one simple affine VOA $V_1(\mathfrak{g})$.
Let $V$ denote the commutant of $V_1(A_1)$ in $V_1(\mathfrak{g})$.
Let $N$ denote the simple current of $V$ defined to be
$V_1(\mathfrak{g})\cong V\otimes V_1(A_1)\oplus N\otimes V_1(A_1;\theta/2)$ as $V\otimes V_1(A_1)$-modules.
Consider the simple current extension $W'=V\otimes L(-3/5,0)\oplus N\otimes L(-3/5,3/4)$.
Then, as $\frac{1}{2}\mathbb{Z}_+$-graded vertex operator algebras, the simple $\mathcal{W}$-algebra $\mathcal{W}_{-h^\vee/6}(\mathfrak{g},f_\theta)$ is isomorphic to
$W'$.

We show the isomorphism by explicitly comparing the operator product expansions (OPEs) of certain generators of $W$ and $W'$.

In order to compute the VOA structure of $W'$,
we first consider the {\it abelian intertwining algebras} (AIAs) $V\oplus N$ and $L(-3/5,0)\oplus L(-3/5,3/4)$.
Then, we consider the tensor product of the AIAs and realize $W'$ as a subVOA (graded tensor product) of the tensor product.

In order to simplify the construction, we introduce the notions of {\it quasi generalized vertex algebra} (quasi-GVA) which is a subclass of the AIAs and include the {\it generalized vertex algebras} (GVAs).
We show that the above two AIAs are indeed quasi-GVAs. 
For the detail of the definitions, see section \ref{sec:appenda}.

The case when $(\mathfrak{g},k)=(C_2,1/2)$ is proved similarly and will be considered in the forthcoming paper.

By the general theory of the simple current extensions, we see that  $\mathcal{W}_{-h^\vee/6}(\mathfrak{g},f_\theta)$ are $C_2$-cofinite and $\mathbb{Z}_2$-rational (cf.\ \cite{Car,Lam,Yam}).
By \cite{DLM,E}, we obtain the modular invariance of the characters of (twisted) modules of the $\mathcal{W}$-algebras.
For $E_8$ case, we see that the characters of the $\iota$-twisted modules (Ramond twisted modules) coincide with modular invariant characters of the {\it intermediate vertex subalgebra} $V_{E_{7+1/2}}$ \cite{Kaw1}.
The Ramond-twisted irreducible characters of $\mathcal{W}_{-h^\vee/6}(\mathfrak{g},f_\theta)$ form a basis of the solutions of modular differential equation (\ref{eq:diff1}) with $\mu=-551/900$, which is the ``hole'' of the
$2$-character RCFTs \cite{MMS}.

The $\mathcal{W}$-algebras associated with other series of $\mathfrak{g}$ and higher levels $k$ will be considered in the forthcoming papers.
The affine vertex operator algebras $V_{-h^\vee/6}(\mathfrak{g})$ ($\mathfrak{g}$ is Deligne exceptional) will be considered in the forthcoming paper.

In section 2, we consider the structure (branching rule) of the $\mathcal{W}$-algebras and explicitly 
give the isomorphisms in Theorem \ref{sec:mainthm} and eq.\ (\ref{eqn:a2case}).
In section 3, we explicitly show the key lemma for $\mathfrak{g}=D_4$ and $\mathfrak{g}=E_8$, which are the smallest and largest examples with non-admissible levels. The remaining cases are shown similarly.
In section 4, we mention some remarks about modular invariance of the characters of $W$ and other series.
Section 5--8 are appendixes.
In section 5 (Appendix A), we review and introduce the notion of AIAs and quasi GVAs and modification of the quasi-GVAs.
Following \cite{BK}, we use the {\it locality} to define the AIAs.
In section 6 (Appendix B), we consider the extension of $L(-3/5,0)$ and
modification of $V_{A_1^\circ}$.
In section 7 (Appendix C), we recall some well-known vertex algebras and generalized vertex algebras.
In section 8 (Appendix D), we consider some general facts for the abelian intertwining algebras for the reader's convenience.

\subsection*{Acknowledgments}
The author wishes to express his thanks to his advisor, Professor A.\,Matsuo for helpful advice and kind encouragement.
He wishes to express his thanks to Professor T.\,Arakawa for suggesting him to study  $\mathcal{W}$-algebras
 to develop the author's previous work, helpful advice and kind encouragement.
He also wishes to express his thanks to H.\,Yamauchi, S.\,Carnahan, M.\,Okumura and N.\,Genra for helpful discussions.
Part of this work was done when the author visited Centro di Ricerca Matematica Ennio De Giorgi, Pisa, Italy.
He also wishes to express his thanks to the institute for the great hospitality.
This work is supported by JSPS KAKENHI Grant Number 14J09236.

\subsection*{Notations.}
We denote the non-negative integers by $\mathbb{Z}_+$.
We denote the positive integers by $\mathbb{Z}_{>0}$ and
 the negative integers by $\mathbb{Z}_{<0}$.
We denote the non-negative rational numbers by $\mathbb{Q}_+$.
All vector spaces are over the field of the complex numbers $\mathbb{C}$.
We denote $\mathbb{Z}_n=\mathbb{Z}/n\mathbb{Z}$ for $n\in \mathbb{Z}_{>0}$.
For $x=k+n\mathbb{Z}\in \mathbb{Z}_n$, we denote $k=x$.
For a finite set $A$, we denote the cardinality of $A$ by $\sharp A$.

\section{$\mathcal{W}$-algebras associated with a minimal nilpotent element}

\subsection{Preliminaries}\label{sec:wkigou}

In this section, we recall the notations of $\mathcal{W}$-algebras associated with a minimal nilpotent element \cite{KW1}.

Let $\mathfrak{g}$ be a finite dimensional simple Lie algebra.
Let $\mathfrak{h}$ be a Cartan subalgebra and $\Delta\subset \mathfrak{g}^*$ the set of root of $\mathfrak{h}$ in $\mathfrak{g}$ with a set of positive roots $\Delta_+\subset \Delta$.
Let $\theta\in \Delta$ denote the highest root of $\mathfrak{g}$.
Let $(\cdot|\cdot)$ denote the non-degenerate invariant bilinear form normalized as $(\theta|\theta)=2$.
Identify $\mathfrak{h}$ with $\mathfrak{h}^*$ by using this form, and set $x=\theta/2\in \mathfrak{h}$.
Let $f=f_\theta$ be a non-zero lowest root vector with a $sl_2$-triple $(e,x,f)$ with $[x,e]=e$, $[x,f]=-f$ and $[e,f]=x$,
so that $e$ is a highest root vector.
As the $\mathrm{ad}\,x$-eigenspace decomposition, we have the {\it minimal} gradation
\[
\mathfrak{g}=\mathfrak{g}_{-1}\oplus \mathfrak{g}_{-1/2}\oplus \mathfrak{g}_{0}\oplus\mathfrak{g}_{1/2}\oplus\mathfrak{g}_1
\]
with $\mathfrak{g}_1=\mathbb{C}e$ and $\mathfrak{g}_{-1}=\mathbb{C}f$.
Here, $\mathfrak{g}_n:=\{v\in \mathfrak{g}|[x,v]=nv\}$.
Denote by $\mathfrak{g}^f$ the centralizer of $f$ in $\mathfrak{g}$ and by $\mathfrak{g}^\natural$ the subspace $\mathfrak{g}^f\cap \mathfrak{g}_0$.
Then, $\mathfrak{g}^\natural$ coincides with the centralizer of the $sl_2$-triple $(e,x,f)$, and $\mathfrak{g}^f$ the subspace $\mathfrak{g}_{-1}\oplus \mathfrak{g}_{-1/2}\oplus \mathfrak{g}^\natural$.
Set $\mathfrak{h}^\natural:=\{h\in\mathfrak{h}|(x|h)=0\}$.
Then, $\mathfrak{h}^\natural$ is a 
Cartan subalgebra of $\mathfrak{g}^\natural$, and
 we have $\mathfrak{h}=\mathfrak{h}^\natural\oplus \mathbb{C}x$.
Define the skew-symmetric bilinear form $\langle\cdot,\cdot\rangle_{\mathrm{ne}}$ on 
$\mathfrak{g}_{1/2}$  to be $\langle a,b\rangle_{\mathrm{ne}}=(f|[a,b])$.
Since $(e|f)=1/2$ and $[a,b]\in\mathbb{C}e$ for $a,b\in\mathfrak{g}_{1/2}$, we have
\[
[a,b]=2\langle a,b\rangle_{\mathrm{ne}}\,e,\quad a,b\in\mathfrak{g}_{1/2}.
\]
Note that
\[
\mathrm{tr}_{\mathfrak{g}}(\mathrm{ad}\,a)(\mathrm{ad}\,b)=2h^\vee (a|b), \quad a,b\in\mathfrak{g}.
\]

Let $\{u_\alpha\}_{\alpha\in S^\natural}$ be a basis of $\mathfrak{g}^\natural$ with the index set $S^\natural$,
and $\{u^\alpha\}_{\alpha\in S^\natural}$ the dual basis such that $(u_\alpha|u^\beta)=\delta_{\alpha,\beta}$ for $\alpha,\beta\in S^\natural$.
Set $S_{\pm 1/2}=\{\beta\in \Phi(\mathfrak{g})|(\beta|\theta/2)=\pm 1/2\}$.
Let $\{u_\gamma\}_{\gamma\in S_{1/2}}$ be a basis of $\mathfrak{g}_{1/2}$,
and $\{u^\gamma\}_{\gamma\in S_{1/2}}$ the dual basis such that $\langle u_\gamma|u^\eta\rangle_{\mathrm{ne}}=\delta_{\gamma,\eta}$ for $\gamma,\eta\in S_{1/2}$.
For $v\in \mathfrak{g}_0$, we denote by $v^\natural$ the orthogonal projection of $\mathfrak{g}_0$ on $\mathfrak{g}^\natural$.
Let $\kappa_{\mathfrak{g}_0}$ denote the Killing form of $\mathfrak{g}_0$.
Denote by $h^\vee_{0,i}$ the dual Coxeter number of the $i$-th simple component $\mathfrak{g}_i^\natural$ of $\mathfrak{g}^\natural$ with respect to the bilinear form $(\cdot|\cdot)$ restricted to $\mathfrak{g}_i^\natural$.

Let $k$ be a complex number.
Recall that the universal $\mathcal{W}$-algebra $\mathcal{W}^k(\mathfrak{g},f_\theta)$
 is a $\frac{1}{2}\mathbb{Z}_+$-graded vertex operator algebra with the conformal vector $\omega$
of central charge
\[
c_\mathcal{W}=\frac{k\,\mathrm{dim}\,\mathfrak{g}}{k+h^\vee}-6k+h^\vee-4.
\]

Recall the $\lambda$-bracket $[a_\lambda b]=\sum_{n=0}^\infty \lambda^n a{(n)}b/n!$, $a,b\in W^k(\mathfrak{g},f)$.
Note that the $\lambda$-brackets are substitutes of the OPEs.

Let $V$ be a vertex algebra, and $B$ be a subspace of $V$.
The subspace $B$ {\it strongly generates} $V$ if the monomials
\[
v_1(m_1)v_2(m_2)\cdots v_s(m_s)|0\rangle \in V,
\]
($s\in \mathbb{Z}_+$ with $v_i\in B$, $m_i\in \mathbb{Z}_{<0}$, $i=1,\ldots,s$)
span $V$.
Let $S$ be a basis of $B$ with a total order on $S$.
The subspace $B$ {\it obeys the PBW theorem} if the monomials
\[
v_1(m_1)v_2(m_2)\cdots v_s(m_s)|0\rangle \in V,
\]
($s\in \mathbb{Z}_+$ with $v_i\in S$, $m_i\in \mathbb{Z}_{<0}$, $i=1,\ldots,s$, where
the 
sequence of pairs $(v_1,m_1),(v_2,m_2),\ldots, (v_s,m_s)$ is non-increasing in the
lexicographical order)
 form a basis of $V$.
We call the basis a {\it PBW-basis} of $V$ and denote it by $\tilde{S}$.

\begin{prop}\cite[Theorem 5.1, Theorem 4.1]{KW1}\cite[pp.454]{KWc}\label{sec:propkw}
The universal $\mathcal{W}$-algebra $\mathcal{W}^k(\mathfrak{g},f)$ is strongly generated by 
the conformal vector  $\omega$ and certain linearly-defined
primary vectors $J^{\{a\}}$, $a\in \mathfrak{g}^\natural$ of conformal weight $1$
and $G^{\{v\}}$, $v\in \mathfrak{g}_{-1/2}$ of conformal weight $3/2$, with the OPEs ($\lambda$-brackets)
\[
[{J^{\{a\}}}_\lambda J^{\{b\}}]=J^{\{[a,b]\}}+\lambda\left(\left(k+\frac{1}{2}h^\vee\right)(a|b)-\frac{1}{4}\kappa_{\mathfrak{g}_0}(a,b)\right)|0\rangle,
\]
\[
[{J^{\{a\}}}_\lambda G^{\{v\}}]=G^{\{[a,v]\}},
\]
\begin{eqnarray*}
&&[{G^{\{u\}}}_\lambda G^{\{v\}}]=-2(k+h^\vee)(e|[u,v])\omega+(e|[u,v])\sum_{\alpha\in S^\natural} J^{\{u^\alpha\}}(-1)J^{\{u_\alpha\}}\nonumber\\
&&\quad+\sum_{\gamma\in S_{1/2}}J^{\{[u,u^\gamma]^\natural\}}(-1)J^{\{[u_\gamma,v]^\natural\}}+2(k+1)(\partial +2\lambda)J^{\{[[e,u],v]^\natural\}}\nonumber\\
&&\quad+\lambda\sum_{\gamma\in S_{1/2}} J^{\{[[u,u^\gamma],[u_\gamma,v]]^\natural\}}
+\frac{\lambda^2}{3}\biggl((e|[u,v])\biggl(-(k+h^\vee)c_{\mathcal{W}}\nonumber\\
&&\quad +\left(k+\frac{1}{2}h^\vee\right)\mathrm{dim}\,\mathfrak{g}^\natural-\frac{1}{2}
\sum_i h^\vee_{0,i}\,\mathrm{dim}\,\mathfrak{g}_i^\natural\biggr) +\sum_{\gamma\in S_{1/2}}\nonumber\\
&&\quad\biggl( \left(k+\frac{1}{2}h^\vee\right)\left([u,u^\gamma]^\natural\bigl|[u_\gamma,v]^\natural\right)
-\frac{1}{4}\kappa_{\mathfrak{g}_0}\left([u,u^\gamma]^\natural,[u_\gamma,v]^\natural\right)\nonumber\\
&&\quad+\frac{1}{4}\mathrm{tr}_{\mathfrak{g}_{1/2}\oplus\mathfrak{g}_1}\,\mathrm{ad}\,\left([[u,u^\gamma]^\natural,[u_\gamma,v]^\natural]\right)\biggr)\biggr)|0\rangle.
\end{eqnarray*}
Moreover, the space of the generators $\{\omega,J^{\{a\}},G^{\{v\}}|a\in \mathfrak{g}^\natural,v\in \mathfrak{g}_{-1/2}\}$ obeys the PBW theorem.
\end{prop}

Let $V^k$ denote the vertex operator subalgebra generated by $J^{\{a\}}$, $a\in \mathfrak{g}^\natural$.

Suppose that $\mathfrak{g}$ is a Lie algebra not of type $A_l$.
Then, $\mathfrak{g}^\natural$ is semi-simple, and
 $V^k$ is isomorphic to the universal affine vertex operator algebra associated with $\mathfrak{g}^\natural$ and
the invariant bilinear form $(\cdot,\cdot)^\natural:\mathfrak{g}^\natural \times \mathfrak{g}^\natural \rightarrow \mathbb{C}$ defined to be
$(a,b)^\natural=(k+h^\vee/2)(a|b)-(1/4)\kappa_{\mathfrak{g}_0}(a,b)$, $a,b\in\mathfrak{g}^\natural$.
Here, for a Lie algebra $\mathfrak{k}$, the bilinear form $\kappa_{\mathfrak{k}}(\cdot,\cdot)$ denotes the Killing form of $\mathfrak{k}$.
Let $\omega^\natural$ denote the Segal-Sugawara conformal vector of $V^k$.
Since $J^{\{a\}}$ ($a\in \mathfrak{g}^\natural$) is a primary vector of conformal weight $1$ with respect to $\omega$,
we see that the vector $\omega^\mathrm{Vir}=\omega-\omega^\natural$ is a Virasoro vector (cf.\ \cite{LL}).
Let $U^k$ denote the Virasoro vertex operator algebra generated by $\omega^\mathrm{Vir}$.
Then, $V^k\otimes U^k \subset \mathcal{W}^k(\mathfrak{g},f)$.

Let $W=\mathcal{W}_k(\mathfrak{g},f)$ denote the simple quotient of $\mathcal{W}^k(\mathfrak{g},f)$.
Let $V$ and $U$ denote the image of $V^k$ and $U^k$ in $W$.
Then, $V\otimes U\subset W$ (cf. \cite[Proposition 4.3.5]{Li}).

By case-by-case computation, under the assumption of Theorem \ref{sec:mainthm}, we see that the pair $(\mathfrak{g},k)$ must be a pair in the list of the theorem.

\subsection{Level one affine VOAs associated with the Deligne exceptional Lie algebras}
\label{sec:levelone}

Let $\mathfrak{g}$ be a Deligne exceptional Lie algebra not of type $A_1$ with a fixed Cartan subalgebra $\mathfrak{h}\subset \mathfrak{g}$.
Let $\Phi(\mathfrak{g})$, $Q(\mathfrak{g})$ and $P(\mathfrak{g})$ denote the root system, root lattice and weight lattice of $\mathfrak{g}$.
Let $(\cdot|\cdot):\mathfrak{g}\times \mathfrak{g}\rightarrow \mathbb{C}$ denote the normalized bilinear form on $\mathfrak{g}$ with $(\alpha|\alpha)=2$ for each long root $\alpha\in \Phi(\mathfrak{g})$.
Fix a base $\alpha_1,\ldots,\alpha_l\in \Phi(\mathfrak{g})$ with the highest root $\theta\in \Phi(\mathfrak{g})$.
Let $(e^\theta,\theta,e^{-\theta})=A_1\subset \mathfrak{g}$ denote the $sl_2$-triple for $\theta$.
Consider the level one affine VOA $V_1(A_1)=V_{A_1}$ and  simple current
extension $V_1(A_1)\oplus V_1(A_1;\theta/2)$
 as in Proposition \ref{sec:propsce}.
That is, we consider the structure
\[
Y(e^{\beta},z)e^{\gamma}=\varepsilon(\beta,\gamma) X(e^\beta,z)e^\gamma,\quad \beta,\gamma\in A_1^\circ
\]
with
\[
\varepsilon(n\theta/2,m\theta/2)=\begin{cases} -1& \mbox{if}\ (n,m)\equiv (1,2),(2,2),(2,3),(3,1)\ (\mathrm{mod}\ 4),\\
1& \mathrm{otherwise},\end{cases}
\]
($n,m\in\mathbb{Z}$).
For $\beta,\gamma\in A_1+\theta/2$, we denote $I(e^\beta,z)e^\gamma=Y(e^\beta,z)e^\gamma$.

Suppose $\{e_\alpha|\alpha\in \Phi(\mathfrak{g})\}\cup\{\alpha_1^\vee,\ldots,\alpha_l^\vee\}$ be a Chevalley basis of $\mathfrak{g}$ with the function $\varepsilon:\Phi(\mathfrak{g})\times \Phi(\mathfrak{g})\rightarrow \mathbb{C}^\times$ such that
$[e_\alpha,e_\beta]=\varepsilon(\alpha,\beta)e_{\alpha+\beta}$ for $\alpha,\beta\in \Phi(\mathfrak{g})$ with $\alpha+\beta \in \Phi(\mathfrak{g})$ and $\alpha+\beta\neq 0$, $[e_{\alpha_i},e_{-\alpha_i}]=\varepsilon(\alpha_i,-\alpha_i)\alpha_i^\vee$
, $[e_{-\alpha_i},e_{\alpha_i}]=\varepsilon(-\alpha_i,\alpha_i)(-\alpha_i^\vee)$
and $e_{\pm \theta}=e^{\pm \theta}$.

Consider the level one affine VOA $V_1(\mathfrak{g})$.
Then, $V_1(A_1)$ is a subVOA of $V_1(\mathfrak{g})$.
Consider the commutant $V:=\mathrm{Comm}_{V_1(\mathfrak{g})}(V_1(A_1))$ of $V_1(A_1)$ in $V_1(\mathfrak{g})$.
Explicitly,
\begin{itemize}
\item $V=V_{\sqrt{3}A_1}$ for $\mathfrak{g}=A_2$;
\item $V=V_3(A_1)$ for $\mathfrak{g}=G_2$;
\item $V=V_1(A_1)^{\otimes 3}$ for $\mathfrak{g}=D_4$;
\item $V=V_{1}(\mathfrak{g}^\natural)$ for $\mathfrak{g}=F_4,E_6,E_7,E_8$.
\end{itemize}
Note that $V\otimes V_1(A_1)$ is embedded in $V_1(\mathfrak{g})$.
Let $k^\natural$ denote the level of $V$, that is, $k^\natural:=1$ for $\mathfrak{g}=A_2,D_4,F_4,E_6,E_7,E_8$ and 
$k^\natural:=3$ for $\mathfrak{g}=G_2$.
Let $h^{\vee,\natural}$ denote the dual Coxeter number of $\mathfrak{g}^\natural$ for $\mathfrak{g}=G_2,F_4,E_6,E_7,E_8$.
For $\mathfrak{g}=A_2$, set $h^{\vee,\natural}:=0$, and for $\mathfrak{g}=D_4$, set $h^{\vee,\natural}:=2$.
Note that $h^{\vee,\natural}=h^\vee_{0}$.

Let $N$ denote the module of $V$ defined to be $V_1(\mathfrak{g})\cong V\otimes V_1(A_1)\oplus N\otimes V_1(A_1;\theta/2)$ as $V\otimes V_1(A_1)$-modules.
Then, $N$ is a simple current of $V$. Explicitly,
\begin{itemize}
\item $N=V_{\sqrt{3}A_1+\sqrt{3}\alpha/2}$ for $\mathfrak{g}=A_2$;
\item $N=V_3(A_1;\alpha/2)$ for $\mathfrak{g}=G_2$;
\item $N=V_1(A_1;\alpha/2)^{\otimes 3}$ for $\mathfrak{g}=D_4$;
\item $N=V_{1}(\mathfrak{g}^\natural;\varpi_n)$ for $\mathfrak{g}=F_4,E_6,E_7,E_8$ with $n=3$ if $\mathfrak{g}=E_6$, and $n=l-1$ otherwise.
\end{itemize}
Here, $\alpha$ denotes the positeve root of $A_1$, and for $\mathfrak{g}=F_4,E_6,E_7,E_8$, the weights $\varpi_1,\ldots,\varpi_{l-1}$ denote the fundamental weights of the simple Lie algebra $\mathfrak{g}^\natural$ labelled as those in \cite{Bou}.
Note that we have the conformal weight grading
\[
N=\bigoplus_{n=0}^\infty N_{n+4/3}
\]
with 
\begin{itemize}
\item $N_{3/4}=\mathbb{C}e^{\sqrt{3}\alpha/2}\oplus \mathbb{C}e^{-\sqrt{3}\alpha/2}$ for $\mathfrak{g}=A_2$;
\item $N_{3/4}\cong L(\alpha/2)$ as $\mathfrak{g}^\natural=A_1$-module for $\mathfrak{g}=G_2$;
\item $N_{3/4}\cong L(\alpha/2)^{\otimes 3}$ as $\mathfrak{g}^\natural=A_1^{\times 3}$-module for $\mathfrak{g}=D_4$;
\item $N_{3/4}\cong L(\varpi_n)$ as $\mathfrak{g}^\natural$-module for $\mathfrak{g}=F_4,E_6,E_7,E_8$.
\end{itemize}

Consider the simple current extension $V\oplus N$ of the VOA $V$ with the intertwining operator $I:N\times N\rightarrow V[[z]]z^{\mathbb{C}}$ \cite{DL}.

\begin{lem}\label{sec:lemv}
There exists a non-zero vector $u\in N$ such that
\[
z^{-3/2}I(u,z)u\in V[[z]]\ \mbox{and} \ z^{-3/2}I(u,z)u|_{z=0}\neq 0.
\]
\end{lem}

\begin{proof}
First, suppose that $\mathfrak{g}$ is not of type $A_2$.
Let $i$ be an element of $\{1,\ldots,l\}$ with $(\alpha_i|\theta)\neq 0$.
Then, a non-zero weight vector $u\in N_{3/4}$ of weight $\alpha_i-\theta/2$ satisfies $z^{-3/2}I(u,z)u\in V[[z]]$ and $z^{-3/2}I(u,z)u|_{z=0}\neq 0$.
Explicitly, when $\mathfrak{g}$ is not of type $A_2$, the vector $z^{-3/2}I(u,z)u|_{z=0}=u(-5/2)u$ is a non-zero  weight vector with weight $2\alpha_i-\theta$ of the $\mathfrak{g}^\natural$-module $V_3$ (the conformal weight homogeneous subspace of $V$ with conformal weight $3$).
When $\mathfrak{g}$ is of type $A_2$, the vector $z^{-3/2}I(u,z)u|_{z=0}=u(-5/2)u$ belongs to $V_3$ and is a non-zero multiple of the vector $e^{2\alpha_i-\theta}\in V_3$.
\end{proof}

Consider the tensor product AIA $(V\oplus N)\otimes (V_1(A_1)\oplus V_1(A_1;\theta/2))$.
Then, by Lemma \ref{sec:lemv}, Lemma \ref{sec:lemgraded} and Proposition \ref{sec:propvir}, the subalgebra 
$V\otimes V_1(A_1)\oplus N\otimes V_1(A_1;\theta/2)$ is a vertex algebra and is the simple current extension of the VOA $V\otimes V_1(A_1)$.
As vertex algebras,
\[
V_1(\mathfrak{g})\cong V\otimes V_1(A_1)\oplus N\otimes V_1(A_1;\theta/2).
\]

We fix an isomorphism $\phi:V_1(\mathfrak{g})\rightarrow V\otimes V_1(A_1)\oplus N\otimes V_1(A_1;\theta/2)$ such that 
\[
\phi^{-1}(u\otimes v)=u(-1)v=v(-1)u,\quad u\in V, v\in V_1(A_1).
\]
In particular, $\phi(u)=u\otimes |0\rangle$ for $u\in V$ and $\phi(e_{\pm \theta})=|0\rangle\otimes e^{\pm \theta}$.
Let $(\mathfrak{h}^*)^\theta$ denote the subspace of $\mathfrak{h}^*$ orthogonal to $\theta$ with respect to $(\cdot|\cdot)|_{\mathfrak{h}^*}$
Let $P'$ denote the set
\[
P'=\{\mu\in (\mathfrak{h}^*)^\theta| \mu-\theta/2 \in S_{-1/2}\}.
\]
Let 
\[
\{e_\mu|\mu\in P'\} \subset N_{3/4}
\]
 denote the basis of $N_{3/4}$ such that $\phi(e_{\mu- \theta/2})=e_\mu\otimes e^{ -\theta/2}$ ($\mu\in P'$).
Then, $e_\mu\otimes e^{\theta/2}=-(|0\rangle\otimes e^{\theta})(0)(e_\mu\otimes e^{-\theta/2})=-\phi(e_\theta)(0)\phi(e_{\mu-\theta/2})=-\phi(e_{\theta}(0)e_{\mu-\theta/2})=-\varepsilon(\theta,\mu-\theta/2)\phi(e_{\mu+\theta/2})$.

Note that
\[
\mathfrak{g}_{-1/2}=\langle \phi^{-1}(e_\mu\otimes e^{-\theta/2})|\mu \in P'\rangle_{\mathbb{C}},
\]
and
\[
\mathfrak{g}_{1/2}=\langle \phi^{-1}(e_\mu\otimes e^{\theta/2})|\mu \in P'\rangle_{\mathbb{C}}.
\]

Note that the submodule containing $\{e_\mu|\mu\in P'\}$ of $N$ coincides with $N$.


Let $L(-3/5,0)$ denote the Virasoro minimal model of central charge $-3/5$ with the Virasoro vector $\omega^{\mathrm{Vir}}=L(-2)|0\rangle$ and $L(-3/5,3/4)$ the simple current of conformal weight $3/4$ with highest weight vector $|3/4\rangle$.
Consider the intertwining operator normalized as
\begin{eqnarray*}
I(|3/4\rangle,z)|3/4\rangle&=&-(k+h^\vee)\frac{c}{2h}|0\rangle z^{-3/2}-(k+h^\vee)\omega^{\mathrm{Vir}} z^{1/2}+\cdots\\
&=& \frac{h^\vee}{3}|0\rangle z^{-3/2}-\frac{5h^\vee}{6}\omega^{\mathrm{Vir}}z^{1/2}+\cdots.
\end{eqnarray*}
Here, $k=-h^\vee/6$, $c=-3/5$, and $h=3/4$.
Consider the tensor product AIA $(V\oplus N)\otimes (L(-3/5,0)\oplus L(-3/5,3/4))$.
Then, by Lemma \ref{sec:lemv} Lemma \ref{sec:lemgraded} and Proposition \ref{sec:propvir},
the subalgebra
\[
W'=V\otimes L(-3/5,0)\oplus N\otimes L(-3/5,3/4)
\]
is a vertex algebra and is the simple current extension 
of $V\otimes L(-3/5,0)$ with the intertwining operator $I\otimes I$.

Set $f=e_{-\theta}$, $e=-(1/2) e_{\theta}$ and $x=(1/2)\theta$.
Consider the universal $\mathcal{W}$-algebra $W^k=\mathcal{W}^k(\mathfrak{g},f)$
and simple quotient $W=\mathcal{W}_k(\mathfrak{g},f)$.

\begin{prop}\label{sec:propassignments}
The assignments
\[
W^k\ni J^{\{e_\beta\}}\mapsto e_\beta\otimes |0\rangle\in W',\quad \beta\in \Phi(\mathfrak{g}^\natural),
\]
\[
W^k\ni J^{\{\alpha_i^\vee\}}\mapsto \alpha_i^\vee\otimes |0\rangle \in W',\quad (\alpha_i|\theta)=0,
\]
\[
W^k\ni G^{\{e_{\mu}\}}\mapsto e_{\mu+\theta/2}\otimes |3/4\rangle \in W',\quad \mu\in S_{-1/2},
\]
\[
W^k\ni \omega \mapsto \omega^{\natural}\otimes |0\rangle+|0\rangle \otimes \omega^{\mathrm{Vir}}\in W'
\]
induces a surjective vertex operator algebra homomorphism $W^k\rightarrow W'$
and vertex operator algebra isomorphism $W\cong W'$.
\end{prop}

Let $S_{1/2}$ denote the set of all roots of $\mathfrak{g}_{1/2}$.
Consider the basis $u_{\gamma}=e_{\gamma}$, $\gamma \in S_{1/2}$.
Let $\{u^\gamma\}$ denote the dual basis such that $(f|[u_\gamma,u^\eta])=\delta_{\gamma,\eta}$.
Then, $u^{\gamma}=-\varepsilon(\gamma,-\gamma+\theta)e_{-\gamma+\theta}$.

In order to show the proposition, we need the following three lemmas.

Let $\mu,\nu$ be elements of $P'$.
Put $u=e_{\mu-\theta/2}$ and $v=e_{\nu-\theta/2}$.

\begin{lem}\label{sec:lemeqn}
As elements of $V_1(\mathfrak{g})$,
\begin{eqnarray}\label{eqn:lambda0}
&&\frac{2h^\vee}{3}\biggl([e,u](-1)v-\frac{1}{2}\theta(-1)[[e,u],v]\\
&&\quad\quad +\frac{1}{8}(e|[u,v])\left(\theta(-1)^2-2\theta(-2)\right)|0\rangle \biggr)\nonumber\\
&&\quad=-2\left(\frac{5h^\vee}{6}-\frac{k^\natural+h^{\vee,\natural}}{k^\natural}\right)(e|[u,v])\omega^{\natural}\nonumber\\
&&\quad\quad -\sum_{\gamma \in S_{1/2}}
\varepsilon(\gamma,-\gamma+\theta)[u,e_{-\gamma+\theta}]^\natural(-1)[e_{\gamma},v]^\natural\nonumber\\
&&\quad\quad+2\left(-\frac{h^\vee}{6}+1\right)\partial ([[e,u],v]^\natural)\nonumber.
\end{eqnarray}
\end{lem}

We show Lemma \ref{sec:lemeqn} in the next section.

\begin{lem}\label{sec:lemred}
As elements of $V_1(\mathfrak{g})$, we have
\[
e_\mu(1/2)e_\nu=-\varepsilon(\theta,\mu-\theta/2)e_{\mu+\theta/2}(1)e_{\nu-\theta/2},
\]
\[
e_\mu(-1/2)e_\nu=-\varepsilon(\theta,\mu-\theta/2)\biggl(e_{\mu+\theta/2}(0)e_{\nu-\theta/2}-\frac{1}{2}\theta(-1)e_{\mu+\theta/2}(1)e_{\nu-\theta/2}\biggr),
\]
\begin{eqnarray*}
&&e_\mu(-3/2)e_\nu=-\varepsilon(\theta,\mu-\theta/2)\biggl(e_{\mu+\theta/2}(-1)e_{\nu-\theta/2}\\
&&\quad -\frac{1}{2}\theta(-1)e_{\mu+\theta/2}(0)e_{\nu-\theta/2}+\frac{1}{8}(\theta(-1)^2-2\theta(-2))e_{\mu+\theta/2}(1)e_{\nu-\theta/2}\biggr).
\end{eqnarray*}

\end{lem}

\begin{proof}
Since the conformal weights of $e_\mu,e_\nu$ are $3/4$, we have
\[
I(e_\mu,z)e_\nu=\sum_{n=0}^\infty e_\mu(1/2-n) e_\nu z^{n-3/2}.
\]
By explicit computation, we have
\[
I(e^{\theta/2},z)e^{-\theta/2}=|0\rangle z^{-1/2}+\frac{1}{2}\theta z^{1/2}+\frac{1}{8}(\theta(-1)^2+2\theta(-2))|0\rangle z^{3/2}+\cdots.
\]
Also, since the conformal weights of the non-zero elements of $\mathfrak{g}\subset V_1(\mathfrak{g})$ are $1$, we have
\begin{eqnarray*}
Y(e_{\mu+\theta/2},z)e_{\nu-\theta/2}=\sum_{n=-1}^\infty e_{\mu+\theta/2}(-n)e_{\nu-\theta/2} z^{n-1}.
\end{eqnarray*}
Since $\phi(Y(e_{\mu+\theta/2},z)e_{\nu-\theta/2})=-\varepsilon(\theta,\mu-\theta/2)^{-1}Y(e_\mu\otimes e^{\theta/2},z)(e_\nu\otimes e^{-\theta/2})$
and $Y(e_\mu\otimes e^{\theta/2},z)(e_\nu\otimes e^{-\theta/2})=(I(e_\mu,z)e_\nu)\otimes(I(e^{\theta/2},z)e^{-\theta/2})$,
we have
\[
e_{\mu+\theta/2}(1)e_{\nu-\theta/2}=-\varepsilon(\theta,\mu-\theta/2)\phi^{-1}((e_\mu(1/2)e_\nu)\otimes |0\rangle),
\]
\begin{eqnarray*}
&&e_{\mu+\theta/2}(0)e_{\nu-\theta/2}=-\varepsilon(\theta,\mu-\theta/2)\\
&&\cdot \left(\frac{1}{2}\phi^{-1}((e_\mu(1/2)e_\nu)\otimes \theta)+\phi^{-1}((e_\mu(-1/2)e_\nu)\otimes |0\rangle)\right),
\end{eqnarray*}
and
\begin{eqnarray*}
&&e_{\mu+\theta/2}(-1)e_{\nu-\theta/2}=-\varepsilon(\theta,\mu-\theta/2)\\
&&\cdot\biggl( \frac{1}{8}\phi^{-1}((e_\mu(1/2)e_\nu)\otimes (\theta(-1)^2+\theta(-2))|0\rangle)\\
&&+\frac{1}{2}\phi^{-1}((e_\mu(-1/2)e_\nu)\otimes \theta)+\phi^{-1}((e_\mu(-3/2)e_\nu)\otimes |0\rangle)\biggr).
\end{eqnarray*}
Thus, we have the lemma.
\end{proof}

\begin{lem}\label{sec:lemtaiou}
\begin{eqnarray*}
e_{\mu+\theta/2}(1)e_{\nu-\theta/2}=-2\varepsilon(\theta,\mu-\theta/2)^{-1}(e|[u,v])|0\rangle,
\end{eqnarray*}
\begin{eqnarray*}
e_{\mu+\theta/2}(0)e_{\nu-\theta/2}=-2\varepsilon(\theta,\mu-\theta/2)^{-1}[[e,u],v],
\end{eqnarray*}
 and 
\[
e_{\mu+\theta/2}(-1)e_{\nu-\theta/2}=-2\varepsilon(\theta,\mu-\theta/2)^{-1}[e,u](-1)v.
\]
\end{lem}

\begin{proof}
We show the first equality.
We have $e_{\mu+\theta/2}(1)e_{\nu-\theta/2}=( e_{\mu+\theta/2}|e_{\nu-\theta/2})|0\rangle$.
By the invariance of $(\cdot|\cdot)$, 
\begin{eqnarray*}
(e_{\mu+\theta/2}|e_{\nu-\theta/2})&=&\varepsilon(\theta,\mu-\theta/2)^{-1}([e_\theta,e_{\mu-\theta/2}]|e_{\nu-\theta/2})\\
&=&\varepsilon(\theta,\mu-\theta/2)^{-1}(e_\theta|[e_{\mu-\theta/2},e_{\nu-\theta/2}]).
\end{eqnarray*}
Therefore,
\begin{eqnarray*}
e_{\mu+\theta/2}(1)e_{\nu-\theta/2}&=&\varepsilon(\theta,\mu-\theta/2)^{-1}(e_\theta|[e_{\mu-\theta/2},e_{\nu-\theta/2}])|0\rangle \\
&=&-2\varepsilon(\theta,\mu-\theta/2)^{-1}(e|[e_{\mu-\theta/2},e_{\nu-\theta/2}])|0\rangle\\
&=&-2\varepsilon(\theta,\mu-\theta/2)^{-1}(e|[u,v])|0\rangle.
\end{eqnarray*}
Hence, we have $e_{\mu+\theta/2}(1)e_{\nu-\theta/2}=-2\varepsilon(\theta,\mu-\theta/2)^{-1}(e|[u,v])|0\rangle$.
\end{proof}

\begin{proof}[Proof of Proposition \ref{sec:propassignments}]
Put $B_1=\{J^{\{e_\beta\}}| \beta\in \Phi(\mathfrak{g}^\natural)\}\sqcup \{J^{\{\alpha_i^\vee\}}|
(\alpha_i|\theta)=0\}$ and $B_2=\{G^{\{e_{\mu-\theta/2}\}}| \mu\in P'\}$.
Put $B=B_1\sqcup B_2\sqcup\{\omega\}$ and fix a total order on $B$.
Since the space spanned by the generators $B$ with the basis $B$ obeys the PBW-theorem, we obtain the linear map
$\psi:W^k\rightarrow W'$ induced from the assignments of the proposition by using the PBW-basis $\tilde{B}$ of $W^k$.
That is, we set
\[
\psi(v_1(m_1)v_2(m_2)\cdots v_s(m_s)|0\rangle)=\psi(v_1)(m_1)\psi(v_2)(m_2)\cdots \psi(v_s)(m_s)|0\rangle,
\]
($s\in \mathbb{Z}_+$ with $v_i\in S$, $k_i\in \mathbb{Z}_+$, $m_i\in \mathbb{Z}_{<0}$, $i=1,\ldots,s$, where
the 
sequence of pairs $(v_1,m_1),(v_2,m_2),\ldots, (v_s,m_s)$ is non-increasing in the
lexicographical order).
We show the compatibility of the OPE ($\lambda$-bracket)
\begin{equation}\label{eqn:compatibility}
[\psi(u)_\lambda \psi(v)]=\psi[u_\lambda v],
\end{equation}
 ($u,v\in B$) under $\psi$.
Let $u,v$ be elements of $B$, and suppose that $u=\omega$ or $v=\omega$.
Then, the coefficients of $[u_\lambda v]$ in $\lambda$ are multiples of monomials of the PBW-basis $\tilde{B}$.
Since $c_\mathcal{W}=c^\natural-3/5$, by comparing the conformal weights, we have eq.\ (\ref{eqn:compatibility}).
Here, $c^\natural$ is the central charge of $V$.

Let $u,v$ be elements of $B_1$.
Then, the coefficients of $[u_\lambda v]$ in $\lambda$ are multiples of monomials of the PBW-basis $\tilde{B}$.
We show eq.\ (\ref{eqn:compatibility}).
When $\mathfrak{g}=A_2$, the set $B_1$ is empty.
Therefore, we have eq.\ (\ref{eqn:compatibility}).
Suppose $\mathfrak{g}$ is not of type $A_2$.
Then, it suffices to show $(e_\alpha,e_\beta)^\natural=k^\natural(e_\alpha|e_\beta)$.
Since $\mathfrak{g}_0=\mathfrak{g}^\natural\oplus \mathbb{C} \theta$, we have $(e_\alpha,e_\beta)^\natural=(k+h^\vee/2)(e_\alpha|e_\beta)-(1/4)2h^{\vee,\natural}(e_\alpha|e_\beta)=(h^\vee/3-h^{\vee,\natural}/2)(e_\alpha|e_\beta)$.
Since $h^\vee/3-h^{\vee,\natural}/2=1$, we have eq.\ (\ref{eqn:compatibility}).
Consider the vertex operator subalgebra $T:=\langle B_1,\omega\rangle_{\mathrm{v.a.}}\subset W^k$.
Here,  $\langle A\rangle_{\mathrm{v.a.}}$  the smallest vertex subalgebra containing the subset $A\subset W^k$.
Since the OPEs among the elements of $B_1\sqcup \{\omega\}$ are compatible under $\psi$, the restriction $\psi|_T:T\rightarrow W'$ is a vertex operator algebra homomorphism.

Let $u$ and $v$ be elements of $B$, and suppose $u\in B_1\sqcup\{\omega\}$ or $v\in B_1\sqcup\{\omega\}$.
Then, the coefficients of $[u_\lambda v]$ in $\lambda$ are multiples of monomials of the PBW-basis $\tilde{B}$,
and we have eq.\ (\ref{eqn:compatibility}).

Finally, let $u,v$ be elements of $B_2$ with $\mu,\nu\in P'$
such that $u=G^{\{e_{\mu-\theta/2}\}}$ and $v=G^{\{e_{\nu-\theta/2}\}}$.
Then, the coefficients of $[u_\lambda v]$ in $\lambda$ belong to the vertex operator subalgebra $T$.
We show the compatibility of the OPE $\psi[u_\lambda v]=[\psi(u)_\lambda \psi(v)]$.
Since $W,W'$ are $(1/2)\mathbb{Z}_+$-graded VOAs, and the vectors $u,v,\psi u,\psi v$ are of conformal weight $3/2$, 
we have $\psi(u(i)v)=0=\psi u(i)\psi v$ for $i\geq 3$.
Therefore, it suffices to show $\psi(u(i)v)=\psi u(i)\psi v$ for $i=0,1,2$.
By eq.\ (\ref{eqn:omega}), it suffices to show $\psi(u(0)v)=\psi u(0)\psi v$.
We have $(e_\mu\otimes |3/4\rangle)(0) (e_\nu\otimes |3/4\rangle)= (e_\mu(-3/2)e_\nu)\otimes (h^\vee/3)|0\rangle-(e_\mu(1/2)e_\nu)\otimes (5h^\vee/6)\omega^\mathrm{Vir}$.
Therefore, by Proposition \ref{sec:propkw} and Lemma \ref{sec:lemeqn}--\ref{sec:lemtaiou}, we have $\psi(u(0)v)=\psi u(0)\psi v$, since $\psi|_T$ is a vertex operator algebra homomorphism.

Hence, we have eq.\ (\ref{eqn:compatibility}) for each $u,v\in B$.
Therefore, $\psi:W^k\rightarrow W'$ is a homomorphism of vertex operator algebras.
Since  $\psi(B)$ generates the vertex algebra $W'$, $\psi$ is surjective.
Since $W'$ is a simple vertex algebra, the homomorphism $\psi$ induces the isomorphism $W\cong W'$.
Thus, we have the proposition.
\end{proof}


By Proposition \ref{sec:propassignments}, we obtain Theorem \ref{sec:mainthm} except for the case $(\mathfrak{g},k)=(C_2,1/2)$,
which is proved similarly.
It will be considered in the forthcoming paper.

\begin{rem}
Suppose $\mathfrak{g}=D_4,E_6,E_7,E_8$. Then, the number $k=-h^\vee/6$ is not an admissible number.
Therefore, $\mathcal{W}_k (\mathfrak{g},f_\theta)$ is a new example of a $C_2$-cofinite $\mathcal{W}$-algebra.
When $\mathfrak{g}=C_2,A_1,A_2,G_2,F_4$ and $k=1/2,-1/3,-1/2,-2/3,-3/2$, the levels $k$ are admissible numbers, and the simple $\mathcal{W}$-algebras $\mathcal{W}_k(\mathfrak{g},f)$ have already been known to be $C_2$-cofinite \cite{A3}.
The vertex operator algebra $\mathcal{W}_{-1/2}(A_2,f_\theta)$ with certain other conformal vector is a 
{\it Bershadsky-Polyakov algebra}, and it has already known to be rational \cite{A4}.
Note that the abelian intertwining subalgebra $V_{\sqrt{3}A_1}\otimes L(-3/5,0) \oplus V_{\sqrt{3}A_1+\sqrt{3}\alpha/2}\otimes L(-3/5,3/4)$
of $V_{(\sqrt{3}/2)A_1}\otimes (L(-3/5,0)\oplus L(-3/5,3/4))$ 
is considered in \cite{FJM}.
\end{rem}

\section{Proof of Lemma \ref{sec:lemeqn}}

In this section, we give the proof of Lemma \ref{sec:lemeqn} for $\mathfrak{g}=D_4$ and $E_8$, which are the smallest and
largest examples with non-admissible levels, by using structure of the lattice vertex operator algebras $V_{D_4}$ and $V_{E_8}$.
The remaining cases are proved similarly (for non simply-laced cases, it is convenient to consider the ``folding''
 (cf. \cite[\S 7.9]{Kac})).
Let us take over the setting and notation in Section \ref{sec:levelone}.

Suppose that $\mathfrak{g}$ is simply-laced, that is, $\mathfrak{g}=A_2,D_4,E_6,E_7,E_8$.
Let $Q$ denote the root lattice of $\mathfrak{g}$.
Fix orientations $i\rightarrow j$ on the edges in the Dynkin diagram,
where $i,j=1,\ldots,l$ are nodes of the Dynkin diagram such that $(\alpha_i|\alpha_j)=-1$.
Define the $2$-cocycle $\varepsilon:Q\times Q\rightarrow \{1,-1\}$ 
by bimultiplicatively extending the assignment
\[
\varepsilon(\alpha_i,\alpha_j)=\begin{cases}-1& i=j,\ \mathrm{or}\ i \rightarrow  j \\
1& \mathrm{otherwise},\end{cases}
\]
(cf. \cite[\S 7.8]{Kac}).
Note that $\varepsilon(\theta,-\theta)=\varepsilon(-\theta,\theta)=-1$.
Consider the lattice vertex operator algebra 
\[
V_Q=\bigoplus_{n=0}^\infty (V_Q)_n
\]
 associated with $Q$ and the $2$-cocycle $\varepsilon$.
Since $V_Q$ is isomorphic to $V_1(\mathfrak{g})$, we consider $V_Q$ instead of $V_1(\mathfrak{g})$.
The weight $1$ subspace $(V_Q)_1$ with the Lie bracket $[a,b]=a(0)b$ is isomorphic to
the Lie algebra $\mathfrak{g}$.
We identify $\mathfrak{g} \cong (V_Q)_1$.
Then, the basis
\[
\{e^\alpha|\alpha\in \Phi(Q)\}\sqcup \{\alpha_i|i=1,\ldots,l\}\subset (V_Q)_1
\]
is a Chevalley basis of $\mathfrak{g}$ with $\varepsilon|_{\Phi(\mathfrak{g})\times \Phi(\mathfrak{g})}:\Phi(\mathfrak{g})\times \Phi(\mathfrak{g})\rightarrow \{1,-1\}$
satisfying the assumption in \S \ref{sec:levelone}.
Note that for $\alpha\in Q$, the vector $e^\alpha$ belongs to $(V_Q)_1$ if and only if $(\alpha|\alpha)=2$.

Then, it suffices to show the following lemma to prove Lemma \ref{sec:lemeqn}.
Let $\alpha,\beta$ be elements of $S_{-1/2}$.

\begin{lem}\label{sec:lemeqn2}
\begin{enumerate}
\item Suppose $(\alpha|\beta)=1$. Then,
\begin{eqnarray}\label{eqn:1lambda0}
\frac{h^\vee}{3}=\sharp\{\gamma\in S_{1/2}|(\alpha|-\gamma)=0,(\gamma|\beta)=-1\}.
\end{eqnarray}
\item Suppose $(\alpha|\beta)=0$. Then,
\begin{eqnarray}\label{eqn:0lambda0}
\sum_{\gamma\in S_{1/2},(\alpha|-\gamma)=0,(\gamma|\beta)=-1}\gamma=\left(\frac{h^\vee}{6}-1\right)(2\theta+\alpha-\beta).
\end{eqnarray}

\item Let $\alpha$ be an element of $S_{-1/2}$. Then,
\begin{eqnarray}\label{eqn:-1lambda0}
&&-\frac{h^\vee}{6}\left(\left(\alpha+\frac{1}{2}\theta\right)(-1)^2+\left(\alpha+\frac{1}{2}\theta\right)(-2)\right)|0\rangle
\\
&&=\left(\frac{5h^\vee}{6}-1-h^{\vee,\natural}\right)\omega^\natural \nonumber\\
&&\quad-\frac{1}{2}\sum_{\gamma\in S_{1/2},(\alpha|-\gamma)=0} \left( (\alpha-\gamma+\theta)(-1)^2+(\alpha-\gamma+\theta)(-2)\right)|0\rangle\nonumber\\
&&\quad-\left(\alpha+\frac{1}{2}\theta\right)(-1)^2|0\rangle
-\left(-\frac{h^\vee}{6}+1\right)\left(\alpha+\frac{1}{2}\theta\right)(-2)|0\rangle.\nonumber
\end{eqnarray}
\end{enumerate}
\end{lem}

We prove the remaining equations of the lemma in the next section for $\mathfrak{g}=D_4$ and $E_8$,
which are the smallest and largest examples with non-admissible levels.

\begin{proof}[Proof of Lemma \ref{sec:lemeqn} (when $\mathfrak{g}$ is simply-laced)]
Let $\alpha$, $\beta$ be elements of $S_{-1/2}$.
Set $u=e^\alpha$ and $v=e^\beta$.
Since the vectors $e^\sigma$, $\sigma\in S_{-1/2}$ span $\mathfrak{g}_{-1/2}$, it suffices to show
eq.\ (\ref{eqn:lambda0}) for $u,v$.
Note that $(\alpha|\beta)=-1,0,1,2$ and $(\alpha|\theta)=(\beta|\theta)=-1$.
We show by case-by-case computation.
Put $X=\{\gamma\in S_{1/2}|(\alpha|-\gamma)=0,(\gamma|\beta)=-1\}$.
Put $C:=\varepsilon(\theta,\alpha)\varepsilon(\theta+\alpha,\beta)$.
Then, for any $\gamma\in S_{1/2}$, we have $C=\varepsilon(\gamma,-\gamma+\theta)\varepsilon(\alpha,-\gamma+\theta)\varepsilon(\gamma,\beta)\varepsilon(\alpha-\gamma+\theta,\gamma+\beta)$.

When $(\alpha|\beta)=2$, the both hand sides of eq.\ (\ref{eqn:lambda0}) are $0$.

Suppose $(\alpha|\beta)=1$.
Since $(\alpha|\beta)=1$, the vector $e^{\alpha+\beta+\theta}$ is of conformal weight $2$, and we have 
$[[e,u],v]=0$ and $([e,u]|v)=0$.
The LHS of eq.\ (\ref{eqn:lambda0}) is equal to $-Ch^\vee/3 e^{\alpha+\beta+\theta}$.
The RHS of eq.\ (\ref{eqn:lambda0}) is equal to $-\sum_{\gamma\in X}C e^{\alpha+\beta+\theta}$.
By eq.\ (\ref{eqn:1lambda0}),
$-h^\vee/3e^{\alpha+\beta+\theta}=-\sum_{\gamma\in X}e^{\alpha+\beta+\theta}$.
Hence, we have eq.\ (\ref{eqn:lambda0}).

Suppose $(\alpha|\beta)=0$.
By comparing the orthogonal projections to $\mathbb{C}\theta$ of the both sides of eq.\ (\ref{eqn:0lambda0}),
\begin{equation}\label{eqn:0lambda1}
\sharp\{\gamma\in S_{1/2}|(\alpha|-\gamma)=0,(\gamma|\beta)=-1\}=4(h^\vee/6-1).
\end{equation}
We have $e^{\alpha+\beta+\theta}\in (V_Q)_1$, and $([e,u]|v)=0$.
The LHS of eq.\ (\ref{eqn:lambda0}) is equal to $-Ch^\vee/3 (\alpha+1/2 \theta)(-1)e^{\alpha+\beta+\theta}$.
The RHS of eq.\ (\ref{eqn:lambda0}) is equal to 
$-\sum_{\gamma\in X}C (\alpha-\gamma+\theta)(-1)e^{\alpha+\beta+\theta}-C(\alpha+1/2\theta)(-1)e^{\alpha+\beta+\theta}-C(\alpha+1/2\theta)(-1)e^{\alpha+\beta+\theta}-C(-h^\vee/6+1)(\theta+\alpha+\beta)(-1)e^{\alpha+\beta+\theta}$.
By eq.\ (\ref{eqn:0lambda0}) and (\ref{eqn:0lambda1}), $\sum_{\gamma\in X}(\alpha-\gamma+\theta)(-1)e^{\alpha+\beta+\theta}=(h^\vee/6-1)(2\theta+3\alpha+\beta)(-1)e^{\alpha+\beta+\theta}$.
Therefore, we have eq.\ (\ref{eqn:lambda0}), as desired.

Suppose $(\alpha|\beta)=-1$, that is, $\beta=-\alpha-\theta$.
Similarly, we have eq.\ (\ref{eqn:lambda0}) by using eq.\ (\ref{eqn:-1lambda0}).

Thus, we have the lemma.
\end{proof}

\subsection{Proof of Lemma \ref{sec:lemeqn2}}

We show Lemma \ref{sec:lemeqn2} when $\mathfrak{g}=D_4$ and $E_8$, which are the smallest and largest
example with non-admissible levels.
The remaining cases are proved similarly.

\subsubsection{The case $\mathfrak{g}=D_4$}

\begin{figure}[htbp]
\begin{center}
\def\maru{\lower3pt\hbox{\Large\hskip-0.0pt$\circ\hskip-0.0pt$}}
$$
\xymatrix@M-3.77pt@C40pt@R5pt{
\alpha_1&&\alpha_3\cr
\maru\ar@{-}[ddr]&&\maru\ar@{-}[ddl]\cr
\vbox to 35pt{}\cr
&\maru&\cr
&\alpha_2&\cr
\vbox to 20pt{}\cr
\ar@{.}[uuur]\maru&&\maru\ar@{-}[uuul]\cr
\theta&&\alpha_{4}\cr
\vbox to 20pt{}\cr
}
$$
\caption{Dynkin diagram of $D_{4}$}
\label{fig:fig2}
\end{center}
\end{figure}
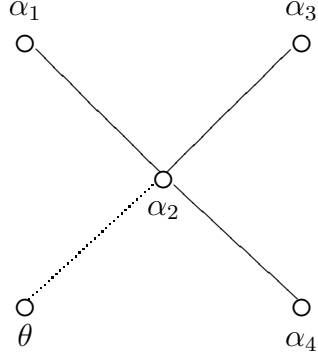

Suppose $\mathfrak{g}=D_4$. Then, $h^\vee=6$, and $h^{\vee,\natural}=2$.
We explicitly use the root system of $D_4$.
The root system $\Phi(D_4)$ of $D_4$ consists of the following $24$ elements (cf. \cite{Bou}):
\[
\mu \epsilon_i+\nu \epsilon_j \ (i,j=1,\ldots,4, i<j, \mu,\nu=\pm 1),
\]
with the indeterminate elements $\epsilon_1,\ldots,\epsilon_4$ with the bilinear form $(\cdot|\cdot)$ defined by linearly extending $(\epsilon_i|\epsilon_j)=\delta_{i,j}$.
Consider the simple roots 
$\alpha_i=\epsilon_{i}-\epsilon_{i+1}$ ($i=1,2,3$) and $\alpha_4=\epsilon_3+\epsilon_4$
with the highest root $\theta=\epsilon_1+\epsilon_2$.
The Dynkin diagram of $D_4$ is illustrated in Figure \ref{fig:fig2}.
Set
\[
S_{\pm 1/2}:=\{\alpha\in \Phi(D_4)| (\alpha|\theta/2)=\pm 1/2\}.
\]
Then, the sets of vectors $\{e^\alpha\in D_4|\alpha\in S_{\pm 1/2}\}$ are bases of $\mathfrak{g}_{\pm 1/2}$.
Explicitly, we have
\begin{eqnarray*}
S_{1/2}=\{\epsilon_i+\mu\epsilon_j|i\in\{1,2\},j\in\{3,4\},\mu\in\{\pm 1\}\}
\end{eqnarray*}
and
\begin{eqnarray*}
S_{-1/2}=\{-\epsilon_i+\mu\epsilon_j|i\in\{1,2\},j\in\{3,4\},\mu\in\{\pm 1\}\}
\end{eqnarray*}
Note that
\[
S_{\pm 1/2}=\{\alpha\in \Phi(D_4)| (\alpha|\varpi_2)=\pm 1\}.
\]
Here, $\varpi_1,\ldots,\varpi_4$ are the fundamental weights of $D_4$.

We show Lemma \ref{sec:lemeqn2} by case-by-case computation.

Note that the Weyl group $W_{D_4}$ of $D_4$ acts on $\Phi(D_4)$
and $S_{\pm 1/2}$ are invariant under the stabilizer $W_{D_4}^\theta$ of $\theta$,
and $W_{D_4}^\theta$ coincides with the subgroup $(\mathrm{Sym}_2\times \mathrm{Sym}_2)\ltimes (\mathbb{Z}/2\mathbb{Z})$.
Here, two $\mathrm{Sym}_2$'s are the symmetry group of the set $\{\epsilon_1,\epsilon_2\}$
and $\{\epsilon_4,\epsilon_3\}$,
and $(\mathbb{Z}/2\mathbb{Z})$ the transformations $\epsilon_i\mapsto \nu_i\epsilon_i$ ($i=1,\ldots,4$) with $\nu_1=\nu_2=1$ and $\nu_3=\nu_4\in \{\pm 1\}$.

Put $X=\{\gamma\in S_{1/2}|(\alpha|-\gamma)=0,(\gamma|\beta)=-1\}$.
Note that $\{\gamma\in S_{1/2}|(-\epsilon_1+\epsilon_3|\gamma)=0\}=\{\epsilon_1+\epsilon_3,\epsilon_2+\epsilon_4,\epsilon_2-\epsilon_4\}$.


\paragraph{{\bf The case (1)}}
Suppose $(\alpha|\beta)=1$.
We show eq.\ (\ref{eqn:1lambda0}).
The LHS of eq.\ (\ref{eqn:1lambda0}) is equal to $2$.
Therefore, we show
\[
\sharp\{\gamma\in S_{1/2}|(\alpha|-\gamma)=0,(\gamma|\beta)=-1\}=2.
\]
The all pairs $(\alpha,\beta)$ such that $\alpha,\beta\in S_{-1/2}$, $(\alpha|\beta)=1$ are given by the following pairs:
\[
(-\epsilon_i+\mu \epsilon_j, -\epsilon_i+\nu\epsilon_k),\quad \mu,\nu\in\{\pm 1\},i\in\{1,2\},j,k\in\{3,4\},j\neq k;
\]
\[
(-\epsilon_i+\mu\epsilon_j,-\epsilon_k+\mu\epsilon_j),\quad \mu\in\{\pm 1\},i,k\in\{1,2\},i\neq k,j\in\{3,4\};
\]

By the action of the Weyl group $W_{D_4}$, it suffices to consider $(\alpha,\beta)=(-\epsilon_1+\epsilon_3,-\epsilon_1+\epsilon_4)$, $(-\epsilon_1+\epsilon_3,-\epsilon_1-\epsilon_4)$,$(-\epsilon_1+\epsilon_3,-\epsilon_2+\epsilon_3)$.

Set $\alpha=-\epsilon_1+\epsilon_3$ and $\beta=-\epsilon_1+\epsilon_4$.
Then, $X=\{\epsilon_1+\epsilon_3,-\epsilon_2-\epsilon_4\}$.
Hence, $\sharp X=2$.

Set $\alpha=-\epsilon_1+\epsilon_3$ and $\beta=-\epsilon_1-\epsilon_4$.
Then, $X=\{\epsilon_1+\epsilon_3,-\epsilon_2+\epsilon_4\}$.
Hence, $\sharp X=2$.

Set $\alpha=-\epsilon_1+\epsilon_3$ and $\beta=-\epsilon_2+\epsilon_3$.
Then, $X=\{\epsilon_2+\nu\epsilon_4|\nu \in \{\pm 1\}\}$.
Hence, $\sharp X=2$.

Thus, we have eq.\ (\ref{eqn:1lambda0}).

\paragraph{{\bf The case (2)}}
Suppose $(\alpha|\beta)=0$.
We show eq.\ (\ref{eqn:0lambda0}).
The LHS of eq.\ (\ref{eqn:0lambda0}) is $0$.
Therefore, we show $X=\emptyset$.

The all pairs $(\alpha,\beta)$ such that $\alpha,\beta\in S_{-1/2}$, $(\alpha|\beta)=0$ are given by the following pairs:
\[
(-\epsilon_i+\mu\epsilon_j, -\epsilon_i-\mu\epsilon_j),\quad \mu\in\{\pm 1\},i\in\{1,2\},j\in \{3,4\};
\]
\[
(-\epsilon_i+\mu\epsilon_s,-\epsilon_j+\nu\epsilon_t)\quad 
\mu,\nu\in\{\pm 1\},\{i,j\}=\{1,2\},\{s,t\}=\{3,4\}.
\]

By the action of the Weyl group $W_{D_4}$, it suffices to consider $(\alpha,\beta)=(-\epsilon_1+\epsilon_3,-\epsilon_1-\epsilon_3)$, $(-\epsilon_1+\epsilon_3,-\epsilon_2+\epsilon_4)$, $(-\epsilon_1+\epsilon_3,-\epsilon_2-\epsilon_4)$.
For each case, we have $X=\emptyset$.

Thus, we have  eq.\ (\ref{eqn:0lambda0}).

\paragraph{{\bf The case (3)}}
Let $\alpha$ be an element of $S_{-1/2}$.
We show eq.\ (\ref{eqn:-1lambda0}).
By the action of the Weyl subgroup $W_{D_4}$, it suffices to consider $\alpha=-\epsilon_2+\epsilon_3=-\alpha_2$.
Set $\alpha=-\epsilon_2+\epsilon_3$.
Put $X=\{\gamma\in S_{1/2}|(\alpha|\gamma)=0\}$.
Then, $X=\{\epsilon_2+\epsilon_3,\epsilon_1+\epsilon_4,\epsilon_1-\epsilon_4\}$.
Note that $X=\{\alpha_2+\alpha_3+\alpha_4,\alpha_1+\alpha_2+\alpha_4,\alpha_1+\alpha_2+\alpha_3\}$.
We have
\begin{eqnarray*}
\omega^\natural&=&\frac{1}{4}\sum_{i=1,3,4}\alpha_i(-1)^2|0\rangle.
\end{eqnarray*}
We have
$
\theta=\alpha_1+2\alpha_2+\alpha_3+\alpha_4.
$
Therefore, $\alpha+1/2\theta=1/2\sum_{i=1,3,4}\alpha_i$.
Then, the RHS of eq.\ (\ref{eqn:-1lambda0}) is equal to
\begin{eqnarray*}
&&\frac{1}{2}\sum_{i=1,3,4}\alpha_i(-1)^2|0\rangle-\frac{1}{2}\sum_{i=1,3,4}\bigl(\alpha_i(-1)^2+\alpha_i(-2)\bigr)|0\rangle\\
&&\quad-\left(\frac{1}{2}\sum_{i=1,3,4}\alpha_i\right)(-1)^2|0\rangle=-\biggl(\left(\alpha+\frac{1}{2}\theta\right)
(-1)^2\\
&&\quad+\left(\alpha+\frac{1}{2}\theta\right)(-2)\biggr)|0\rangle,
\end{eqnarray*}
which coincides with the LHS.
Hence, we have eq.\ (\ref{eqn:-1lambda0}).

Thus, we have Lemma \ref{sec:lemeqn2}.

\subsubsection{The case $\mathfrak{g}=E_8$}

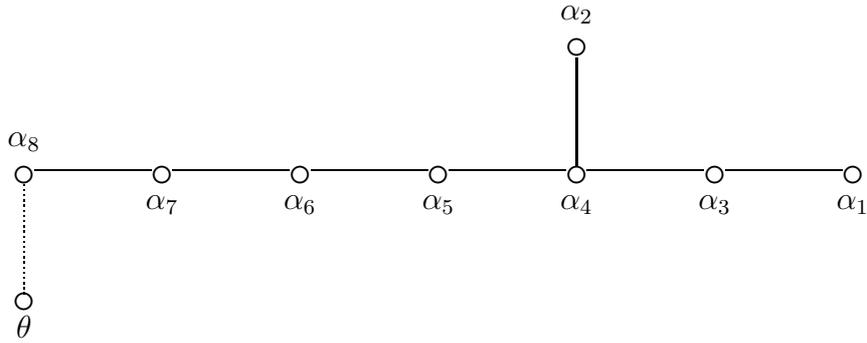
\begin{figure}[htbp]
\begin{center}
\def\maru{\lower3pt\hbox{\Large\hskip-0.0pt$\circ\hskip-0.0pt$}}
$$
\xymatrix@M-3.77pt@C40pt@R5pt{
&&&&\alpha_2\cr
&&&&\maru\ar@{-}[ddd]\cr
\vbox to 20pt{}\cr
\alpha_8\cr
\maru\ar@{-}[r]&\maru\ar@{-}[r]&\maru\ar@{-}[r]&\maru\ar@{-}[r]&\maru\ar@{-}[r]&\maru\ar@{-}[r]&\maru\cr
&\alpha_7&\alpha_6&\alpha_5&\alpha_4&\alpha_3&\alpha_1\cr
\vbox to 20pt{}\cr
\ar@{.}[uuu]\maru\cr
\theta\cr
\vbox to 20pt{}\cr
}
$$
\caption{Dynkin diagram of $E_8$}
\label{fig:fig1}
\end{center}
\end{figure}

Suppose $\mathfrak{g}=E_8$.
We explicitly use the root system of $E_8$.
The root system $\Phi(E_8)$ of $E_8$ consists of the following $240$ elements (cf. \cite{Bou}):
\[
\mu \epsilon_i+\nu \epsilon_j \ (i,j=1,\ldots,8, i<j, \mu,\nu=\pm 1),
\]
\[
\frac{1}{2}\sum_{i=1}^8 (-1)^{\nu_i}\epsilon_i\ \left(\nu_1,\ldots,\nu_8=\pm 1\ \mbox{such that} \sum_{i=1}^8 \nu_i \ \mbox{is even}\right),
\]
with the indeterminate elements $\epsilon_1,\ldots,\epsilon_8$ with the bilinear form $(\cdot|\cdot)$ defined by linearly extending $(\epsilon_i|\epsilon_j)=\delta_{i,j}$.
Consider the simple roots $\alpha_1=1/2(\epsilon_1-\epsilon_2-\epsilon_3-\epsilon_4-\epsilon_5-\epsilon_6-\epsilon_7+\epsilon_8)$,
$\alpha_2=\epsilon_1+\epsilon_2$, $\alpha_i=-\epsilon_{i-2}+\epsilon_{i-1}$ ($i=3,\ldots,8$)
with the highest root $\theta=\epsilon_7+\epsilon_8$.
The Dynkin diagram of $E_8$ is illustrated in Figure \ref{fig:fig1}.
Set
\[
S_{\pm 1/2}:=\{\alpha\in \Phi(E_8)| (\alpha|\theta/2)=\pm 1/2\}.
\]
Then, the sets of vectors $\{e^\alpha\in E_8|\alpha\in S_{\pm 1/2}\}$ are bases of $\mathfrak{g}_{\pm 1/2}$.
Explicitly, we have
\begin{eqnarray*}
&&S_{1/2}=\{\mu\epsilon_i+\epsilon_7,\mu\epsilon_i+\epsilon_8|i=1,\ldots,6,\mu\in\{\pm 1\}\}\\
&&\quad\sqcup \left\{\left.\frac{1}{2}\sum_{i=1}^8 (-1)^{\nu_i} \epsilon_i\right|\nu_1,\ldots,\nu_6\in\{\pm 1\},\nu_7=\nu_8=1,\sum_{i=1}^8 \nu_i:\ \mbox{even}\right\},
\end{eqnarray*}
and
\begin{eqnarray*}
&&S_{-1/2}=\{\mu\epsilon_i-\epsilon_7,\mu\epsilon_i-\epsilon_8|i=1,\ldots,6,\mu\in\{\pm 1\}\}\\
&&\quad\sqcup \left\{\left.\frac{1}{2}\sum_{i=1}^8 (-1)^{\nu_i} \epsilon_i\right|\nu_1,\ldots,\nu_6\in\{\pm 1\},\nu_7=\nu_8=-1,\sum_{i=1}^8 \nu_i:\ \mbox{even}\right\}.
\end{eqnarray*}
Note that
\[
S_{\pm 1/2}=\{\alpha\in \Phi(E_8)| (\alpha|\varpi_8)=\pm 1\}.
\]
Here, $\varpi_1,\ldots,\varpi_8$ are the fundamental weights of $E_8$.

We show Lemma \ref{sec:lemeqn2} by case-by-case computation.

Note that the Weyl group $W_{D_8}$ of $D_8\subset E_8$ acts on $\Phi(E_8)$
and $S_{\pm 1/2}$ are invariant under the stabilizer $W_{D_8}^\theta$ of $\theta$,
and $W_{D_8}^\theta$ coincides with the subgroup $(\mathrm{Sym}_6\times \mathrm{Sym}_2)\ltimes (\mathbb{Z}/2\mathbb{Z})^5$.
Here, $\mathrm{Sym}_6$ and $\mathrm{Sym}_2$ are the symmetry group of the set $\{\epsilon_1,\ldots,\epsilon_6\}$
and $\{\epsilon_7,\epsilon_8\}$,
and $(\mathbb{Z}/2\mathbb{Z})^5$ the transformations $\epsilon_i\mapsto \nu_i\epsilon_i$ ($i=1,\ldots,6$) with $\nu_1,\ldots,\nu_6\in \{\pm 1\}$ and $\sum_{i=1}^6\nu_i=1$.

Put $X=\{\gamma\in S_{1/2}|(\alpha|-\gamma)=0,(\gamma|\beta)=-1\}$.


\paragraph{{\bf The case (1)}}
Suppose $(\alpha|\beta)=1$.
We show eq.\ (\ref{eqn:1lambda0}).
The LHS of eq.\ (\ref{eqn:1lambda0}) is equal to $10$.
Therefore, we show
\[
\sharp\{\gamma\in S_{1/2}|(\alpha|-\gamma)=0,(\gamma|\beta)=-1\}=10.
\]
The all pairs $(\alpha,\beta)$ such that $\alpha,\beta\in S_{-1/2}$, $(\alpha|\beta)=1$ are given by the following pairs:
\[
(\mu \epsilon_i-\epsilon_k, \nu\epsilon_j-\epsilon_k),\quad \mu,\nu\in\{\pm 1\},i,j\in\{1,\ldots,6\},i\neq j,k\in \{7,8\};
\]
\[
(\mu\epsilon_i-\epsilon_7,\mu\epsilon_i-\epsilon_8),(\mu\epsilon_i-\epsilon_8,\mu\epsilon_i-\epsilon_7)\quad \mu\in\{\pm 1\},i\in\{1,\ldots,6\};
\]
\begin{eqnarray*}
&&\left(\mu \epsilon_s-\epsilon_t,\frac{1}{2}\sum_{i=1}^8\nu_i\epsilon_i\right),
\left(\frac{1}{2}\sum_{i=1}^8\nu_i\epsilon_i,\mu \epsilon_s-\epsilon_t\right),
\quad s\in\{1,\ldots,6\},\\
&&t\in\{7,8\},\mu,\nu_1,\ldots,\nu_8\in\{\pm1\}, \mu=\nu_s, \nu_7=\nu_8=-1,\sum_{i=1}^8\nu_i=1;
\end{eqnarray*}
\begin{eqnarray*}
&&\left(\frac{1}{2}\sum_{i=1}^8\nu_i\epsilon_i,\frac{1}{2}\left(-\sum_{i= j,k}\nu_i\epsilon_i+\sum_{i=1,\ldots,8,i\neq j,k}\nu_i\epsilon_i\right)\right),\\
&&\nu_1,\ldots,\nu_8\in \{\pm 1\},\nu_7=\nu_8=-1,\sum\nu_i=1,j,k\in\{1,\ldots,6\},j\neq k.
\end{eqnarray*}

By the action of the Weyl subgroup $W_{D_8}$, it suffices to consider $(\alpha,\beta)=(\epsilon_1-\epsilon_7,\epsilon_2-\epsilon_7)$, $(\epsilon_1-\epsilon_7,\epsilon_1-\epsilon_8)$,
 $(\epsilon_1-\epsilon_7,1/2(\sum_{i=1}^6\epsilon_i-\epsilon_7-\epsilon_8))$,
 $(-1/2\sum_{i=1}^8\epsilon_i,-1/2(-\epsilon_1-\epsilon_2+\sum_{i=3}^8\epsilon_i))$.

Set $\alpha=\epsilon_1-\epsilon_7$ and $\beta=\epsilon_2-\epsilon_7$.
Then, $X=\{\epsilon_1+\epsilon_7,-\epsilon_2+\epsilon_8,1/2(\epsilon_1-\epsilon_2+\sum_{i=3}^6\nu_i\epsilon_i+\epsilon_7+\epsilon_8)|\nu_3,\ldots,\nu_6\in \{\pm 1\},\sum_{i=3}^6 \nu_i=-1\}$.
Hence, $\sharp X=2+2^4/2=10$.

Set $\alpha=\epsilon_1-\epsilon_7$ and $\beta=\epsilon_1-\epsilon_8$.
Then, $X=\{\nu \epsilon_i+\epsilon_8|\nu \in \{\pm 1\},i=2,\ldots,6\}$.
Hence, $\sharp X=2\times 5=10$.

Set $\alpha=\epsilon_1-\epsilon_7$ and $\beta=1/2(\sum_{i=1}^6\epsilon_i-\epsilon_7-\epsilon_8)$.
Then, $X=\{-\epsilon_j+\epsilon_8,1/2(\epsilon_1+\epsilon_k-\sum_{i=2,\ldots,6,i\neq k}\epsilon_i+\epsilon_7+\epsilon_8)|j,k\in\{2,\:\ldots,6\}\}$.
Hence, $\sharp X=5+5=10$.

Set $\alpha=-1/2\sum_{i=1}^6\epsilon_i$ and $\beta=-1/2(-\epsilon_1-\epsilon_2+\sum_{i=3}^8\epsilon_i)$.
Then, $X=\{-\epsilon_s+\epsilon_t,1/2(-\sum_{i=1,\ldots,6,i\neq j,k}\epsilon_i+\sum_{i=j,k,7,8}\epsilon_i)|s\in\{1,2\}, t\in\{7,8\}, j,k\in\{2,\:\ldots,6\},j\neq k\}$.
Hence, $\sharp X=2\times 2+\begin{pmatrix}4\\ 2\end{pmatrix}=10$.

Thus, we have eq.\ (\ref{eqn:1lambda0}).

\paragraph{{\bf The case (2)}}
Suppose $(\alpha|\beta)=0$.
We show eq.\ (\ref{eqn:0lambda0}).
The LHS of eq.\ (\ref{eqn:0lambda0}) is $8\theta+4\alpha-4\beta$.
Therefore, we show $\sum_{\gamma\in X}\gamma=8\theta+4\alpha-4\beta$.

The all pairs $(\alpha,\beta)$ such that $\alpha,\beta\in S_{-1/2}$, $(\alpha|\beta)=0$ are given by the following pairs:
\[
(\mu \epsilon_i-\epsilon_k, -\mu\epsilon_i-\epsilon_k),\quad \mu\in\{\pm 1\},i\in\{1,\ldots,6\},k\in \{7,8\};
\]
\[
(\mu\epsilon_i-\epsilon_7,\nu\epsilon_j-\epsilon_8),(\nu\epsilon_j-\epsilon_8,\mu\epsilon_i-\epsilon_7)\quad 
\mu,\nu\in\{\pm 1\},i,j\in\{1,\ldots,6\},i\neq j;
\]
\begin{eqnarray*}
&&\left(-\nu_s \epsilon_s-\epsilon_t,\frac{1}{2}\sum_{i=1}^8\nu_i\epsilon_i\right),
\left(\frac{1}{2}\sum_{i=1}^8\nu_i\epsilon_i,-\nu_s \epsilon_s-\epsilon_t\right),
\quad s\in\{1,\ldots,6\},\\
&&t\in\{7,8\},\nu_1,\ldots,\nu_8\in\{\pm1\}, \nu_7=\nu_8=-1,\sum_{i=1}^8\nu_i=1;
\end{eqnarray*}
\begin{eqnarray*}
&&\left(\frac{1}{2}\sum_{i=1}^8\nu_i\epsilon_i,\frac{1}{2}\left(-\sum_{i=1,\ldots,6,i\neq j,k}\nu_i\epsilon_i+\sum_{i=j,k,7,8}\nu_i\epsilon_i\right)\right),\\
&&\nu_1,\ldots,\nu_8\in \{\pm 1\},\nu_7=\nu_8=-1,\sum\nu_i=1,j,k\in\{1,\ldots,6\},j\neq k.
\end{eqnarray*}

By the action of the Weyl subgroup $W_{D_8}$, it suffices to consider $(\alpha,\beta)=(\epsilon_1-\epsilon_7,-\epsilon_1-\epsilon_7)$, $(\epsilon_1-\epsilon_7,\epsilon_2-\epsilon_8)$, $(\epsilon_1-\epsilon_7,-1/2\sum_{i=1}^8\epsilon_i)$, $(-1/2\sum_{i=1}^8\epsilon_i,-1/2(-\sum_{i=1}^4\epsilon_i+\sum_{i=5}^8\epsilon_i))$.

Set $\alpha=\epsilon_1-\epsilon_7$ and $\beta=-\epsilon_1-\epsilon_7$.
Then, $X=\{1/2(\sum_{i=1,7,8}\epsilon_i+\sum_{i=2}^6\nu_i\epsilon_i)|\nu_2,\ldots,\nu_6\in \{\pm 1\},\sum_{i=2}^6 \nu_i=1\}$.
Hence,  $\sum_{\gamma\in X}\gamma=8(\epsilon_1+\epsilon_7+\epsilon_8)=8\theta+4\alpha-4\beta$.

Set $\alpha=\epsilon_1-\epsilon_7$ and $\beta=\epsilon_2-\epsilon_8$.
Then, $X=\{\mu \epsilon_k+\epsilon_8,1/2(\epsilon_1-\epsilon_2+\sum_{i=3}^6\nu_i\epsilon_i+\epsilon_7+\epsilon_8|\mu,\nu_3,\ldots,\nu_6 \in \{\pm 1\},k\in\{3,\ldots,6\},\sum_{i=1}^6\nu_i=-1\}$.
Hence, $\sum_{\gamma\in X}\gamma=12\epsilon_8+8\epsilon_7+4(\epsilon_1-\epsilon_2)=8\theta+4\alpha-4\beta$.

Set $\alpha=\epsilon_1-\epsilon_7$ and $\beta=-1/2\sum_{i=1}^8\epsilon_i$.
Then, $X=\{\epsilon_1+\epsilon_7,\epsilon_k+\epsilon_8,1/2(\sum_{i=1,\ldots,8,i\neq s,t}\epsilon_i-\sum_{i=s,t}\epsilon_i)|k,s,t\in\{2,\:\ldots,6\},s\neq t\}$.
Hence, $\sum_{\gamma\in X}\gamma=8\theta+4\alpha-4\beta$.

Set $\alpha=-1/2\sum_{i=1}^6\epsilon_i$ and $\beta=-1/2(-\sum_{i=1}^4\epsilon_i+\sum_{i=5}^8\epsilon_i)$.
Then, $X=\{-\epsilon_s+\epsilon_t,1/2(\sum_{i=j,k,7,8}\epsilon_i-\sum_{i=1,\ldots,6,i\neq j,k}\epsilon_i)|s,j\in\{1,\ldots,4\}, t\in\{7,8\}, k\in\{5,6\}\}$.
Hence, $\sum_{\gamma\in X}\gamma=8\theta+4\alpha-4\beta$.

Thus, we have  eq.\ (\ref{eqn:0lambda0}).

\paragraph{{\bf The case (3)}}
Let $\alpha$ be an element of $S_{-1/2}$.
We show eq.\ (\ref{eqn:-1lambda0}).
Note that $X=\{\gamma\in S_{1/2}|(\alpha|\gamma)=0\}$.
By the action of the Weyl subgroup $W_{D_8}$, it suffices to consider $\alpha=-\alpha_8, -\alpha_1-\sum_{i=3}^8\alpha_i$.
We have
\begin{eqnarray*}
\omega^\natural&=&\frac{1}{2(1+h^{\vee,\natural})}\left(\sum_{\mu\in \Phi(\mathfrak{g}^\natural)}\frac{1}{2}(\mu(-1)^2+\mu(-2))+\sum_{i=1}^7\alpha_i(-1)\varpi^\natural_i(-1)\right)|0\rangle\\
&=&\frac{1}{4}\bigl(4 \alpha_1(-1)^2 + 8 \alpha_1(-1) \alpha_2(-1) + 7 \alpha_2(-1)^2 + 12 \alpha_1(-1) \alpha_3(-1)\\
&& + 16 \alpha_2(-1) \alpha_3(-1) + 
 12 \alpha_3(-1)^2 + 16 \alpha_1(-1) \alpha_4(-1)\\
&& + 24 \alpha_2(-1) \alpha_4(-1) + 32 \alpha_3(-1) \alpha_4(-1) + 24 \alpha_4(-1)^2 \\
&&+ 12 \alpha_1(-1) \alpha_5(-1) + 18 \alpha_2(-1) \alpha_5(-1) + 24 \alpha_3(-1) \alpha_5(-1) \\
&&+ 36 \alpha_4(-1) \alpha_5(-1) + 
 15 \alpha_5(-1)^2 + 8 \alpha_1(-1) \alpha_6(-1) \\
&&+ 12 \alpha_2(-1) \alpha_6(-1) + 16 \alpha_3(-1) \alpha_6(-1) + 
 24 \alpha_4(-1) \alpha_6(-1) \\
&&+ 20 \alpha_5(-1) \alpha_6(-1) + 8 \alpha_6(-1)^2 + 4 \alpha_1(-1) \alpha_7(-1)\\
&& + 6 \alpha_2(-1) \alpha_7(-1) + 
 8 \alpha_3(-1) \alpha_7(-1) + 12 \alpha_4(-1) \alpha_7(-1)\\
&& + 10 \alpha_5(-1) \alpha_7(-1) + 8 \alpha_6(-1) \alpha_7(-1) + 3 \alpha_7(-1)^2\bigr)|0\rangle.
\end{eqnarray*}
Here, $\Phi(\mathfrak{g}^\natural)=\{\gamma\in \Phi(\mathfrak{g})|(\theta|\gamma)=0\}$, and $\varpi^\natural_1,\ldots,\varpi^\natural_7$ are  the fundamental weights for the simple roots $\alpha_1,\ldots,\alpha_7$ of $\mathfrak{g}^\natural=E_7$.
We have
\[
\theta=2\alpha_1+3\alpha_2+4\alpha_3+6\alpha_4+5\alpha_5+4\alpha_6+3\alpha_7+2\alpha_8.
\]
Put $-X+\theta:=\{-\gamma+\theta|\gamma\in X\}$.
Then, $-X+\theta=\{\gamma\in S_{1/2}|(\gamma|\alpha)=1\}$.

Set $\alpha=-\alpha_8=\epsilon_6-\epsilon_7$.
Then, the set $-X+\theta$ consists of the
following elements:
\[
-\epsilon_6+\epsilon_8,\mu\varepsilon_k+\epsilon_7,\quad k\in\{1,\ldots,5\},\mu\in\{\pm 1\};
\]
\[
\frac{1}{2}\left(\sum_{i=1}^5\nu_i\epsilon_i-\epsilon_6+\epsilon_7+\epsilon_8\right),\quad \nu_1,\ldots,\nu_5\in\{\pm 1\},\sum_{i=1}^5\nu_i=-1.
\]
Note that the latter elements are all elements $\gamma\in S_{1/2}$ such that $(\gamma|\varpi_1)=(\gamma|\varpi_7)=(\gamma|\varpi_8)=1$.

Therefore, the set $-X+\theta$ consists of the following elements:
\[
2\alpha_1+2\alpha_2+3\alpha_3+4\alpha_4+3\alpha_5+2\alpha_6+\alpha_7+\alpha_8,
\quad\alpha_2+\sum_{k=4}^8\alpha_k,
\]
\[
\sum_{k=i+2}^8\alpha_k,\ \alpha_2+\alpha_3+2\sum_{k=4}^{j+1}\alpha_k+\sum_{k=j+2}^8\alpha_k,\quad i\in\{1,\ldots,5\},j\in\{2,\ldots,5\};
\]
\begin{eqnarray*}
(1, 1, 1, 2, 1, 1, 1, 1),
(1, 1, 2, 2, 1, 1, 1, 1),
(1, 1, 1, 2, 2, 1, 1, 1),\\
(1, 1, 2, 2, 2, 1, 1, 1),
(1, 1, 1, 2, 2, 2, 1, 1),
(1, 1, 2, 3, 2, 1, 1, 1),\\
(1, 1, 2, 2, 2, 2, 1, 1),
(1, 2, 2, 3, 2, 1, 1, 1),
(1, 1, 2, 3, 2, 2, 1, 1),\\
(1, 2, 2, 3, 2, 2, 1, 1),
(1, 1, 2, 3, 3, 2, 1, 1),
(1, 2, 2, 3, 3, 2, 1, 1),\\
(1, 2, 2, 4, 3, 2, 1, 1),
(1, 2, 3, 4, 3, 2, 1, 1),
(1, 0, 1, 1, 1, 1, 1, 1),\\
(1, 1, 1, 1, 1, 1, 1, 1).
\end{eqnarray*}
Here, the symbol $(n_1,\ldots,n_8)$ denotes $\sum_{k=1}^8n_k\alpha_k$.
Then, we see that the RHS of eq.\ (\ref{eqn:-1lambda0}) equals $-5((\alpha+\theta/2)(-1)^2+(\alpha+\theta/2)(-2))|0\rangle$, which coincides with the LHS.

Set $\alpha=-\alpha_1-\sum_{i=3}^8\alpha_i=1/2\left(\sum_{i=1}^6\epsilon_i-\sum_{i=7}^8\epsilon_i\right)$.
Then, the set $-X+\theta$ consists of the
following elements:
\[
-\epsilon_s+\epsilon_t,\quad s\in\{1,\ldots,6\},t\in\{7,8\};
\]
\[
\frac{1}{2}\left(\sum_{i=j,k,7,8}\nu_i\epsilon_i-\sum_{i=1,\ldots,6,i\neq j,k}\epsilon_i\right),\quad j,k\in \{1,\ldots,6\},j\neq k.
\]
Note that the latter elements are all elements $\gamma\in S_{1/2}$ such that $(\gamma|\varpi_2)=(\gamma|\varpi_8)=1$, $(\gamma|\varpi_1)\geq 1$.

Therefore, the set $-X+\theta$ consists of the following elements:
\[
\sum_{k=i+2}^8\alpha_k,\quad i\in\{1,\ldots,6\},
\]
\[
2\alpha_1+2\alpha_2+3\alpha_3+5\alpha_4+4\alpha_5+3\alpha_6+2\alpha_7+\alpha_8,
\]
\[
2\alpha_1+2\alpha_2+4\alpha_3+5\alpha_4+4\alpha_5+3\alpha_6+2\alpha_7+\alpha_8,
\]
\[
\sum_{k=i+2}^8\alpha_k,\alpha_2+\alpha_3+2\sum_{k=4}^{j+1}\alpha_k+\sum_{k=j+2}^8\alpha_k,\quad i\in\{1,\ldots,5\},j\in\{2,\ldots,5\};
\]
\begin{eqnarray*}
 (1, 1, 1, 2, 1, 1, 1, 1), 
 (1, 1, 2, 2, 1, 1, 1, 1), 
 (1, 1, 1, 2, 2, 1, 1, 1),\\ 
 (1, 1, 2, 2, 2, 1, 1, 1), 
 (1, 1, 1, 2, 2, 2, 1, 1), 
 (1, 1, 2, 3, 2, 1, 1, 1), \\
 (1, 1, 2, 2, 2, 2, 1, 1), 
 (1, 1, 2, 3, 2, 2, 1, 1), 
 (1, 1, 2, 3, 3, 2, 1, 1), \\
 (1, 1, 1, 1, 1, 1, 1, 1), 
 (1, 1, 1, 2, 2, 2, 2, 1), 
 (1, 1, 2, 2, 2, 2, 2, 1), \\
 (1, 1, 2, 3, 2, 2, 2, 1), 
 (1, 1, 2, 3, 3, 2, 2, 1), 
 (1, 1, 2, 3, 3, 3, 2, 1).
\end{eqnarray*}
Here, $(n_1,\ldots,n_8)$ denotes $\sum_{k=1}^8n_k\alpha_k$.
Then, we see that the RHS of eq.\ (\ref{eqn:-1lambda0}) equals $-5((\alpha+\theta/2)(-1)^2+(\alpha+\theta/2)(-2))|0\rangle$, which coincides with the LHS.

We see that for $\alpha=-\alpha_8=\epsilon_6-\epsilon_7$ and $\alpha=-\alpha_1-\sum_{i=2}^8\alpha_i=1/2(\sum_{i=1}^6\epsilon_i-\sum_{i=7,8}\epsilon_i)$, the both hand sides of eq.\ (\ref{eqn:-1lambda0}) equals $-5((\alpha+\theta/2)(-1)^2+(\alpha+\theta/2)(-2))|0\rangle$.
Hence, we have eq.\ (\ref{eqn:-1lambda0}).

Thus, we have Lemma \ref{sec:lemeqn2}.

\section{Remarks}
\begin{rem}
Let $\mathfrak{g}$ be a Deligne exceptional Lie algebra not of type $A_1$ with $k=-h^\vee/6$ or $C_2$ with $k=1/2$.
Set $W=\mathcal{W}_k(\mathfrak{g},f_\theta)$.
Consider the group $\mathbb{Z}_2=\{\mathrm{id},\iota\}$ as in Theorem \ref{sec:mainthm}.

We call the $\mathrm{id}$-twisted modules  the {\it $\iota$-weight twisted modules} and 
 the $\iota$-twisted modules  the {\it $\mathrm{id}$-weight twisted modules}.
The $\mathrm{id}$-weight twisted modules are usually called the {\it Ramond twisted modules}.
We call the usual modules ($=\mathrm{id}$-twisted modules $=\mathrm{\iota}$-weight twisted modules)
the {\it Neveu-Schwarz twisted modules}.

Note that the (twisted) modules of $W$ have compatible actions of $\mathbb{Z}_2$.
We fix an action of $\mathbb{Z}_2$ for each (twisted) module.

Let $v$ be an element of $W$.
Let $g,h$ be elements of $\mathbb{Z}_2$.
Let  $M$ be a $g$-weight twisted module.
Define the {\it $1$-point correlation function} associated with $g,h$ to be 
\[
S^h_M(v,\tau):=\mathrm{tr}_M \left(o(v)\circ hq^{L_0-c/24}\right),
\]
where $\tau$ is a point on the complex upper-half plane.
Here, $o(v)$ denotes the {\it zero-mode} of $v$ defined by linearly extending
the assignment $o(u)=u(\mathrm{wt}(u)-1)$ for each conformal weight-homogeneous $u\in W$,
$L_0$ the zero-mode of the conformal vector $\omega\in W$, and $c=c_\mathcal{W}$ the central charge of $W$.

Define the {\it conformal block} associated with $g,h$ to be
\[
C(g,h;v):=\langle S_M^h(v,\tau)|M\ \mbox{is a}\ g\mbox{-weight twisted module}\rangle_\mathbb{C}.
\]

Then, by the result of \cite{DLM,E}, we have the following corollary.

\begin{cor}
Let $v$ be an element of $W$ and $g,h$ elements of $\mathbb{Z}_2$.
Then, the elements $\begin{pmatrix} a &b\\c&d\end{pmatrix}\in SL_2(\mathbb{Z})$ with the usual action induces the transformation
\[
\begin{pmatrix} a &b\\c&d\end{pmatrix}:C(g,h;v)\rightarrow C(g^ah^c,g^bh^d;v).
\]
\end{cor}

Thus, we have the modular invariance of the (twisted) modules of $W$.
In particular, we have the modular invariance $\begin{pmatrix} a &b\\c&d\end{pmatrix}:C(\mathrm{id},\mathrm{id};v)\rightarrow C(\mathrm{id},\mathrm{id};v)$
of the $1$-point correlation functions of the Ramond twisted representations.
\end{rem}

\begin{rem}
Suppose $\mathfrak{g}=E_8$.
Then, the Neveu-Schwarz twisted irreducible modules of $W$ are explicitly given by the following:
\[
M_0=V_1(E_7)\otimes L(-3/5,0)\oplus V_1(E_7;\varpi_7)\otimes L(-3/5,3/4),
\]
\[
M_1=V_1(E_7)\otimes L(-3/5,1/5)\oplus V_1(E_7;\varpi_7)\otimes L(-3/5,-1/20).
\]
Note that $M_0$ is the adjoint module of $W$.

The Ramond twisted irreducible modules are explicitly given by the following:
\[
M_2=V_1(E_7)\otimes L(-3/5,-1/20)\oplus V_1(E_7;\varpi_7)\otimes L(-3/5,1/5),
\]
\[
M_3=V_1(E_7)\otimes L(-3/5,4/3)\oplus V_1(E_7;\varpi_7)\otimes L(-3/5,0).
\]
By the above result,
the characters $\chi_2(\tau)=S_{M_2}^{\mathrm{id}}(|0\rangle;\tau)$ and $\chi_3(\tau)=S_{M_3}^{\mathrm{id}}(|0\rangle;\tau)$
span a $SL_2(\mathbb{Z})$-invariant vector space $C(\mathrm{id},\mathrm{id},|0\rangle)$.
Note that $\chi_2,\chi_3$ coincide with the characters of the {\it intermediate vertex subalgebra} $V_{E_{7+1/2}}$ and its module $V_{E_{7+1/2}+\alpha_1}$ \cite{Kaw1}.
By the result of \cite{Kaw1}, the characters $\chi_2$, $\chi_3$ form a basis of the solutions of the modular differential equation (\ref{eq:diff1}) with $\mu=-551/900$, which was the ``hole'' of the $2$-character rational conformal field theories
observed in \cite{MMS}.
See also \cite{T}.
Note that similar ``hole'' were observed in the study of the Deligne dimension formulas \cite{CdM,D} and filled in by using the
{\it intermediate Lie algebra} $E_{7+1/2}$ \cite{LaM3,W}.
\end{rem}

\begin{rem}
Let $\mathfrak{g}$ be a simple Lie algebra with the {\it Vogel parameter} $(\alpha,\beta,\gamma)\in \mathbb{P}_2$ \cite{V,LaM2,MV}.
Consider the {\it extended Vogel parameter} $(\alpha,\beta,\gamma,\kappa)\in\mathbb{P}_3$ \cite{MV}, and set
 $\kappa=-(\alpha+\beta)/2$.
Normalize the parameters as $\alpha=-2$ and set $k=\kappa$.
When $\mathfrak{g}$ is a Deligne exceptional Lie algebra, we have $k=-h^\vee/6$, and
when $\mathfrak{g}$ is of type $C_l$, we have $k=1/2$.
The branching rules of $\mathcal{W}_k(\mathfrak{g},f_\theta)$ for the other simple Lie algebras $\mathfrak{g}$ with $k=-(\alpha+\beta)/2$ will be considered in the forthcoming paper.
\end{rem}

\section{Appendix A. The abelian intertwining algebras, generalized vertex algebras and quasi generalized vertex algebras}
\label{sec:appenda}
\subsection{Abelian $3$-cocycles}

Let $Q$ be an abelian group.
Recall the {\it Eilenberg-Mac Lane abelian cohomology} \cite{EM}.
We use the $3$-cocycles of the cohomology.
Let $F:Q\times Q\times Q\rightarrow \mathbb{C}^\times$ and $\Omega:Q\times Q\rightarrow \mathbb{C}^\times$ be arbitrary functions.

\begin{dfn}
$(F,\Omega)$ is a {\it normalized abelian $3$-cocycle} (NA3) if
\begin{description}
\item [(A1)] $F(i,j,k)F(i,j,k+l)^{-1}F(i,j+k,l)F(i+j,k,l)^{-1}F(j,k,l)=1$,
\item [(A2)] $F(i,j,k)^{-1}\Omega(i,j+k)F(j,k,i)^{-1}=\Omega(i,j)F(j,i,k)^{-1}\Omega(i,k)$,
\item [(A3)] $F(i,j,k)\Omega(i+j,k)F(k,i,j)=\Omega(j,k)F(i,k,j)\Omega(i,k)$,
\item[(A4)] $F(i,j,0)=F(i,0,j)=F(0,i,j)=1$,
\item[(A5)] $\Omega(i,0)=\Omega(0,i)=1$,
\end{description}
for $i,j,k,l\in Q$.
\end{dfn}

Define $B:Q\times Q\times Q\rightarrow\mathbb{C}^\times$ to be
\[
B(i,j,k)=F(j,i,k)^{-1}\Omega(i,j)F(i,j,k)
\]
($i,j,k\in Q$).
Let $q:Q\rightarrow \mathbb{C}^\times$ ($q(i)=\Omega(i,i)$) denote the associated quadratic form \cite{EM}, and
 $\Delta:Q\times Q\rightarrow \mathbb{C}/\mathbb{Z}$ the unique symmetric bilinear map
such that
\[
e^{-2\pi i \Delta(i,j)}=q(i+j)q(i)^{-1}q(j)^{-1}=\Omega(i,j)\Omega(j,i).
\]

\begin{rem}\label{sec:lem1}
The map
\begin{eqnarray*}
f:\{\mbox{Normalized abelian 3-cocycles}\}&\rightarrow& \mathrm{Map}(Q\times Q\times Q,\mathbb{C}^\times),\\
 (F,\Omega)&\mapsto& B
\end{eqnarray*} 
is injective.
(This remark is shown in Section \ref{sec:computational}).

\subsection{Fields, locality and abelian intertwining algebras}

Let $U$ be a vector space.
We identify the subsets of $\mathbb{C}/\mathbb{Z}$ and the $\mathbb{Z}$-invariant subsets of $\mathbb{C}$.
Let $n$ be a positive integer.
Let $\Gamma(1),\ldots,\Gamma(n)$ be subsets of $\mathbb{C}/\mathbb{Z}$.
We denote by $U[[z_1,z_1^{\Gamma(1)},\ldots,z_n,z_n^{\Gamma(n)}]]$ the space of all formal infinite series 
\[
\sum_{k(1)\in \Gamma(1),\ldots,k(n)\in \Gamma(n)}f_{(k(1),\ldots,k(n))} z_1^{k(1)}\ldots z_n^{k(n)}
\]
 with $f_{(k(1),\ldots,k(n))}\in U$.
We denote by $U[[z_1,\ldots,z_n]]z_1^{\Gamma(1)}\cdots z_n^{\Gamma(n)}$ the space of all finite sums of the form 
\[
\sum_i \psi_i(z_1,\ldots,z_n)z_1^{d(1)_i}\cdots z_n^{d(n)_i}
\]
 with $d(1)_i\in\Gamma(1),\ldots,d(n)_i\in\Gamma(n)$ and
$\psi_i(z_1,\ldots,z_n)\in U[[z_1,\ldots,z_n]]$.

For $N\in \mathbb{C}$, we define the formal expansions
\begin{eqnarray*}
\iota_{z,w} (z+e^{-\pi i}w)^N &:=& e^{-w \partial_{z}} z^N \\
&=& \sum_{j\in\mathbb{Z}_+} \begin{pmatrix}N\\ j\end{pmatrix} z^{N-j} (-w)^j
 \in (\mathbb{C}[[z]]z^{N+\mathbb{Z}})[[w]],
\end{eqnarray*}
and
\begin{eqnarray*}
\iota_{w,z} (z+e^{\pi i}w)^N &:=& e^{\pi i N}e^{-z \partial_{w}} w^N \\
&=& e^{\pi i N}\sum_{j\in\mathbb{Z}_+} \begin{pmatrix}N\\ j\end{pmatrix} w^{N-j} (-z)^j
 \in (\mathbb{C}[[w]]w^{N+\mathbb{Z}})[[z]].
\end{eqnarray*}
We denote $\iota_{z,w}(z-w)^N=\iota_{z,w}(z+e^{-\pi i}w)^N$ and
$\iota_{w,z}(z-w)^N=\iota_{w,z}(z+e^{\pi i}w)^N$.
By using this, we define
\[
\iota_{z,w}:U[[z,w]]z^{\Gamma_1} w^{\Gamma_2} (z+e^{-\pi i}w)^N\rightarrow (U[[z]]z^{N+\Gamma_1})[[w]]w^{\Gamma_2},
\]
\[
\iota_{w,z}:U[[z,w]]z^{\Gamma_1} w^{\Gamma_2} (z+e^{\pi i}w)^N\rightarrow (U[[w]]w^{N+\Gamma_2})[[z]]z^{\Gamma_1},
\]
for $\Gamma_1,\Gamma_2\subset\mathbb{C}$ and $N\in \mathbb{C}$.

Let $Q$ be an abelian group and $(F,\Omega)$ a normalized abelian $3$-cocycle with maps $B$ and $\Delta$.
Let $V=\bigoplus_{\alpha\in Q}V^\alpha$ be a $Q$-graded vector space.

\begin{dfn}
 A {\it field} with {\it charge} $\alpha\in Q$ on $V$
is a formal series $a(z)\in (\mathrm{End} (V)) [[z,z^{\mathbb{C}}]]$
with the property that
\[
a(z)b\in V^{\alpha+\beta}[[z]]z^{-\Delta(\alpha+\beta)}
\]
for $b\in V^\beta$.
\end{dfn}

Let $a(z)$ and $b(z)$ be fields with charges $\alpha\in Q$ and $\beta\in Q$.

\begin{dfn}
The pair $(a,b)$ of fields
is called {\it local} if there exists $N\in \Delta(\alpha,\beta)$ such that for any $\gamma\in Q$ and $c\in V^\gamma$,
\begin{eqnarray*}
\iota_{z,w}(z-w)^N a(z)b(w)c
=B(\alpha,\beta,\gamma)
 \iota_{w,z}(z-w)^N b(w)a(z)c.
\end{eqnarray*}
We call such an $N$ a {\it locality bound} of $(a,b)$.
\end{dfn}

Now we define the abelian intertwining algebras \cite{DL} by using the locality axiom following the generalization of the generalized vertex algebras in \cite{BK}.

Let $V$ be a vector space, $|0\rangle \in V$ a non-zero vector, $\partial:V\rightarrow V$ an endomorphism and 
\[
Y:V\rightarrow \mathrm{End}(V)[[z,z^\mathbb{C}]],\quad a\mapsto Y(a,z)=\sum_{n\in \mathbb{C}} a{(n)} z^{-n-1}
\]
 a linear map.

\begin{dfn}
The quadruple $(V,Y,|0\rangle,\partial)$ is called an {\it abelian intertwining algebra} ({\it AIA})
if
\begin{enumerate}
\item {\it (vacuum axiom)} $\partial|0\rangle =0$, $Y(|0\rangle,z)=\mathrm{id}_V$,
$Y(a,z)|0\rangle\in V[[z]]$ and $Y(a,z)|0\rangle|_{z=0}=a$ ($a\in V$);
\item {\it (translation covariance)} $[\partial,Y(a,z)]=\partial_z Y(a,z)$ ($a\in V$).
\item {\it (field and locality axiom)} there exist 
\begin{enumerate}
\item an abelian group $Q$;
\item a normalized abelian $3$-cocycle $(F,\Omega)$ with the maps $B,\Delta$;
\item a $Q$-grading $V=\bigoplus_{\alpha\in Q}V^\alpha$ on $V$,
\end{enumerate}
 such that
\begin{enumerate}
\item[(d)] the vector $|0\rangle$ belongs to $V^0$;
\item[(e)] the operator $\partial$ is $Q$-grading preserving;
\item[(f)] for $\alpha\in Q$ and $a\in V^\alpha$, the formal series $Y(a,z)$ is a field with charge $\alpha$.
\item[(g)] the pair of fields $(Y(a,z),Y(b,w))$ is local ($\alpha,\beta\in Q$, $a\in V^\alpha$, $b\in V^\beta$).
\end{enumerate}
\end{enumerate}
\end{dfn}

We denote the AIAs by $(V,Y,|0\rangle,\partial)$ or $V$.
The linear map $Y$ is called the {\it state-field correspondence} or {\it vertex operator},
vector $|0\rangle$ the {\it vacuum vector}, and
operator $\partial$ the {\it translation operator} or {\it derivative}.

Let $(V,Y,|0\rangle,\partial)$ be an AIA with an abelian group $Q$, a $Q$-grading $V=\bigoplus_{\alpha\in Q}V^\alpha$ on $V$ and a normalized abelian $3$-cocycle $(F,\Omega)$ satisfying the axioms.
We call the pair $((V,Y,|0\rangle,\partial),(Q,F,\Omega))$ with the grading $V=\bigoplus_{\alpha\in Q}V^\alpha$ a {\it $Q$-charged abelian intertwining algebra ($Q$-charged AIA)}.
We denote it by $(V,(Q,F,\Omega))$ or $V$.

Let $(V,Y,|0\rangle,\partial)$ be an AIA.
We denote the transposed operator of $Y$ by $Y:V\times V\rightarrow V[[z]]z^\mathbb{C}$, $(u,v)\mapsto Y(u,z)v$.
Note that the operator induces $Y:V^\alpha\times V^\beta\rightarrow V^{\alpha+\beta}[[z]]z^{-\Delta(\alpha,\beta)}$
for $\alpha,\beta\in Q$.
Let $V$ be a $Q$-charged AIA.
For $X\subset Q$, we denote $V^X:=\bigoplus_{\alpha\in X}V^\alpha$.

Let $(V_1,Y_1,|0\rangle_1,\partial_1),(V_2,Y_2,|0\rangle_2,\partial_2)$ be AIAs.

\begin{dfn}
A {\it homomorphism} of AIAs is a linear map $f:V_1\rightarrow V_2$
which satisfies
\begin{enumerate}
\item $f(Y_1(a,z)b)=Y_2(f(a),z)f(b)$ ($a,b\in V_1$);
\item $f(|0\rangle_1)=|0\rangle_2$;
\item $f\circ \partial_1=\partial_2\circ f$.
\end{enumerate}
\end{dfn}

General theorems for the AIAs are given in Section \ref{sec:aia}.

\subsection{Vertex algebras, generalized vertex algebras and quasi generalized vertex algebras}

We consider three subclasses of the AIAs.
Let $V$ be a vector space, $|0\rangle \in V$ a non-zero vector, $\partial:V\rightarrow V$ an endomorphism, and 
\[
Y:V\rightarrow \mathrm{End}(V)[[z,z^\mathbb{C}]],\quad a\mapsto Y(a,z)=\sum_{n\in \mathbb{C}} a{(n)} z^{-n-1}
\]
 a linear map.

\begin{dfn}
The quadruple $(V,Y,|0\rangle,\partial)$ is called a {\it vertex algebra} if 
\begin{enumerate}
\item $Y(|0\rangle,z)=\mathrm{id}_V$, $Y(a,z)|0\rangle\in V[[z]]$, $a=\mathrm{Res}_{z=0} Y(a,z)|0\rangle$;
\item $\partial|0\rangle=0$, $[\partial,Y(a,z)]=\partial_zY(a,z)$ ($a\in V$);
\item  $Y(a,z)b\in V[[z]]z^{\mathbb{Z}}(=V((z)))$ ($a,b\in V$);
\item for any $a,b\in V$, there exists $N\in \mathbb{Z}$ such that 
\[
\iota_{z,w}(z-w)^N Y(a,z)Y(b,w)=\iota_{w,z}(z-w)^N Y(b,w)Y(a,z).
\]
\end{enumerate}
\end{dfn}

\begin{dfn}\cite{DL,BK}
The quadruple $(V,Y,|0\rangle,\partial)$ is called a {\it generalized vertex algebra (GVA)} if 
\begin{enumerate}
\item $Y(|0\rangle,z)=\mathrm{id}_V$, $Y(a,z)|0\rangle\in V[[z]]$, $a=\mathrm{Res}_{z=0} Y(a,z)|0\rangle$;
\item $\partial|0\rangle=0$, $[\partial,Y(a,z)]=\partial_zY(a,z)$;
\item there exist an abelian group $Q$, a $Q$-grading $V=\bigoplus_{\alpha\in Q}V^\alpha$ and  a function $\eta:Q\times Q\rightarrow\mathbb{C}^\times$
such that for any $\alpha,\beta\in Q$, $a\in V^\alpha$, $b\in V^\beta$,
\begin{enumerate}
\item $\eta$ is bimultiplicative;
\item $Y(a,z)b\in V^{\alpha+\beta}[[z]]z^{-\Delta(\alpha,\beta)}$;
\item there exists $N\in \Delta(\alpha,\beta)$ such that 
\[
\iota_{z,w}(z-w)^N Y(a,z)Y(b,w)=\eta(\alpha,\beta)\iota_{w,z}(z-w)^N Y(b,w)Y(a,z);
\]
\item $|0\rangle\in V^0$, $\partial(V^\alpha)\subset V^\alpha$.
\end{enumerate}
Here, $\Delta:Q\times Q\rightarrow \mathbb{C}/\mathbb{Z}$ is the bilinear map defined to be
$e^{-2\pi i\Delta(\alpha,\beta)}=\eta(\alpha,\beta)\eta(\beta,\alpha)$.
\end{enumerate}
\end{dfn}

\begin{dfn}
The quadruple $(V,Y,|0\rangle,\partial)$ is a {\it quasi generalized vertex algebra (quasi-GVA)} if

\begin{enumerate}
\item $Y(|0\rangle,z)=\mathrm{id}_V$, $Y(a,z)|0\rangle\in V[[z]]$, $a=\mathrm{Res}_{z=0} Y(a,z)|0\rangle$;
\item $\partial|0\rangle=0$, $[\partial,Y(a,z)]=\partial_zY(a,z)$;
\item there exist an abelian group $Q$, a $Q$-grading $V=\bigoplus_{\alpha\in Q}V^\alpha$
and a function $\eta:Q\times Q\rightarrow\mathbb{C}^\times$
such that for any $\alpha,\beta\in Q$, $a\in V^\alpha$, $b\in V^\beta$,
\begin{enumerate}
\item $\mu_\eta(\alpha,\beta,\gamma):=\eta(\alpha,\beta,\gamma)^{-1}\eta(\alpha,\gamma)\eta(\beta,\gamma)$ is multiplicative in $\gamma$,
\item there exists the bilinear map $\Delta:Q\times Q\rightarrow \mathbb{C}/\mathbb{Z}$ such that $\eta(\alpha,\beta)\eta(\beta,\alpha)=e^{-2\pi i \Delta(\alpha,\beta)}$,
\item $\eta(Q,0)=\eta(0,Q)=1$,
\item $Y(a,z)b\in V^{\alpha+\beta}[[z]]z^{-\Delta(\alpha,\beta)}$;
\item there exists $N\in \Delta(\alpha,\beta)$ such that 
\[
\iota_{z,w}(z-w)^N Y(a,z)Y(b,w)=\eta(\alpha,\beta)\iota_{w,z}(z-w)^N Y(b,w)Y(a,z);
\]
\item $|0\rangle\in V^0$, $\partial(V^\alpha)\subset V^\alpha$.
\end{enumerate}
\end{enumerate}
\end{dfn}

We denote the vertex algebras, GVAs and quasi-GVAs by $(V,Y,|0\rangle,\partial)$ or $V$.

Let $(V,Y,|0\rangle,\partial)$ be a GVA (resp., quasi-GVA) with an abelian group $Q$, a $Q$-grading $V=\bigoplus_{\alpha\in Q}V^\alpha$ on $V$ and a function $\eta$ satisfying the axioms.
We call the pair $((V,Y,|0\rangle,\partial),(Q,\eta))$ with the grading $V=\bigoplus_{\alpha\in Q}V^\alpha$ a {\it $Q$-charged GVA} (resp., {\it $Q$-charged quasi-GVA}).

Note that a $Q$-charged quasi-GVA $(V,(Q,\eta))$ is a $Q$-charged GVA if and only if $\eta$ is bimultiplicative.

Note that any GVA is a quasi-GVA, and
by Lemma \ref{sec:lem3}, any quasi-GVA is an AIA.
By Lemma \ref{sec:lem3} and Lemma \ref{sec:lem4}, we have the following lemma.
 Let $(V,(Q,F,\Omega))$ be a $Q$-charged AIA. 
Set $\eta=\Omega$.

\begin{lem}
\begin{enumerate}\label{sec:lemz2}
\item   The pair $(V,(Q,\eta))$ is a $Q$-charged GVA if  $F(\alpha,\beta,\gamma)=1$ for any $\alpha,\beta,\gamma\in Q$.
\item The pair $(V,(Q,\eta))$ is a $Q$-charged quasi-GVA if  $F(\alpha,\beta,\gamma)=F(\beta,\alpha,\gamma)$ for any $\alpha,\beta,\gamma\in Q$. Moreover, then, $\mu_{\eta}=F$.
\item If $Q=\mathbb{Z}_2$, then $(V,(\mathbb{Z}_2,\eta))$ is a $\mathbb{Z}_2$-charged quasi-GVA with $\mu_\eta=F$.
\end{enumerate}
\end{lem}

Let $(V,(Q,\eta))$ be a $Q$-charged quasi-GVA.
If $\eta$ is identically $1$, the quasi-GVA $V$ is a vertex algebra.

\subsection{Modification of quasi-GVAs}

Let $Q$ be an abelian group.
Let $\varepsilon:Q\times Q\rightarrow \mathbb{C}^\times$ be a function.
Let $(f,\omega)$ denote the abelian group cohomology coboundary of $\varepsilon$, that is,
\[
f(\alpha,\beta,\gamma)=\varepsilon(\alpha,\beta+\gamma)\varepsilon(\beta,\gamma)\varepsilon(\alpha+\beta,\gamma)^{-1}\varepsilon(\alpha,\beta)^{-1}, \quad \alpha,\beta,\gamma\in Q,
\]
\[
 \omega(\alpha,\beta)=\varepsilon(\alpha,\beta)\varepsilon(\beta,\alpha)^{-1}, \quad \alpha,\beta\in Q.
\]

\begin{dfn}
A {\it quasi $2$-cocycle} of $Q$ is a function
 $\varepsilon:Q\times Q\rightarrow \mathbb{C}^\times$  such that
\begin{enumerate}
\item $\varepsilon(0,Q)=\varepsilon(Q,0)=1$;
\item $f(\alpha,\beta,\gamma)=f(\beta,\alpha,\gamma)$ for any $\alpha,\beta,\gamma\in Q$.
\end{enumerate}
\end{dfn}

Note that when $\varepsilon$ is a quasi $2$-cocycle, the coboundary $(f,\omega)$ is a
normalized abelian $3$-cocycle.

\begin{rem}
A $2$-cocycle $\varepsilon$ of $Q$ in terms of the usual group cohomology
satisfies $f(\alpha,\beta,\gamma)=1$ for any $\alpha,\beta,\gamma\in Q$, so is a quasi $2$-cocycle.
\end{rem}

Let $((V,Y,|0\rangle,\partial),(Q,\eta))$ be a $Q$-charged quasi-GVA.
Let $\varepsilon:Q\times Q\rightarrow \mathbb{C}^\times$ be a quasi $2$-cocycle.

Define $Y^\varepsilon:V\times V\rightarrow \mathrm{End}(V)[[z,z^{\mathbb{C}}]]$ by
linearly extending the assignment
\[
Y^\varepsilon(v,z)w=\varepsilon(\alpha,\beta)Y(v,z)w
\]
for $v\in V^\alpha$, $w\in V^\beta$ ($\alpha,\beta\in Q$).

Define $\eta^\varepsilon:Q\times Q\rightarrow \mathbb{C}^\times$ to be
\[
\eta^\varepsilon(\alpha,\beta)=\omega(\alpha,\beta)\cdot \eta(\alpha,\beta) \quad \alpha,\beta\in Q.
\]

\begin{prop}
The pair $((V,Y^\varepsilon,|0\rangle,\partial),(Q,\eta^\varepsilon))$ is a $Q$-charged quasi-GVA.
\end{prop}

\begin{proof}
It suffices to show the locality axiom.
Let $\alpha,\beta$ be elements of $Q$ and $a,b$ elements of $V^\alpha,V^\beta$.
Let $N\in \Delta(\alpha,\beta)$ be a locality bound of $(Y(a,z),Y(b,w))$.
Let $c$ be an element of $V^\gamma$ with $\gamma\in Q$.
Since $f(\alpha,\beta,\gamma)=f(\beta,\alpha,\gamma)$, we have
$\omega(\alpha,\beta)=\varepsilon(\alpha,\beta+\gamma)\varepsilon(\beta,\gamma) \varepsilon(\beta,\alpha+\gamma)^{-1}\varepsilon(\alpha,\gamma)^{-1}$. 
Therefore, we have $\iota_{z,w}(z-w)^N Y^\varepsilon(a,z)Y^\varepsilon(b,w)c-\eta^\varepsilon(\alpha,\beta)\iota_{w,z}(z-w)^N Y^\varepsilon(b,w)Y^\varepsilon(a,z)c=0$, as desired.
\end{proof}

We call the quasi-GVA $(V,Y^\varepsilon,|0\rangle,\partial)$ the {\it $\varepsilon$-modified quasi-GVA}
and denote it by $V^\varepsilon$.

\begin{rem}
Note that it is a special case of the $\varepsilon$-modify (twist) of the AIAs by an abelian $2$-cochain $\varepsilon$ \cite{DL,Car}.
When $V$ is a GVA and $\varepsilon$ a $2$-cocycle, the $\varepsilon$-modified quasi-GVA $V^\varepsilon$ is a GVA and coincides with the $\varepsilon$-modified GVA \cite{BK}.
\end{rem}

\subsection{Weight gradings and vertex operator algebras}

Let $V$ be an abelian intertwining algebra.
A $\mathbb{C}$-grading $V=\bigoplus_{n\in\mathbb{C}}V_n$ is called a {\it weight grading}
if 
\[
a(n)b\in V_{k+l-n-1}, \quad a\in V_k,\ b\in V_l,\ k,l,n\in \mathbb{C}.
\]

Let $V=\bigoplus_{\alpha\in Q}V^\alpha$ be a $Q$-charged abelian intertwining algebra.
A non-zero vector $\omega\in V^0$ is called a {\it Virasoro vector} of central charge $c$ if
the field $Y(\omega,z)=\sum_{n\in\mathbb{Z}}L_n z^{-n-2}$ defines the module structure on $V$ over the Virasoro algebra $\mathrm{Vir}=\bigoplus_{n\in\mathbb{Z}}\mathbb{C}L_n\oplus \mathbb{C}C$ with the central charge $c$, that is,
the OPE ($\lambda$-bracket) of $\omega$ and itself has the form
\[
[\omega_\lambda \omega]=\partial \omega+2\lambda \omega+\frac{c}{12}\lambda^3|0\rangle.
\]
Here, for $a\in V^0$ and $b\in V$, the $\lambda$-bracket is defined to be $[a_\lambda b]=\sum_{n=0}^\infty \lambda^n a(n)b/n!$.
We call $c$ the {\it central charge} of $\omega$.

Let $(V=\bigoplus_{\alpha\in Q}V^\alpha,Y,|0\rangle,\partial)$ be a $Q$-charged abelian intertwining algebra with the weight grading
$V=\bigoplus_{n\in\mathbb{C}}V_n$.
A Virasoro vector $\omega\in V^0$ with the field $Y(\omega,z)=\sum_{n\in \mathbb{Z}}L_nz^{-n-2}$ is called a {\it conformal vector} if $L_0 v=n v$ ($v\in V_n$, $n\in\mathbb{C}$)
and $L_{-1}=\partial$.
Each conformal vector $\omega$ belongs to $V_2$.

Let $V$ be an abelian intertwining algebra equipped with a conformal vector $\omega$.
A vector $v\in V$ is called a {\it primary vector} of conformal weight $n$ if $L_mv=0$ for $m\in\mathbb{Z}_{>0}$ and
$L_0v=nv$.

\begin{dfn}
Let $V$ be a vertex algebra with the weight grading $V=\bigoplus_{n\in \mathbb{C}}V_n$.
Let $\omega\in V_2$ be a conformal vector of central charge $c$.
The pair $(V,\omega)$ is called a {\it vertex operator algebra} of central charge $c$.
\end{dfn}

Let $\Gamma$ be a subset of $\mathbb{C}$.
We call a VOA $V$ {\it $\Gamma$-graded} if $V_n=0$ for each $n\in\mathbb{C}\setminus \Gamma$.

A $\mathbb{Q}_+$-graded VOA $V$ is called  {\it of CFT-type} if $V_0=\mathbb{C}|0\rangle$.

Let $V$ and $W$ be VOAs with conformal vectors $\omega$ and $\tau$.
A vertex algebra homomorphism $f:V\rightarrow W$ is called a {\it vertex operator algebra homomorphism}
if $f(\omega)=\tau$.

Let $V$ be a VOA with a confomal vector $\omega$.
Let $u,v\in V$ be primary vectors of conformal weight $3/2$.
Then, by the commutation relation (\ref{eqn:comm}) of the vertex algebras and $L_n=\omega(n+1)$ ($n\in\mathbb{Z}$),
\begin{equation}\label{eqn:omega}
L_1(u(0)v)=u(1)v,\quad \mbox{and}\quad L_2(u(0)v)=\frac{3}{2}u(2)v.
\end{equation}

Let $V$ be a VOA.
Set $C_2(V)=\{a(n)b\in V|a,b\in V, n\leq -2\}$.
A VOA $V$ is called {\it $C_2$-cofinite} if $V/C_2(V)$ is finite dimensional.

\subsection{Tensor products and graded tensor products}

Let $i$ be an element of $\{1,2\}$.
Let $Q_i$ be an abelian group with the normalized abelian $3$-cocycle 
$(F_i,\Omega_i)$.
Let $(V_i,Y_i,|0\rangle_i,\partial_i)$ be a $Q_i$-charged abelian intertwining algebra with $(F_i,\Omega_i)$.

The {\it tensor product abelian intertwining algebra} $V_1\otimes V_2$ of $V_1,V_2$ is
the $Q_1\times Q_2$-charged abelian intertwining algebra
\[
(V_1\otimes V_2, Y_1\otimes Y_2,|0\rangle_1\otimes |0\rangle_2,\partial_1\otimes \mathrm{id}_{V_1}+\mathrm{id}_{V_2}\otimes \partial_2)
\]
with the normalized abelian $3$-cocycle $(F_1\times F_2,\Omega_1\times \Omega_2)$.

Suppose $Q=Q_1=Q_2$.
Then, the abelian intertwining subalgebra
\[
\bigoplus_{\alpha\in Q}V_1^\alpha\otimes V_2^\alpha\subset V_1\otimes V_2
\]
is called the {\it graded tensor product} of $V_1,V_2$.
The graded tensor product is a $Q$-charged abelian intertwining algebra with the normalized abelian $3$-cocycle $(F_1F_2,\Omega_1\Omega_2)$.
Here, $(F_1F_2)(\alpha,\beta,\gamma)=F_1(\alpha,\beta,\gamma)F_2(\alpha,\beta,\gamma)$, and $(\Omega_1\Omega_2)(\alpha,\beta)=\Omega_1(\alpha,\beta)\Omega_2(\alpha,\beta)$ ($\alpha,\beta,\gamma\in Q$).

Suppose $Q_1=Q_2=\mathbb{Z}_2=\{0,1\}$.
Then, $V_i$ is a $\mathbb{Z}_2$-charged quasi-GVA with $\eta_i=\Omega_i$ ($i=1,2$).

\begin{lem}\label{sec:lemgraded}
If $\Omega_1(1,1)\Omega_2(1,1)=1$, then the graded tensor product
$V_1^0\otimes V_2^0\oplus V_1^1\otimes V_2^1$ is a vertex algebra.
\end{lem}

\begin{proof}
Since $\Omega_i(0,0)=\Omega_i(1,0)=\Omega_i(0,1)=1$, we have the lemma.
\end{proof}

Suppose that $V_i$ equips with the conformal vector $\omega_i$ of central charge $c_i$ ($i=1,2$).
Unless otherwise noted, we equip the tensor product and graded tensor product of $V_1$ and $V_2$ with the conformal vector
$\omega=\omega_1\otimes |0\rangle_2+|0\rangle_1\otimes \omega_2$ of central charge $c_1+c_2$, as usual.
Note that for each primary vector $u_i\in V_i$ of conformal weight $n_i$ ($i=1,2$), the vector $u_1\otimes u_2$
is primary of conformal weight $n_1+n_2$.

\section{Appendix B. Simple current extensions}
\label{sec:appendb}
In this section, we consider extended abelian intertwining algebras.
See \cite{DL,Car2,Car}.

\subsection{$\mathbb{Z}_2$ simple current extensions}

Let $V$ be a VOA with a simple current irreducible $V$-module $M$ such that $M\boxtimes M=V$.
Here, $\boxtimes$ denotes the fusion product.
Let $Y:V\times V\rightarrow V[[z]]z^\mathbb{Z}$ denote the vertex operator of $V$ and $Y_M:V\times M\rightarrow M[[z]]z^\mathbb{Z}$ the module structure over $M$.
Let $Y_M^*:M\times V\rightarrow M[[z]]z^\mathbb{Z}$ denote the transpose of $Y_M$ defined to be
\[
Y_M^*(u,z)a=e^{zL_{-1}}Y_M(a,-z)u,\quad u\in M,\ a\in V.
\]
Let  $I:M\times M\rightarrow V[[z]]z^\mathbb{C}$ be a non-zero intertwining operator.
We have
\begin{eqnarray}\label{eqn:fixedvector1}
z^t I(v,z)v\in V[[z]],
\end{eqnarray}
\begin{eqnarray}\label{eqn:fixedvector2}
z^tI(v,z)v|_{z=0}\neq 0
\end{eqnarray}
with some $v\in M$ and $t\in \mathbb{C}$.
Then, $I(u,z)w\in V[[z]]z^{-t+\mathbb{Z}}$ for any $u,w\in M$,
since $M$ is an irreducible $V$-module.

Suppose that $\bar{V}=V\oplus M$ is a $\mathbb{Z}_2$-graded AIA with the vertex operator
 $\bar{Y}:\bar{V}\times \bar{V}\rightarrow \bar{V}[[z]]z^{\mathbb{C}}$ defined to be
\[
\bar{Y}(a+u,z)(b+w)=Y(a,z)b+Y_M(a,z)w+Y_M^*(u,z)b+I(u,z)w,
\]
for $a,b\in V$ and $u,w\in M$, the $\mathbb{Z}_2$-gradation $\bar{V}=\bar{V}_0\oplus \bar{U}_1$ with $\bar{V}_0=V$ and $\bar{V}_1=M$ and certain normalized abelian $3$-cocycle $(F,\Omega)$ with $\Delta(1,1)=t+\mathbb{Z}$.
By Lemma \ref{sec:lemz2} (3), the pair $(\bar{V},(\mathbb{Z}_2,\Omega))$ is a $\mathbb{Z}_2$-charged quasi-GVA.
We call it
the {\it $\mathbb{Z}_2$ simple current extension} of $V$.

\begin{prop}\label{sec:propvir}
$\Omega(0,0)=\Omega(0,1)=\Omega(1,0)=1$, $\Omega(1,1)=e^{-\pi i t}$.
\end{prop}

\begin{proof}
By Lemma \ref{sec:skewsymmetry} (skew-symmetry) and eq.\,(\ref{eqn:fixedvector1}),
\[
z^t I(v,z)v|_{z=0}=z^t \Omega(1,1)e^{zL_{-1}}I(v,e^{-\pi i}z)v|_{z=0}.
\]
By eq.\,(\ref{eqn:fixedvector1}) again,
The RHS is equal to $\Omega(1,1)e^{-\pi i (-t)}z^t I(v,z)v$.
By eq.\,(\ref{eqn:fixedvector2}), we have $\Omega(1,1)=e^{-\pi i t}$, which completes the proof.
\end{proof}

\subsection{$\mathbb{Z}_2$ simple current extension of the $\mathcal{M}(3,p)$ Virasoro minimal models}
\label{sec:vir}

Let $p$ be a positive integer such that $p>3$ and $(3,p)=1$.
Set $c=c_{3,p}=1-6(3-p)^2/(3p)$.
Consider the stress energy tensor $T(z)=\sum_{n\in\mathbb{Z}}L_n z^{-n-2}$.
Let $U$ denote the Virasoro minimal model $L(c,0)=\mathcal{M}(3,p)$.
Then, $U$ has the VOA structure
generated by
\[
Y(L_{-2}|0\rangle,z)=T(z)=\sum_{n\in\mathbb{Z}}L_n z^{-n-2}.
\]

Set $h=(p-2)/4$.
The irreducible module 
\[
M=L(c,h)=U\cdot |h\rangle
\]
 is a simple current of the VOA $U$.
Here, $|h\rangle$ is a highest weight vector of central charge $c$ and conformal weight $h$,
and $U\cdot|h\rangle$ denote the irreducible cyclic module.
Put $v=|h\rangle$.
Let $Y_M:U\times M\rightarrow M((z))$ denote the module structure of $M$ over $U$ with 
\[
Y_M(L_{-2}|0\rangle,z)=\sum_{n\in\mathbb{Z}}L_n z^{-n-2}.
\]
Let $Y_M^*:M\times U\rightarrow M((z))$ denote the transpose of $Y_M$ defined by
\[
Y_M^*(u,z)a=e^{zL_{-1}}Y_M(a,-z)u,
\]
$a\in U$, $u\in M$.

Set $\Gamma=2h+\mathbb{Z}$.
Let $I:M\times M\rightarrow U[[z]]z^{-\Gamma}$ denote the intertwining operator
of type $\begin{pmatrix}\quad U\quad\\M\ M\end{pmatrix}$ normalized as
$z^{2h}I(v,z)v|_{z=0}=|0\rangle$.
Then,
\[
I(|h\rangle,z)|h\rangle= |0\rangle z^{-2h}+0+\frac{2h}{c}L_{-2}|0\rangle z^{-2h+2}+\cdots.
\]

The extended vertex operator $\bar{Y}=Y+Y_M+Y_M^*+I$ is local with itself (cf.\ \cite{FJM}).
Therefore the pair $(\bar{U},(\mathbb{Z}_2,\Omega))$ of $U$ with $\bar{U}=U\oplus M$ is the $\mathbb{Z}_2$ simple current extension of $U$.
By Proposition \ref{sec:propvir}, we have $\Omega(0,0)=\Omega(0,1)=\Omega(1,0)=1$ and $\Omega(1,1)=e^{-2\pi i h}$.

Note that the extended algebras of $\mathcal{M}(3,p)$ are also considered in \cite{JM,MR} with another algebraic structure.

\subsection{The lattice generalized vertex algebra associated with $A_1^\circ$ and the simple current extension of $V_1(A_1)$}

In this section, we consider the simple current extension $V_1(A_1)\oplus V_1(A_1;\alpha/2)$ as the modification of the lattice GVA $V_{A_1^\circ}$, in order to obtain the vertex operators explicitly.

Let $A_1=\mathbb{Z}\alpha$ be the root lattice of type $A_1$ with 
the bilinear form defined by $(\alpha|\alpha)=2$.
Let $A_1^\circ=\mathbb{Z}\alpha/2$ be the dual lattice of $A_1$.

Consider the lattice GVA $F=V_{A_1^\circ}=M(1)\otimes \mathbb{C}[A_1^\circ]$ with the vertex operator ($\beta\in A_1^\circ$)
\[
X(e^\beta,z)=\mathrm{exp}\left(\sum_{n<0}\frac{1}{-n}\beta_{n}z^{-n}\right)\mathrm{exp}\left(\sum_{n>0}\frac{1}{-n}\beta_{n}z^{-n}\right)\otimes e^\beta z^\beta.
\]
The pair $(F,(\mathbb{Z},\eta))$ is a $\mathbb{Z}$-charged GVA with
 the $\mathbb{Z}$-grading $F=\bigoplus_{n\in\mathbb{Z}}F^{n}$ with $F^{n}\cong M(1)\otimes e^{n\alpha/2}$ and $\eta:\mathbb{Z}\times \mathbb{Z}\rightarrow \mathbb{C}^\times$ defined by $\eta(k,l)=e^{\pi i (k\alpha/2|l\alpha/2)}=e^{\pi i kl/2}$ (cf. \cite{BK,DL}).

Note that the subGVA $F^{2\mathbb{Z}}$ is isomorphic to the VOA $V_{A_1}$, and as a $V_{A_1}$-submodule, 
the subspace $F^{2\mathbb{Z}+1}$ is isomorphic to the $V_{A_1}$-module $V_{A_1+\alpha/2}$.

Actually,  since $\eta(k+4n,l+4m)=\eta(k,l)$ for any $k,l,n,m\in\mathbb{Z}$, the bimultiplicative function $\eta:\mathbb{Z}_4\times \mathbb{Z}_4\rightarrow \mathbb{C}^\times$ is induced, and
 the pair $(F,(\mathbb{Z}_4,\eta))$ is a $\mathbb{Z}_4$-charged GVA with the $\mathbb{Z}_4$-grading
$F=\bigoplus_{n\in \mathbb{Z}_4} F^{n+4\mathbb{Z}}$.

Now we modify the $\mathbb{Z}_4$-charged GVA $F$ with certain quasi $2$-cocycle $\varepsilon:\mathbb{Z}_4\times \mathbb{Z}_4\rightarrow \mathbb{C}^\times$ such that
the $\varepsilon$-modified quasi-GVA $F^\varepsilon$ is isomorphic to the simple current extension of $V_{A_1}$.

Define the function $\varepsilon:\mathbb{Z}_4\times  \mathbb{Z}_4\rightarrow \mathbb{C}^\times$ to be
\begin{eqnarray*}
\varepsilon(k,l)=\begin{cases}-1 & (k,l)=(1,2),(2,2),(2,3)\ \mbox{and}\ (3,1),\\ 1& \mbox{otherwise,}\end{cases}
\end{eqnarray*}
for $k,l\in \mathbb{Z}_4$.
Therefore, $\varepsilon(0,k)=1$, $\varepsilon(1,k)=(-1)^{\delta_{2,k}}$,
 $\varepsilon(2,k)=(-1)^{\delta_{2,k}+\delta_{3,k}}$, and $\varepsilon(3,k)=(-1)^{\delta_{1,k}}$ for $k\in \mathbb{Z}_4$.
Let $(f,\omega)$ denote the abelian group cohomology coboundary of $\varepsilon$.
Then, 
\[
\omega(k,l)=(-1)^{kl(kl-1)/2}
\] for $k,l\in Q$.

\begin{lem}
The function $\varepsilon$ is a quasi $2$-cocycle of $Q=\mathbb{Z}_4$.
\end{lem}

\begin{proof}
We have $\varepsilon(0,Q)=\varepsilon(Q,0)=1$.
Let $k,l,m$ be elements of $Q$.
We show $f(k,l,m)=f(l,k,m)$.
When $k=0$ or $l=0$, we have $f(k,l,m)=1=f(l,k,m)$.
Suppose $(k,l)=(1,2)$.
Then, 
\begin{eqnarray*}
\frac{f(1,2,m)}{f(2,1,m)}&=&\frac{\varepsilon(1,2+m)\varepsilon(2,m)\varepsilon(3,m)^{-1}\varepsilon(1,2)^{-1}}{\varepsilon(2,1+m)\varepsilon(1,m)\varepsilon(3,m)^{-1}\varepsilon(2,1)^{-1}}\\
&=&(-1)^{\delta_{0,m}+\delta_{2,m}+\delta_{3,m}-\delta_{1,m}-1-\delta_{1,m}-\delta_{2,m}-\delta_{2,m}+\delta_{1,m}}\\
&=&(-1)^{\delta_{0,m}+\delta_{2,m}+\delta_{3,m}-\delta_{1,m}-1}=1.
\end{eqnarray*}
Suppose $(k,l)=(3,2)$.
Then, 
\begin{eqnarray*}
\frac{f(3,2,m)}{f(2,3,m)}&=&\frac{\varepsilon(3,2+m)\varepsilon(2,m)\varepsilon(1,m)^{-1}\varepsilon(3,2)^{-1}}{\varepsilon(2,3+m)\varepsilon(3,m)\varepsilon(1,m)^{-1}\varepsilon(2,3)^{-1}}\\
&=&(-1)^{\delta_{3,m}+\delta_{2,m}+\delta_{3,m}-\delta_{2,m}-\delta_{0,m}-\delta_{3,m}-\delta_{1,m}+\delta_{2,m}-1}\\
&=&(-1)^{\delta_{3,m}+\delta_{2,m}-\delta_{0,m}-\delta_{1,m}-1}=1.
\end{eqnarray*}
Suppose $(k,l)=(1,3)$.
Then,
\begin{eqnarray*} 
\frac{f(1,3,m)}{f(3,1,m)}&=&\frac{\varepsilon(1,3+m)\varepsilon(3,m)\varepsilon(0,m)^{-1}\varepsilon(1,3)^{-1}}{\varepsilon(3,1+m)\varepsilon(1,m)\varepsilon(0,m)^{-1}\varepsilon(3,1)^{-1}}\\
&=&(-1)^{\delta_{3,m}+\delta_{1,m}-\delta_{0,m}-\delta_{2,m}-1}=1.
\end{eqnarray*}
Thus, $f(k,l,m)=f(l,k,m)$ for any $k,l,m\in Q$, which completes the proof.
\end{proof}

Consider the $\varepsilon$-modified quasi-GVA $F^\varepsilon$.
We have $\eta^\varepsilon(k,l)=\omega(k,l)\eta(k,l)=e^{\pi i k^2l^2/2}$.
Therefore,
\begin{eqnarray*}
\eta^\varepsilon(k,l)=\begin{cases} e^{\pi i/2}&\mbox{if}\ k,l\ \mbox{are odd},\\ 1 & \mbox{otherwise},\end{cases}
\end{eqnarray*}
for $k,l\in\mathbb{Z}_4$.
Hence, $\eta^\varepsilon$ induces the function $\eta^\varepsilon:\mathbb{Z}_2\times\mathbb{Z}_2\rightarrow \mathbb{C}^\times$,
and the pair $(F^\varepsilon,(\mathbb{Z}_2,\eta^\varepsilon))$ is a $\mathbb{Z}_2$-charged quasi-GVA
 with the $\mathbb{Z}_2$-grading $F^\varepsilon=(F^\varepsilon)^{0+2\mathbb{Z}}\oplus (F^\varepsilon)^{1+2\mathbb{Z}}$.

Thus, we have the following proposition (cf.\ \cite{DL}).

\begin{prop}\label{sec:propsce}
The $\varepsilon$-modified quasi-GVA $F^\varepsilon$ is isomorphic to the simple current extension $V_{A_1}\oplus V_{A_1+\alpha/2}$
of the lattice vertex algebra $V_{A_1}$,
thus isomophic to the simple current extension $\bar{V}=V_1(A_1)\oplus V_1(A_1;\alpha/2)$ of the simple level one affine vertex algebra
$V_1(A_1)$.
\end{prop}

\section{Appendix C. Some vertex operator algebras and generalized vertex algebras}\label{sec:appendc}

\subsection{The affine vertex operator algebras}
\label{sec:affine}

In this section, we recall the affine vertex operator algebras (cf.\ \cite{Kac}).

Let $\mathfrak{g}$ be a rank $l$ finite dimensional simple Lie algebra.
Let $\mathfrak{h}$ be a Cartan subalgebra of $\mathfrak{g}$ with the simple roots $\alpha_1,\ldots,\alpha_l$.
Consider the normalized invariant bilinear form $(\cdot|\cdot):\mathfrak{g}\times \mathfrak{g}\rightarrow \mathbb{C}$ such that $(\alpha^\vee|\alpha^\vee)=2$ for each long root $\alpha$.
Consider the affine Kac-Moody Lie algebra $\widehat{\mathfrak{g}}=\mathfrak{g}\otimes \mathbb{C}[t,t^{-1}]\oplus \mathbb{C}K\oplus \mathbb{C}D$
with the central element $K$, degree operator $D$, Lie bracket $[D,a\otimes t^n]=na\otimes t^n$ and 
$[a\otimes t^n,b\otimes t^m]=[a,b]t^{n+m}+n\delta_{n+m,0}(a|b)K$ and
simple roots $\alpha_0,\alpha_1,\ldots,\alpha_l$ and fundamental weights $\Lambda_0,\Lambda_1,\ldots,\Lambda_l$.
Set $a(n)=a\otimes t^n$ ($a\in \mathfrak{g}$, $n\in\mathbb{Z}$).
Set $\widehat{\mathfrak{g}}_-=\mathfrak{g}\otimes \mathbb{C}[t^{-1}]t^{-1}$.

Let $k$ be a complex number.
Consider the highest weight Verma module $M(k\Lambda_0)$ over $\widehat{\mathfrak{g}}$ of highest weight $k\Lambda_0$.
As a $\widehat{\mathfrak{g}}_-$-module, $M(k\Lambda_0)$ is isomorphic to $\mathrm{S}(\widehat{\mathfrak{g}}_-)$.
Here, $\mathrm{S}(V)$ denotes the symmetric algebra of $V$ for each vector space $V$.
Then, $M(k\Lambda_0)\cong \mathrm{S}(\widehat{\mathfrak{g}}_-)$ equips with the compatible vertex algebra structure
with the vacuum vector $|0\rangle=1$,
translation operator $\partial$ defined by $\partial(v\otimes t^n)=-nv\otimes t^{n-1}$ with derivation,
 vertex operator $Y(\cdot,z)$
defined by the assignment $Y(a\otimes t^{-1},z)=\sum_{n\in \mathbb{Z}} a\otimes t^n z^{-n-1}$ and
 extending it by the reconstruction theorem.
The vertex algebra $M(k\Lambda_0)$ has the weight grading induced by $-D$
and {\it Segal-Sugawara conformal vector}
\[
\omega^{\mathrm{aff}}=\frac{1}{2(k+h^\vee)}\sum_{i\in A} v_i(-1)v^i
\]
of central charge  $c=k \dim \mathfrak{g}/(k+h^\vee)$.
Here, $\{v_i\}_{i\in A}$ is a basis of $\mathfrak{g}$ with the index set $A$, and $\{v^i\}_{i\in A}$ the dual-basis
with respect to $(\cdot|\cdot)$.
The VOA $M(k\Lambda_0)$ is called the {\it (universal) affine vertex operator algebra} and 
denoted by $V^k(\mathfrak{g})$.
Note that the weight $1$ subspace $V^k(\mathfrak{g})_1=\{a\otimes t^{-1}|a\in \mathfrak{g}\}$ with the Lie bracket $[a\otimes t^{-1},b\otimes t^{-1}]:=a(0)b=[a,b]$ ($a,b\in\mathfrak{g}$) is isomorphic to 
the Lie algebra $\mathfrak{g}$ via $a\mapsto a\otimes t^{-1}$.
We denote the simple quotient by $V_k(\mathfrak{g})$ and call it the {\it (simple) affine vertex operator algebra}.
Note that as a $\widehat{\mathfrak{g}}$-module, $V_k(\mathfrak{g})$ is isomorphic to
the irreducible highest weight module $L(k\Lambda_0)$.
It is well-known that $V_k(\mathfrak{g})$ is $C_2$-cofinite and rational if and only if $k$ is a positive integer.

\subsection{The lattice generalized vertex algebras associated with the rational lattices and the lattice vertex algebras
associated with the even integral lattices}\label{sec:lattice}

In this section, we recall the lattice GVAs \cite{DL,BK} associated with the rational lattices and the lattice vertex algebras
associated with the even integral lattices.

Let $L=\mathbb{Z}^l$ be a rank $l$ rational lattice with the $\mathbb{Z}$-bilinear form $(\cdot|\cdot):L\times L\rightarrow \mathbb{Q}$.
Consider the vector space $\mathfrak{h}=\mathbb{C}\otimes_\mathbb{Z} L$ with the $\mathbb{C}$-bilinear form $(\cdot|\cdot):\mathfrak{h}\times \mathfrak{h}\rightarrow \mathbb{C}$ defined by linearly extending the $\mathbb{Z}$-bilinear form $(\cdot|\cdot)$.

Let $\widehat{\mathfrak{h}}=\mathfrak{h}\otimes \mathbb{C}[t,t^{-1}]\oplus \mathbb{C}K$ denote the Heisenberg Lie algebra
with the central element $K$ and Lie bracket $[a\otimes t^n,b\otimes b^m]=n\delta_{n+m,0}(a|b)K$.
Set $a(n)=a\otimes t^n$ ($a\in \mathfrak{h}$, $n\in\mathbb{Z}$).
Set $\widehat{\mathfrak{h}}_+=\mathfrak{h}\otimes \mathbb{C}[t]\oplus \mathbb{C}K$ and
 $\widehat{\mathfrak{h}}_-=\mathfrak{h}\otimes \mathbb{C}[t^{-1}]t^{-1}$.
Let $\mathbb{C}_1$ denote the $1$-dimensional module of $\widehat{\mathfrak{h}}_+$ with 
$a\otimes t^n. v=0$ and $K. v=v$ ($a\in \mathfrak{h}$, $n\geq 0$, $v\in \mathbb{C}_1$).
Let $M(1)=\widehat{\mathfrak{h}}\otimes_{\widehat{\mathfrak{h}}_+}\mathbb{C}_1$
denote the induced module of $\widehat{\mathfrak{h}}$.
As vector spaces, $M(1)\cong \mathrm{S}(\widehat{\mathfrak{h}}_-)$.
Here, $\mathrm{S}(V)$ denotes the symmetric algebra of $V$ for each vector space $V$.
Then, $M(1)$ has the vertex algebra structure ({\it Heisenberg vertex algebra}) with the
vacuum vector $1=1\otimes 1$, translation operator $T$ with $T(v\otimes t^n)=-n v\otimes t^{n-1}$ and vertex operator
$X(\cdot,z)$ defined by the assignment
$X(v\otimes t^{-1},z)=\sum_{n\in \mathbb{Z}}v\otimes t^n z^{-n-1}$
and extending it to $M(1)$ by the reconstruction theorem (cf.\ \cite{FKRW,Kac2,FBZ}).
We equip $M(1)$ with the conformal vector $\omega=(1/2)\sum_{i=1}^l v_i(-1)v^i$ of central charge $l$.
Here, $\{v_i\}_{i=1,\ldots,l}$ is a basis of $\mathfrak{h}$, and $\{v^i\}_{i=1,\ldots,l}$ the dual basis of $\{v_i\}$.

Let 
$
\mathbb{C}[L]=\bigoplus_{\alpha\in L} \mathbb{C}e^\alpha
$
 denote the group algebra of $L$ with $e^\alpha \cdot e^\beta=e^{\alpha+\beta}$.
We consider the space $\mathbb{C}e^\alpha$ as the module over $\widehat{\mathfrak{h}}_+$ with
 $a\otimes t^0.e^\alpha=(a|\alpha)e^\alpha$, $a\otimes t^n. e^\alpha=0$ ($a\in \widehat{\mathfrak{h}}$, $n> 0$) and 
$K.e^\alpha=e^\alpha$.
Then, $M(1)\otimes \mathbb{C}e^\alpha$ has the module structure over the vertex algebra $M(1)$.
Consider the $L$-charged extended generalized vertex algebra 
\[
V_L=\bigoplus_{\alpha\in L}M(1)\otimes \mathbb{C}e^\alpha=M(1)\otimes \mathbb{C}[L]\cong \mathrm{S}(\widehat{\mathfrak{h}}_-)\otimes \mathbb{C}[L],
\]
 with the vacuum vector $|0\rangle=1\otimes e^0$,
translation operator $\partial$ with $\partial(e^\alpha)=\alpha{(-1)}e^\alpha$ and the vertex operator $X(\cdot,z)$ defined
by the assignment
\[
X(e^\alpha,z)=\mathrm{exp}\left(\sum_{n<0}\frac{1}{-n}\alpha_{n}z^{-n}\right)\mathrm{exp}\left(\sum_{n>0}\frac{1}{-n}\alpha_{n}z^{-n}\right)\otimes e^\alpha z^\alpha,
\]
($\alpha\in L$) and extending it to $V_L$ by the reconstruction theorem.
Here, $z^\alpha$ is defined by linearly extending the assignment $z^\alpha e^\beta=z^{(\alpha|\beta)}e^\beta$ ($\alpha,\beta\in L$).
We call $V_L$ the {\it lattice generalized vertex algebra} (lattice GVA) associated with $L$.

Suppose that $L$ is {\it even integral}, that is, $(\cdot|\cdot):L\times L\rightarrow \mathbb{Z}$ and $(\alpha|\alpha)\in 2\mathbb{Z}$ for each $\alpha\in L$.
Consider a $2$-cocycle $\varepsilon:L\times L\rightarrow \{\pm 1\}$ satisfying $\varepsilon(\alpha,\beta)\varepsilon(\beta,\alpha)=(-1)^{(\alpha|\beta)}$.
Then, the $\varepsilon$-modified GVA $V_L^\varepsilon$ with the vertex operator $Y(\cdot,z)=X^\varepsilon(\cdot,z)$
is a vertex algebra and called the {\it lattice vertex algebra}.
We denote $V_L=V_L^\varepsilon$.
We consider $V_L$ as a vertex operator algebra with the conformal vector $\omega\otimes e^0$ of central charge $l$, and call $V_L$
a {\it lattice vertex operator algebra}.

Let $\mathfrak{g}$ be a finite dimensional simple Lie algebra.
Suppose that $\mathfrak{g}$ is simply-laced (of ADE-type).
Note that the lattice vertex operator algebra $V_Q$ associated with the root lattice $Q=Q(\mathfrak{g})$ of $\mathfrak{g}$
is isomorphic to the level one affine VOA $V_1(\mathfrak{g})$.

\section{Appendix D. General theorems for the abelian intertwining algebras}
\label{sec:appendd}
\subsection{Some computational lemmas on abelian cocycles}
\label{sec:computational}

We show some computational lemmas on abelian $3$-cocycles.

First, we show Remark \ref{sec:lem1}.

\begin{proof}[Proof of Remark \ref{sec:lem1}]
The map $f$ is a group homomorphism. We show $\mathrm{Ker}(f)=1$.
Let $(F,\Omega)$ be a normalized abelian $3$-cocycle such that $F(j,i,k)^{-1}\Omega(i,j)F(i,j,k)=1$ ($i,j,k\in Q$).
By (A2), we have $F(j,k,i)=\Omega(i,j+k)\Omega(i,k)^{-1}$. 
By letting $k=0$, by (A4) and (A5), we have $\Omega(i,j)=1$ for all $i,j\in Q$.
Therefore, $F(j,k,i)=1$ for all $i,j,k$.
Thus, $(F,\Omega)=1$.
\end{proof}
\end{rem}

\begin{lem}\label{sec:lem2}
\begin{eqnarray}\label{eqn:feq1}
F(i,j,k+l)^{-1}B(j,k,l)B(i,k,j+l)F(i,j,l)=B(i+j,k,l).\nonumber \\
\end{eqnarray}
\end{lem}

\begin{proof}
By the definition of $B$, eq.\,(\ref{eqn:feq1}) is equivalent to
\begin{eqnarray*}
&&F(i,j,k+l)^{-1}F(k,j,l)^{-1}\Omega(j,k)F(j,k,l) F(k,i,j+l)^{-1}\Omega(i,k)\\
&&\quad\cdot F(i,k,j+l)F(i,j,l)=F(k,i+j,l)^{-1}
\Omega(i+j,k)F(i+j,k,l).
\end{eqnarray*}
By (A3), it is equivalent to
\begin{eqnarray}\label{eqn:feq2}
&&F(i,j,k)F(k,i,j)F(i,k,j)^{-1}F(i,j,k+l)^{-1}F(k,j,l)^{-1}\\
&&\quad\cdot F(j,k,l) F(k,i,j+l)^{-1}
F(i,k,j+l)F(i,j,l)\nonumber \\
&&\quad=F(k,i+j,l)^{-1}F(i+j,k,l).\nonumber
\end{eqnarray}
Now, we show that eq.\,(\ref{eqn:feq2}) is true.
By (A1), we have three equations
\begin{eqnarray}\label{eqn:feq3}
F(i,j,k)F(i,j,k+l)^{-1}F(i,j+k,l)F(i+j,k,l)^{-1}F(j,k,l)=1,\nonumber \\
\end{eqnarray}
\begin{eqnarray}\label{eqn:feq4}
F(k,i,j)F(k,i,j+l)^{-1}F(k,i+j,l)F(k+i,j,l)^{-1}F(i,j,l)=1,\nonumber \\
\end{eqnarray}
\begin{eqnarray}\label{eqn:feq5}
F(i,k,j)F(i,k,j+l)^{-1}F(i,k+j,l)F(i+k,j,l)^{-1}F(k,j,l)=1.\nonumber \\
\end{eqnarray}
Then, eq.\,(\ref{eqn:feq2}) is equivalent to eq.\,($(\ref{eqn:feq3})\times (\ref{eqn:feq4})/(\ref{eqn:feq5})$).
Thus, we have the lemma.
\end{proof}

\begin{lem}\label{sec:lem3}
\begin{enumerate}
\item Let $(F,\Omega)$ be a normalized abelian $3$-cocycle, and supporse $F(i,j,k)=1$ for any $i,j,k\in Q$. Then, $\Omega$ is bimultiplicative.

\item Conversely, if the function $\eta:Q\times Q\rightarrow \mathbb{C}^\times$ is bimultiplicative,
then the pair $(1,\eta)$ is a normalized abelian $3$-cocycle.

\item Let $(F,\Omega)$ be a normalized abelian $3$-cocycle, and suppose $F(i,j,k)=F(j,i,k)$ for any $i,j,k\in Q$.
Then,
\begin{itemize}
\item[(a)] $F(i,j,k)=\Omega(i+j,k)^{-1}\Omega(i,k)\Omega(j,k)$, 
\item[(b)] $F(i,j,k)$ is multiplicative in $k$.
\end{itemize}

\item Conversely, if $\eta:Q\times Q\rightarrow \mathbb{C}^\times$ is a function such that
\begin{itemize} 
\item[(a)] $\mu(i,j,k):=\eta(i+j,k)^{-1}\eta(i,k)\eta(j,k)$ is multiplicative in $k$,
\item[(b)] there exists the bilinear map $\Delta:Q\times Q\rightarrow \mathbb{C}/\mathbb{Z}$ such that $\eta(i,j)\eta(j,i)=e^{-2\pi i \Delta(i,j)}$,
\item[(c)] $\eta(Q,0)=\eta(0,Q)=1$,
\end{itemize}
then $(\mu,\eta)$ is a normalized abelian $3$-cocycle such that $\mu(i,j,k)=\mu(j,i,k)$.
\end{enumerate}
\end{lem}

The proof is omitted.

Let $Q$ be an abelian group and $(F,\Omega)$ a normalized abelian $3$-cocycle.

\begin{lem}\label{sec:lem4}
If $Q=\mathbb{Z}_2$, then $F(i,j,k)=F(j,i,k)$ for any $i,j,k\in Q$.
\end{lem}

\begin{proof}
By (A4), we have $F(1,0,k)=F(0,1,k)=1$.
\end{proof}

\subsection{Some theorems for the abelian intertwining algebras}
\label{sec:aia}

We recall some standard theorems for the AIAs.
We generalize the method used in \cite{BK}.

Let $Q$ be an abelian group and $(F,\Omega)$ a normalized abelian $3$-cocycle with maps $B$ and $\Delta$.
Let $V=\bigoplus_{\alpha\in Q} V^\alpha$ be a $Q$-graded vector space with a translation operator $\partial:V\rightarrow V$ and a vacuum vector $|0\rangle\in V$ such that $\partial|0\rangle=0$.

Let $n$ be a positive integers.

\begin{dfn}(\cite{BK})
 An {\it $n$-field} with {\it charge} ($\alpha_1,\ldots,\alpha_n)\in Q^{\times n}$ on $V$
is a formal series $a(z_1,\ldots,z_n)\in (\mathrm{End} (V)) [[z_1,z_1^{\mathbb{C}},\ldots,z_n,z_n^{\mathbb{C}}]]$
with the property that
\begin{eqnarray*}
&&a(z_1,\ldots,z_n)b \\
&&\quad \in V^{\alpha_1+\cdots+\alpha_n+\beta}[[z_1,\ldots,z_n]]z_1^{-\Delta(\alpha_1,\alpha_2+\cdots+\alpha_n+\beta)}\cdots z_n^{-\Delta(\alpha_n,\beta)}
\end{eqnarray*}
for $b\in V^\beta$.
We call $\alpha=\alpha_1+\cdots+\alpha_n$ the {\it total charge} of $a(\mathbf{z})=a(z_1,\ldots,z_m)$.
\end{dfn}
Let $m,n$ be positive integers.
Let $a(z_1,\ldots,z_m)$ be an $m$-field with total charge $\alpha\in Q$ and $b(w_1,\ldots,w_n)$ an $n$-field with total charge $\beta\in Q$.

\begin{dfn}
The pair $(a,b)$ of fields
is called {\it local} if there exist $N_{ij}\in \Delta(\alpha_i,\beta_j)$ ($i=1,\ldots,m$, $j=1,\ldots,n$) such that for any $\gamma\in Q$ and $c\in V^\gamma$,
\begin{eqnarray*}\label{eqn:local1}
&&\left(\prod_{i=1}^m \prod_{j=1}^n \iota_{z_i,w_j}(z_i-w_j)^{N_{ij}}\right) a(z_1,\ldots,z_m)b(w_1,\ldots,w_n)c
=B(\alpha,\beta,\gamma)\\
&&\quad \cdot \left( \prod_{i=1}^m \prod_{j=1}^n \iota_{w_j,z_i}(z_i-w_j)^{N_{ij}}\right) b(w_1,\ldots,w_n)a(z_1,\ldots,z_m)c.
\end{eqnarray*}
We call such an $\{N_{ij}\}_{i,j}$ a system of {\it locality bounds} of $(a,b)$.
\end{dfn}

Let $m,n$ be positive integers.
Let $a(z_1,\ldots,z_m)$ and $b(z_{m+1},\ldots,z_{m+n})$ be an $m$-field and $n$-field with total charges $\alpha$ and $\beta$.
Suppose that the pair $(a,b)$ is local, and fix a system of locality bounds $\{N_{ij}\}_{i\in\{1,\ldots,m\}, j\in\{1,\ldots,n\}}$.

Define the $(m+n)$-field $F_{a,b}(z_1,\ldots,z_{m+n})$ by linearly extending the assignment
\begin{eqnarray*}
F_{a,b}(z_1,\ldots,z_{m+n})d &=& F(\alpha,\beta,\delta)^{-1}\left( \prod_{i=1}^{m} \prod_{j=1}^{n}\iota_{z_i,z_{m+j}}(z_i-z_{m+j})^{N_{ij}}\right)\\
&&\cdot a(z_1,\ldots,z_m)b(z_{m+1},\ldots,z_{m+n})d
\end{eqnarray*}
($\delta\in Q$, $d\in V^\delta$).
Note that the total charge of $F_{a,b}(\mathbf{z})$ is $\alpha+\beta$.
Note that if $a,b$ are translation covariant $\partial:V\rightarrow V$, then $F_{a,b}$ is translation covariant.

Let $c(w_1,\ldots,w_p)$ be a $p$-field with total charge $\gamma\in Q$.

\begin{lem}\label{sec:localf}
If the pairs $(a,c)$ and $(b,c)$ are local, then the pair $(F_{a,b},c)$ is local.
\end{lem}

\begin{proof}
Suppose that the pairs $(a,c)$ and $(b,c)$ are local.
Let $\{M_{ik}\}_{i\in\{1,\ldots,m\},k\in \{1,\ldots,p\}}$and $\{M_{jk}\}_{j\in\{m+1,\ldots,m+n\},k\in\{1,\ldots,p\}}$ be systems of locality bounds of $(a,c)$ and $(b,c)$.
We show that $\{M_{ik}\}_{i\in\{1,\ldots,m+n\},k\in\{1,\ldots,p\}}$ is a system of locality bounds of
$(F_{a,b},c)$.
Let $\delta$ be an element of $Q$ and $d$ an element of $V^\delta$.
We have
\begin{eqnarray*}
&&\left(\prod_{i=1}^{m+n}\prod_{k=1}^p \iota_{z_i,w_k}(z_i,w_k)^{M_{ik}}\right)F_{a,b}(\mathbf{z})c(\mathbf{w})d\\
&=&F(\alpha,\beta,\gamma+\delta)^{-1}
\left(\prod_{i=1}^{m+n}\prod_{k=1}^p \iota_{z_i,w_k}(z_i,w_k)^{M_{ik}}\right)\\
&&\left(\prod_{i=1}^{m}\prod_{j=1}^n \iota_{z_i,z_{m+j}}(z_i,z_{m+j})^{N_{ij}}\right)
 a(z_1,\ldots,z_m)b(z_{m+1},\ldots,z_{m+n})c(\mathbf{w})d\\
&=&F(\alpha,\beta,\gamma+\delta)^{-1} B(\beta,\gamma,\delta)\left(\prod_{i=1}^{m}\prod_{k=1}^p \iota_{z_i,w_k}(z_i,w_k)^{M_{ik}}\right)\\
&&\left(\prod_{i=m+1}^{n}\prod_{k=1}^p \iota_{w_k,z_i}(z_i,w_k)^{M_{ik}}\right)
\left(\prod_{i=1}^{m}\prod_{j=1}^n \iota_{z_i,z_{m+j}}(z_i,z_{m+j})^{N_{ij}}\right)\\
&&\cdot a(z_1,\ldots,z_m)c(\mathbf{w})b(z_{m+1},\ldots,z_{m+n})d\\
&=&F(\alpha,\beta,\gamma+\delta)^{-1} B(\beta,\gamma,\delta) B(\alpha,\gamma,\beta+\delta)\\ &&\cdot \left(\prod_{i=1}^{m+n}\prod_{k=1}^p \iota_{w_k,z_i}(z_i,w_k)^{M_{ik}}\right)
 \left(\prod_{i=1}^{m}\prod_{j=1}^n \iota_{z_i,z_{m+j}}(z_i,z_{m+j})^{N_{ij}}\right)\\
&&\cdot c(\mathbf{w})a(z_1,\ldots,z_m)b(z_{m+1},\ldots,z_{m+n})d\\
&=&F(\alpha,\beta,\gamma+\delta)^{-1} B(\beta,\gamma,\delta) B(\alpha,\gamma,\beta+\delta)\cdot F(\alpha,\beta,\delta)\\
&&\cdot\left(\prod_{i=1}^{m+n}\prod_{k=1}^p \iota_{w_k,z_i}(z_i,w_k)^{M_{ik}}\right)c(\mathbf{w})F_{a,b}(\mathbf{z})d.
\end{eqnarray*}
Therefore, it suffices to show
\[
F(\alpha,\beta,\gamma+\delta)^{-1} B(\beta,\gamma,\delta) B(\alpha,\gamma,\beta+\delta) F(\alpha,\beta,\delta)=B(\alpha+\beta,\gamma,\delta).
\]
By Lemma \ref{sec:lem2}, we have the lemma.
\end{proof}

Let $m$ be a positive integer.
Let $A(z_1,\ldots,z_m)$ be an $m$-field.
Define $\iota_{z,w_1}\cdots\iota_{z,w_{m-1}} A(z+w_1,\ldots,z+w_{m-1},z)$ to be
\begin{eqnarray*}
&&\iota_{z,w_1}\cdots\iota_{z,w_{m-1}} A(z+w_1,\ldots,z+w_{m-1},z)v\\
&&\quad :=\mathrm{exp}(w_1\partial_{z_1}+\cdots+w_{m-1}\partial_{z_{m-1}})A(z_1,\ldots,z_m)v|_{z_1=\cdots=z_m=z}\\
&&\quad \in (V[[z]]z^\mathbb{C})[[w_1,\ldots,w_{m-1}]]
\end{eqnarray*}
($v\in V$).
We call this formal series a {\it operator product expansion (OPE)} of $A$.
Define the $1$-fields $\psi_{i(1),\ldots,i(m-1)}(z)$ ($i(1),\ldots,i(m-1)\in \mathbb{Z}_+$) by
the equality
\begin{eqnarray*}
&&\iota_{z,w_1}\cdots\iota_{z,w_{m-1}} A(z+w_1,\ldots,z+w_{m-1},z)\\
&&\quad =\sum_{i(1),\ldots,i(m-1)=0}^\infty
\psi_{i(1),\ldots,i(m-1)}(z)\, w^{i(1)}\cdots w^{i(m-1)}.
\end{eqnarray*}
Let $i(1),\ldots,i(m-1)$ be elements of $\mathbb{Z}_+$, and put $\psi(z)=\psi_{i(1),\ldots,i(m-1)}(z)$.
If $A$ is translation covariant, then the coefficient $\psi(z)$ is translation covariant.
Let $B$ be an $n$-field ($n\in\mathbb{Z}_{>0}$).
If the pair $(A,B)$ is local, then the pair $(\psi,B)$ is local.
Note that $\iota_{z,w_1}\cdots\iota_{z,w_{m-1}} A(z+w_1,\ldots,z+w_{m-1},z)|0\rangle$
is the Taylor series expansion of
\[
A(z+w_1,\ldots,z+w_{m-1},z)|0\rangle \in V[[z,w_1,\ldots,w_{m-1}]].
\]
Therefore, the linear span of all coefficients of $A(z_1,\ldots,z_m)|0\rangle$ coincides with
the linear span of all coefficients of $\psi_{i(1),\ldots,i(m-1)}(z)|0\rangle$ in $z$ ($i(1),\ldots,i(m-1)\in \mathbb{Z}_+$).

Let $a(z)$ be a translation covariant field with charge $\alpha\in Q$.
Let $b$ be an element of $V^\beta$ with $\beta\in Q$.
The series $Y(a,z)b$ has the form $a(z)b=\sum_{n=0}^\infty a{(-N-n-1)}b z^{N+n}$ with $N\in\Delta(\alpha,\beta)$.
Then, $f(z+e^{-\pi i}w)\in U[[z,w]](z+e^{-\pi i}w)^N$.
We have the following transposed Taylor's theorem.

\begin{lem}\label{sec:taylor}
As elements of $(V[[z]]z^{\Delta(\alpha,\beta)})[[w]]$,
\[
\iota_{z,w}(Y(a,z+e^{-\pi i}w)b)=e^{-w\partial}Y(a,z)e^{w\partial}b.
\]
\end{lem}

\begin{proof}
By the translation covariance, the RHS is equal to $e^{w\partial_z}Y(a,z)b$, which is equal to the LHS.
\end{proof}

Let $\Phi$ be a system of fields such that
\begin{enumerate}
\item (locality) for any $\phi(z),\psi(z)\in \Phi$ with charges $\alpha,\beta\in Q$, the fields $\phi(z)$ and $\psi(z)$ are mutually $\eta(\alpha,\beta)$-local,
\item (translation covariance) every $\phi(z)\in \Phi$ is translation covariant, that is, $[\partial,\phi(z)]=\partial_z \phi(z)$,
\item (completeness) the coefficients of all formal series $\phi_1(z_1)\cdots \phi_n(z_n)|0\rangle$ ($n\in \mathbb{Z}_+$, $\phi_1,\ldots,\phi_n\in \Phi$)
span the vector space $V$.
\end{enumerate}

Let $S$ denote the vector space spanned by all translation covariant fields $\phi(z)$ such that
for any $\psi(z)\in \Phi$, the fields $\phi(z)$ and $\psi(z)$ are mutually local.
Since the elements of $S$ are translation covariant, by \cite[Proposition 2.2(c)]{BK}, we have the well-defined linear map 
\[
\Psi:S\rightarrow V, \quad \chi(z)\mapsto \chi(z)|0\rangle|_{z=0}.
\]

Let $\alpha$ be an element of $Q$ and $a$ an element of $V^\alpha$.

\begin{thm}\label{sec:statefield}
There exists the unique translation covariant field $Y(a,z)$ of charge $\alpha$ such that
for any field $\phi(z)\in \Phi$, the fields
$Y(a,z)$ and $\phi(z)$ is mutually local and $Y(a,z)|0\rangle|_{z=0}=a$.  
\end{thm}

\begin{proof}

We show that $\Psi$ is an isomorphism of vector spaces.

{\bf (Surjectivity)}
Let $m$ be a positive integer.
Let $\phi_1(z_1),\ldots,\phi_m(z_m)$ be fields in $\Phi$ with charges $\alpha_1,\ldots,\alpha_m\in Q$.
We show that the coefficients of $\phi_1(z_1)\cdots\phi_m(z_m)|0\rangle$ belong to the image of $\Psi$, hence, by the completeness of the system $\Phi$, the map $\Psi$ is surjective.

Fix locality bounds $N_{ij}$ of $(\phi_i,\phi_j)$ ($i,j=1,\ldots,m$, $i\neq j$).
Define the $m$-field $A(z_1,\ldots,z_m)$ by linearly extending the assignment
\begin{eqnarray*}
A(z_1,\ldots,z_m)b&=& \prod_{i=1}^{m-1} B\left(\alpha_i,\alpha_{i+1}+\cdots+\alpha_m,\beta\right) \\
&&\cdot \left(\prod_{1\leq i< j\leq m} \iota_{z_i,z_j}(z_i-z_j)^{N_{ij}}\right) \phi_1(z_1)\cdots\phi_m(z_m)b
\end{eqnarray*}
($\beta\in Q$, $b\in V^\beta$).
Then, $A$ is translation covariant, and by Lemma \ref{sec:localf}, the pairs $(A,\phi)$ are local for all fields $\phi$ in $\Phi$.
Consider the operator product expansion of $A$.
Then, all coefficients $\psi_{i(1),\ldots,i(m-1)}(z)$ ($i(1),\ldots,i(m-1)\in \mathbb{Z}_+$)
belong to $S$.
 Therefore, all coefficients of $A(z_1,\ldots,z_m)|0\rangle$ belong to the image of the map $\Psi$.

The product $\phi_1(z_1)\cdots \phi_m(z_m)|0\rangle$ belongs to the space $X:=V[[z_1]]z_1^{\mathbb{C}}\cdots [[z_m]]z_m^{\mathbb{C}}$.
The space $X$ is a module over the algebra
$Y:=\mathbb{C}[[z_1]]z_1^{\mathbb{C}}\cdots [[z_m]]z_m^{\mathbb{C}}$.
Since $\iota_{z_i,z_j}(z_i-z_j)^{N_{ij}}$ is invertible in $Y$ for each $i<j$, we have
\[
\phi_1(z_1)\cdots\phi_m(z_m)|0\rangle=\left(\prod_{1\leq i<j\leq m} \iota_{z_i,z_j}(z_i-z_j)^{-N_{ij}}\right)
A(z_1,\ldots,z_m)|0\rangle.
\]
Therefore, the coefficients of $\phi_1(z_1)\cdots\phi_m(z_m)|0\rangle$ belong to the image of $\Psi$.
By the completeness of the system $\Phi$, the map $\Psi$ is surjective.

{\bf (Injectivity)}
Let $\phi(z)$ be an element of the kernel of $\Psi$.
Then, $\phi(z)|0\rangle|_{z=0}=0$.
Since $S$ is spanned by fields, $\phi(z)$ is a sum of the form $\phi(z)=\sum_{\alpha\in Q}\chi_\alpha(z)$ with the fields $\chi_\alpha(z)$ with charge $\alpha\in Q$.
Let $\alpha$ be an element of $Q$, and put $\chi(z)=\chi_\alpha(z)$.
We show $\chi(z)=0$.
Since $\chi_\beta(z)|0\rangle|_{z=0}$ belongs to $V^\beta$ for any $\beta\in Q$, we have
$\chi(z)|0\rangle|_{z=0}=0$.
Since $\chi(z)|0\rangle=e^{z\partial}(\chi(w)|0\rangle|_{w=0})$ by the translation covariance, 
we have $\chi(z)|0\rangle=0$.
Let  $m$ be a positive integer.
Let $\phi_1(z_1),\ldots,\phi_m(z_m)$ be fields in $\Phi$.
Fix locality bounds $N_{ij}$ of $(\phi_i,\phi_j)$ ($i,j=1,\ldots,m$, $i\neq j$).
Define the $m$-field $A(z_1,\ldots,z_m)$ with total charge $\beta$ as above.
Since the pair $(\chi,A)$ is local, we have
\begin{eqnarray*}
&&\left(\prod_{i=1}^m\iota_{z,z_i}(z-z_i)^{M_i}\right)\chi(z)A(z_1,\ldots,z_m)|0\rangle\\
&&\quad =B(\alpha,\beta,0)\left(\prod_{i=1}^m \iota_{z_i,z}(z-z_i)^{M_i}\right)A(z_1,\ldots,z_m)\chi(z)|0\rangle
\end{eqnarray*}
with the scalars $M_i\in\mathbb{C}$.
Since $\chi(z)|0\rangle=0$, we have
\[
\left(\prod_{i=1}^m \iota_{z,z_i}(z-z_i)^{M_i} \right)\chi(z)A(z_1,\ldots,z_m)|0\rangle=0.
\]
The LHS belongs to the space $X:=V[[z]]z^{\mathbb{C}}[[z_1]]z_1^{\mathbb{C}}\cdots [[z_m]]z_m^{\mathbb{C}}$.
The space $X$ is a module over the algebra 
$Y:=\mathbb{C}[[z]]z^{\mathbb{C}}[[z_1]]z_1^{\mathbb{C}}\cdots [[z_m]]z_m^{\mathbb{C}}$.
Since $\iota_{z,z_i}(z-z_i)^{M_i}$ is invertible in the algebra $Y$
 for each $i=1,\ldots,m$,
we see that $\chi(z) v=0$ for all coefficients $v$ of $A(z_1,\ldots,z_m)|0\rangle$.
Since the coefficients of all $A(z_1,\ldots,z_m)|0\rangle$ span the space $V$,
 we have $\chi(z)=0$, as desired.
\end{proof}

Now, we have a generalization of the reconstruction theorem.

\begin{cor}
The system $\Phi$ generates on $V$ the unique structure of the abelian intertwining algebra.
Moreover, the pair $(V,(Q,F,\Omega))$ is a $Q$-charged abelian intertwining algebra.
\end{cor}

Let $(V,(Q,F,\Omega))$ be a $Q$-charged AIA.
Let $\alpha,\beta$ be elements of $Q$ and $a,b$ elements of $V^\alpha,V^\beta$.

\begin{lem}[Skew-Symmetry]\label{sec:skewsymmetry}
\[
Y(a,z)b=\Omega(\alpha,\beta)e^{z\partial}(Y(b,e^{-\pi i}z)a).
\]
\end{lem}

\begin{proof}
By the locality axiom,
\begin{eqnarray*}
&&\iota_{z,w}(z+e^{-\pi i}w)^N Y(a,z)Y(b,w)|0\rangle\\
&&\quad =B(\alpha,\beta,0)\iota_{w,z}(z+e^{\pi i}w)^N Y(b,w)Y(a,z)|0\rangle
\end{eqnarray*}
with $N\in\Delta(\alpha,\beta)$.
Since $B(\alpha,\beta,0)=\Omega(\alpha,\beta)$, $Y(a,z)|0\rangle=e^{z\partial}a$ and $Y(b,w)|0\rangle=e^{w\partial}b$,
\[
\iota_{z,w}(z+e^{-\pi i}w)^N Y(a,z)e^{w\partial}b=\Omega(\alpha,\beta)\iota_{w,z}(z+e^{\pi i}w)^N Y(b,w)e^{z\partial}a.
\]
By  Lemma \ref{sec:taylor}, the RHS is equal to
\[
\Omega(\alpha,\beta)e^{\pi i}\iota_{w,z}(e^{-\pi i}z+w)^N e^{z\partial}\iota_{w,z}(Y(b,w+e^{-\pi i}z)a),
\]
and since $N$ is a locality bound, it belongs to $\mathrm{End}(V)[[w]][[z]]$.
Therefore, by letting $w=0$, we have the lemma.
\end{proof}

Let $\alpha,\beta$ be elements of $Q$ and $a,b$ elements of $V^\alpha,V^\beta$.
Fix a locality bound $N$ of $(Y(a,z),Y(b,w))$ and set $F_{a,b}(z,w):=F_{Y(a,z),Y(b,w)}(z,w)$.
Consider the OPE
\[
\iota_{z,w} F_{a,b}(z+w,z)=e^{w\partial_{z_1}} F_{a,b}(z_1,z_2)|_{z_1=z_2=z}
\]
 of $Y(a,z)$ and $Y(b,w)$.

Let $k$ be a non-negative integer.

\begin{lem}
\[
\frac{1}{k!}\partial_{z_1}^k F_{a,b}(z_1,z_2)|_{z_1=z_2=z}=Y(a{(N-1-k)}b,z).
\]
\end{lem}

\begin{proof}
Denote the LHS by $X$.
By Theorem \ref{sec:statefield}, it suffices to show $X|0\rangle|_{z=0}=a_{(N-1-k)}b$.
Since $F_{a,b}(z_1,z_2)|0\rangle\in V[[z_1,z_2]]$, we have $F_{a,b}(z_1,z_2)|0\rangle|_{z_2=0}=z_1^N Y(a,z_1)b\in V[[z_1]]$.
Since $(1/k!) \partial_{z_1}^k (z_1^NY(a,z_1)b)|_{z_1=0}=a{(N-1-k)}$, we have the lemma.
\end{proof}

We have the following {\it Jacobi identity}.
Let $\alpha,\beta,\gamma$ be elements of $Q$, $a,b,c$ elements of $V^\alpha,V^\beta,V^\gamma$, and $n$ an element of $\Delta(\alpha,\beta)$.

\begin{prop}
\begin{eqnarray*}
&& Y(a,z)Y(b,w)c\,\iota_{z,w}(z-w)^n-B(\alpha,\beta,\gamma) Y(b,w)Y(a,z)c\,\iota_{w,z}(z-w)^n\\
&&\quad\quad=F(\alpha,\beta,\gamma)^{-1}\sum_{j\in\mathbb{Z}}Y(a{(n+j)}b,w)c\, \partial_w^j \delta_{\Delta(\alpha,\beta)}(z,w)/j!.
\end{eqnarray*}
\end{prop}

The proof is omitted. See \cite{BK}.

Suppose that $V$ is a vertex algebra, that is, $F\equiv 1$ and $\Omega\equiv 1$.
Let $a,b$ be elements of $V$. Then, we have the commutation relation
\begin{equation}\label{eqn:comm}
[a(n),b(m)]=\sum_{k=0}^\infty \begin{pmatrix}n\\ k\end{pmatrix} (a(k)b)(n+m-k).
\end{equation}


\begin{thebibliography}{99}

\bibitem[A1]{A1} Arakawa, T. ``Representation theory of superconformal algebras and the Kac-Roan-Wakimoto conjecture." Duke Mathematical Journal 130.3 (2005): 435--478.

\bibitem[A2]{A2} Arakawa, T. ``Representation theory of W-algebras, II: Ramond twisted representations."  Adv. Stud. Pure Math. 61(2011): 51--90.

\bibitem[A3]{A3} Arakawa, T. ``Associated varieties of modules over Kac-Moody algebras and $ C_2 $-cofiniteness of W-algebras." Int. Math. Res. Not. (2015), published online.

\bibitem[A4]{A4} Arakawa, T. ``Rationality of Bershadsky-Polyakov vertex algebras." Comm.\ Math.\ Phys.\ 323.2 (2013): 627--633.

\bibitem[Bor]{B}
Borcherds, R.\ E. ``Vertex algebras, Kac-Moody algebras, and the Monster'', Proc.\ Nat.\ Acad.\ Sci.\ {\bf 83.10} (1986): 3068--3071.

\bibitem[Bou]{Bou}
Bourbaki, N. ``Lie Groups and Lie Algebras'' Chapters 4-6. Vol. 2. Springer Science \& Business Media, 2008.

\bibitem[BK]{BK} Bakalov, B., and Kac V.\ G. ``Generalized vertex algebras'', Proceedings of the 6-th International Workshop ``Lie Theory and Its Applications in Physics", Varna, Bulgaria  (2006): 3--25.

\bibitem[C1]{Car2} Carnahan, S. ``Generalized moonshine IV: monstrous Lie algebras." arXiv preprint arXiv:1208.6254 (2012).

\bibitem[C2]{Car} Carnahan, S. ``Building vertex algebras from parts." arXiv preprint arXiv:1408.5215 (2014).


\bibitem[CdM]{CdM}
Cohen, A.\ M., and de Man, R. ``Computational evidence for Deligne's conjecture regarding exceptional Lie groups'', C.\ R.\ Acad.\  Sci.\ Paris S\'{e}r.\ I Math.\ {\bf 322}  (1996): 427--432.


\bibitem[D]{D}
Deligne, P. ``La s\'{e}rie exceptionalle de groupes de Lie'', C.\ R.\ Acad.\ Sci.\ Paris S\'{e}r.\ I Math.\ {\bf 322} (1996): 321--326.

\bibitem[DL]{DL} Dong, C., and Lepowsky, J. ``Generalized vertex algebras and relative vertex operators''. Springer, 1993.

\bibitem[DLM]{DLM} Dong, C., Li, H., and Mason, G. ``Modular-Invariance of Trace Functions in Orbifold Theory and Generalized Moonshine." Comm.\ Math.\ Phys.\ 214.1 (2000): 1--56.

\bibitem[E]{E} Ekeren, J.\ V. ``Modular invariance for twisted modules over a vertex operator superalgebra." Comm.\ Math.\ Phys.\ 322.2 (2013): 333--371.

\bibitem[EM]{EM} Eilenberg, S., and MacLane, S. ``On the groups $H(\Pi, n)$, II: Methods of computation." Annals of Mathematics (1954): 49--139.

\bibitem[FBZ]{FBZ} Frenkel, E., and Ben-Zvi, D. ``Vertex algebras and algebraic curves''. Vol. 88. Providence, RI: American mathematical society, 2001.

\bibitem[FF]{FF} Feigin, B., and Frenkel, E. ``Quantization of the Drinfel'd-Sokolov reduction''. Phys.
Lett. B, 246(1-2) (1990):75-–81.

\bibitem[FJM]{FJM} Feigin, B.\ L., Jimbo, M., and Miwa, T. ``Vertex operator algebra arising from the minimal series $M (3, p)$ and monomial basis''. Birkhäuser Boston, (2002): 179--204.

\bibitem[FKRW]{FKRW} Frenkel, E., Kac, V.\ G., Radul, A., and Wang, W. ``$\mathcal{W}_{1+\infty}$ and $\mathcal{W}_N$ with central charge $N$" Comm.\ Math.\ Phys.\ 170.2 (1995): 337--357.

\bibitem[GK1]{GK1} Gorelik, M., and Kac, V.\ G. ``Characters of highest 
weight modules over affine Lie algebras are meromorphic functions."
 Int.\ Math.\ Res.\ Not.\ 2007 (2007): rnm079. 25 pp.

\bibitem[GK2]{GK2} Gorelik, M., and Kac, V.\ G. ``On simplicity of 
vacuum modules." Adv.\ Math.\ 211.2 (2007): 621--677.

\bibitem[JM]{JM} Jacob, P., and Mathieu, P. ``A quasi-particle description of the $M (3, p)$ models." Nucl.\ Phys.\ B 733.3 (2006): 205--232.

\bibitem[Kac1]{Kac} Kac, V.\ G. ``Infinite-dimensional Lie algebras''. Vol. 44. Cambridge university press, 1994.

\bibitem[Kac2]{Kac2} Kac, V.\ G. ``Vertex algebras for beginners'' Vol. 10. Providence: American Mathematical Society, 1998.

\bibitem[KRW]{KRW} Kac, V.\ G., Roan, Shi-Shyr, and Wakimoto, M. ``Quantum reduction for affine superalgebras''.
Comm.\ Math.\ Phys., 241(2-3) (2003): 307–-342.

\bibitem[KW1]{KW0} Kac, V.\ G., and Wakimoto, M. ``Modular invariant representations of infinite-dimensional Lie algebras and superalgebras'' Proc.\ Natl.\ Acad.\ Sci.\ 85 (1988): 4956--4960. 

\bibitem[KW2]{KW1} Kac, V.\ G., and Wakimoto, M. ``Quantum reduction and representation theory of superconformal algebras." Adv.\ Math.\ 185.2 (2004): 400--458.

\bibitem[KWc]{KWc} Kac, V.\ G., and Wakimoto, M. ``Corrigendum to ``Quantum reduction and representation theory of superconformal algebras'' :[Adv. Math. 185 (2004) 400-–458]." Adv.\ Math.\ 193.2 (2005): 453--455.

\bibitem[KW3]{KW2} Kac, V.\ G., and Wakimoto M. ``On rationality of W-algebras." Transformation Groups 13.3-4 (2008): 671--713.


\bibitem[KNS]{KNS}
Kaneko, M., Nagatomo, K., and Sakai, Y. ``Modular forms and second order ordinary differential equations: applications to vertex operator algebras''. Lett.\ Math.\ Phys.\ {\bf 103} (2013): 439--453.


\bibitem[KZ]{Kan3}
Kaneko, M., and Zagier, D. ``Supersingular $j$-invariants, hypergeometric series, and Atkin's orthogonal polynomials'' AMS/IP Stud.\ Adv.\ Math.\ {\bf 7}  (1998): 97--126.


\bibitem[Kaw1]{Kaw1} Kawasetsu, K. ``The intermediate vertex subalgebras of the lattice vertex operator algebras." Lett.\ Math.\ Phys.\ 104.2 (2014): 157--178.

\bibitem[Kaw2]{Kaw2} Kawasetsu, K. ``The Free Generalized Vertex Algebras and Generalized Principal Subspaces." arXiv preprint arXiv:1502.05276 (2015).


\bibitem[La]{Lam} Lam, C.\ H. ``Induced modules for orbifold vertex operator algebras." J.\ Math.\ Soc.\ Japan 53.3 (2001): 541--557.

\bibitem[LM]{LaM2}
Landsberg, J.\ M., and Manivel, L. ``Triality, exceptional Lie algebras, and Deligne dimension formulas'', Adv.\ Math.\ {\bf 171} (2002): 59--85.

\bibitem[LM2]{LaM3}
Landsberg, J.\ M., and Manivel, L. ``The sextonions and $E_{7\frac{1}{2}}$." Adv.\ Math.\ 201.1 (2006): 143--179.

\bibitem[Li]{Li} Li, H. ``Local systems of vertex operators, vertex superalgebras and modules." J.\ Pure and Appl.\ Alg.\ 109.2 (1996): 143--195.

\bibitem[LL]{LL} Lepowsky, J., and Li, H. ``Introduction to vertex operator algebras and their representations. Vol. 227. Springer Science \& Business Media, 2004.

\bibitem[MMS]{MMS}
Mathur, S., Mukhi, S., and Sen, A.
``On the classification of rational conformal field theories''
Phys.\ lett.\ B, Vol.\ {\bf 213}, Issue.\ 3, (1988): 303--308.

\bibitem[MR]{MR} Mathieu, P., and Ridout, D. ``The extended algebra of the minimal models." Nucl.\ Phys.\ B 776.3 (2007): 365--404.

\bibitem[M]{M} Matsuo, A, ``Norton's Trace Formulae for the Griess Algebra of a Vertex Operator Algebra with Larger Symmetry." Comm.\ Math.\ Phys.\ 224.3 (2001): 565--591.

\bibitem[MV]{MV} Mkrtchyan, R.\ L., and Veselov, A.\ P.  ``Universality in Chern-Simons theory." J.\ High Ene.\ Phys.\ 2012.8 (2012): 1--12.

\bibitem[T]{T} Tuite, M.\ P. ``Exceptional vertex operator algebras and the Virasoro algebra." Contemp.\ Math.\ 497 (2009): 213--225.

\bibitem[V]{V} Vogel, P. ``The universal Lie algebra." preprint (1999).

\bibitem[W]{W} Westbury, B.\ W. ``Sextonions and the magic square." J.\  London Math.\ Soc.\ 73.02 (2006): 455--474.

\bibitem[Y]{Yam} Yamauchi, H. ``Module categories of simple current extensions of vertex operator algebras." J.\ Pure and Appl.\ Alg.\ 189.1 (2004): 315--328.

\bibitem[Zam]{Zam} Zamolodchikov, A.\ B. ``Infinite additional symmetries in two-dimensional conformal quantum field theory." Theor.\ Math.\ Phys.\ 65.3 (1985): 1205--1213.

\bibitem[Zhu]{Zhu} 
Zhu, Y. ``Modular invariance of characters of vertex operator algebras'' J.\ AMS 237--302: (1996).


\end{thebibliography}
\end{document}